\documentclass[12pt, a4paper]{amsart}
\usepackage{amscd, amssymb}

\usepackage{tensor}
\usepackage{wrapfig}
\usepackage{tikz}

\allowdisplaybreaks

\def\Aut{\operatorname{Aut}}

\def\sp{\operatorname{span}}
\def\clsp{\overline{\operatorname{span}}}

\def\id{\operatorname{id}}

\def\max{\operatorname{max}}

\def\PI{\operatorname{PIso}}

\def\C{\mathbb{C}}

\def\R{\mathbb{R}}
\def\N{\mathbb{N}}
\def\Z{\mathbb{Z}}
\def\T{\mathbb{T}}

\def\AA{\mathcal{A}}

\def\FF{\mathcal{F}}

\def\KK{\mathcal{K}}
\def\LL{\mathcal{L}}

\def\NN{\mathcal{N}}
\def\OO{\mathcal{O}}

\def\TT{\mathcal{T}}

\newcommand{\inv}{^{-1}}

\newtheorem{thm}{Theorem}[section]
\newtheorem{cor}[thm]{Corollary}

\newtheorem{lemma}[thm]{Lemma}
\newtheorem{prop}[thm]{Proposition}

\theoremstyle{definition}
\newtheorem{definition}[thm]{Definition}

\theoremstyle{remark}
\newtheorem{remark}[thm]{Remark}

\newtheorem{example}[thm]{Example}

\numberwithin{equation}{section}

\newcommand{\thmref}[1]{Theorem~\ref{#1}}

\newcommand{\proref}[1]{Proposition~\ref{#1}}
\newcommand{\lemref}[1]{Lemma~\ref{#1}}

\tikzstyle{vertex}=[circle]
\tikzstyle{goto}=[->,shorten >=1pt,>=stealth,semithick]

\begin{document}

\title[Self-similar actions of groupoids on graphs]{Equilibrium states on operator algebras\\ associated to\\ self-similar actions of groupoids on graphs}
\author{Marcelo Laca}
\address{Marcelo Laca, Department of Mathematics and Statistics\\
University of Victoria\\
Victoria, BC V8W 3P4\\
Canada}
\email{laca@math.uvic.ca}
\author[Iain Raeburn]{Iain Raeburn}
\address{Iain Raeburn, Department of Mathematics and Statistics, University of Otago, PO Box 56, Dunedin 9054, New Zealand}
\email{iraeburn@maths.otago.ac.nz }
\author[Jacqui Ramagge]{Jacqui Ramagge}
\address{Jacqui Ramagge, School of Mathematics and Statistics, University of Sydney,
NSW 2006, Australia.}
\email{Jacqui.Ramagge@sydney.edu.au
}
\author[Michael F. Whittaker]{Michael F. Whittaker}
\address{Michael F. Whittaker, School of Mathematics and Statistics, University of Glasgow, University Gardens, Glasgow Q12 8QW, Scotland}
\email{Mike.Whittaker@glasgow.ac.uk}

\thanks{This research was supported by the Natural Sciences and Engineering Research Council of Canada, the Marsden Fund of the Royal Society of New Zealand, and the Australian Research Council.}

\begin{abstract}
We consider  self-similar actions of groupoids  on the path spaces of finite directed graphs, and construct examples of such self-similar actions using a suitable notion of graph automaton. Self-similar groupoid actions have a Cuntz-Pimsner algebra and a Toeplitz algebra, both of which carry natural dynamics lifted from the gauge actions. We study the equilibrium states (the KMS states) on the resulting dynamical systems. Above a critical inverse temperature, the KMS states on the Toeplitz algebra are parametrised by the traces on the full $C^*$-algebra of the groupoid, and we describe a program for finding such traces. The critical inverse temperature is the logarithm of the spectral radius of the incidence matrix of the graph, and at the critical temperature the KMS states on the Toeplitz algebra factor through states of the Cuntz-Pimsner algebra. Under a verifiable hypothesis on the self-similar action, there is a unique KMS state on the Cuntz-Pimsner algebra. We discuss an explicit method of computing the values of this KMS state, and illustrate with examples. 
\end{abstract}

\maketitle

\section{Introduction}\label{sec:intro}

A self-similar group $(G,X)$ consists of a finite set $X$ and a faithful action of a group $G$ on the set $X^*$ of finite words in $X$, such that:
for each $g \in G$ and $x \in X$, there exists $h \in G$ satisfying
\begin{equation*}\label{self-similarity_intro}
g\cdot(xw)=(g\cdot x)(h \cdot w) \text{ for all } w\in X^*. 
\end{equation*}
Self-similar groups are often defined recursively using data presented in an automaton (see, for example, \cite[Chapter 1]{nek_book} or \S\ref{sec:preamble} below). To each self-similar group $(G,X)$, Nekrashevych associated a $C^*$-algebra $\OO(G,X)$, which is by definition the Cuntz-Pimsner algebra of a Hilbert bimodule over the reduced group algebra $C^*_r(G)$ \cite{nek_jot, nekra}.

We recently studied the Toeplitz algebra $\TT(G,X)$ of this Hilbert bimodule \cite{lrrw}. Both $\TT(G,X)$ and $\OO(G,X)$ carry natural gauge actions of the unit circle, and composing with the exponential map gives actions $\alpha$ of the real line. In \cite{lrrw}, we classified the equilibrium states (the KMS states) of the dynamical systems $(\TT(G,X),\R,\alpha)$ and $(\OO(G,X),\R,\alpha)$. We found a simplex of KMS states on $\TT(G,X)$ at all inverse temperatures larger than a critical value $\ln |X|$, and showed, under a mild hypothesis on $(G,X)$, that there is a single KMS state on $\OO(G,X)$ whose inverse temperature is $\ln |X|$.

Here we consider a new kind of self-similarity involving an action of a groupoid $G$ on the path space $E^*$ of a finite directed graph $E$, which we view as a forest of trees $\{vE^*:v\in E^0\}$. An equation of the form
\[
g\cdot(e\mu)=(g\cdot e)(h\cdot \mu)
\]
for paths $e\mu$ defines an isomorphism $h$ of the subtree $s(e)E^*$ onto $s(g\cdot e)E^*$. We call $h$ a \emph{partial isomorphism} of $E^*$. The partial isomorphisms of $E^*$ form a groupoid $\PI(E^*)$ with unit space $E^0$. A self-similar action of a groupoid $G$ with unit space $E^0$ is then a groupoid homomorphism of $G$ into $\PI(E^*)$. 

Our results are motivated by a construction of Exel and Pardo \cite{EP2}, who studied a family of self-similar actions of groups on path spaces. Their main motivation was to provide a unified theory that accommodates both Nekrashevych's Cuntz-Pimsner algebras and a family of ``Katsura algebras'' \cite{k2} that includes all Kirchberg algebras. We seek a common setting for the analyses of KMS states on self-similar groups in \cite{lrrw} and on the Toeplitz-Cuntz-Krieger algebras of graphs \cite{el, KW, hlrs, hlrs2}.

To each of our self-similar actions $(G,E)$ we associate a Toeplitz algebra $\TT(G,E)$ and a Cuntz-Pimsner algebra $\OO(G,E)$, and study the dynamics arising from the gauge actions of the circle. Above a critical inverse temperature, the KMS states on $\TT(G,E)$ are parametrised by traces on the coefficient algebra, which is the (full) $C^*$-algebra of the groupoid $G$. As similar analyses for graph algebras have consistently shown \cite{EFW, KW, hlrs}, the critical inverse temperature is $\ln \rho(B)$, where $B$ is the vertex matrix of the underlying graph $E$ and $\rho(B)$ is the spectral radius of $B$.  At the critical inverse temperature, and under mild hypotheses, the KMS$_{\ln \rho(B)}$ state  is unique and factors through the Cuntz-Pimsner algebra. Thus our results give those of \cite{lrrw} a distinct graph-theoretic slant. On the other hand, we also show how to compute the values of the KMS$_{\ln \rho(B)}$ state  by counting paths in Moore diagrams, as we did for self-similar groups in \cite[\S8]{lrrw}.

\subsection*{Outline}

After quickly reviewing the process by which automata give rise to self-similar group actions, we describe our self-similar actions in \S\ref{sec:SSGpoids}. We follow closely the analogy between the space $X^*$ of words, which is often viewed as a rooted tree, and the path space $E^*$, which is a forest of rooted trees with the vertices $v\in E^0$ as roots. Our groupoid acts by \emph{partial isomorphisms} of this forest, which are isomorphisms of one rooted tree onto another. These  partial isomorphisms form a groupoid $\PI(E^*)$ with unit space $E^0$ (Proposition~\ref{prop:alltreeisos}). An action of a groupoid with unit space $E^0$ is then simply a homomorphism of $G$ into $\PI(E^*)$, and our definition of self-similarity is formally just the usual one. After a brief discussion of the basic properties, we introduce a notion of $E$-automaton, and show how it gives rise a self-similar groupoid action on $E^*$.

Our treatment and notation are based on the established formalism for self-similar groups, and hence look quite different from that of Exel and Pardo. So we show in an appendix that their data for a group $K$ gives a self-similar action of the groupoid $K\times E^0$. But not all of our self-similar groupoid actions arise this way (see Remark~\ref{EP remark}).

In \S\ref{sec:Toep}, we construct the Toeplitz algebra $\TT(G,E)$ of a self-similar groupoid action $(G,E)$. Our constructions of KMS states follow the strategy developed in \cite{lr,lrr,lrrw}, and are intrinsically representation-theoretic in nature. To construct representations of our algebras, we use presentations of our algebras. Thus we take as our coefficient algebra the full groupoid algebra $C^*(G)$, just as we used the full group algebra in \cite{lrrw} because it is universal for unitary representations of the group. We discuss the universal property of $C^*(G)$ in Proposition~\ref{univC*G}. The presentation of the Toeplitz algebra is given in Proposition~\ref{Toeplitz_repn}: it is generated by a Toeplitz-Cuntz-Krieger family for the graph $E$ and a representation of the groupoid. In Proposition~\ref{Toeplitz_spanning} we identify a convenient spanning family consisting of elements which are analytic for the natural dynamics.

At this point we are ready to begin our analysis of the KMS states. In Proposition~\ref{KMS_algebraic}, we describe how to recognise KMS states in terms of their values on the spanning elements. In Theorem~\ref{Thm:KMS_beta_tau}, we describe the KMS$_\beta$ states on the Toeplitz algebra for $\beta$ above the critical inverse temperature $\ln\rho(B)$, finding that they are parametrised by the traces on the coefficient algebra $C^*(G)$. Thus in \S\ref{sec:traces} we consider the problem of constructing traces on the groupoid algebra $C^*(G)$. Since our groupoids have finite unit space $E^0$, the orbit space for the canonical action of $G$ on $E^0$ is finite, and $C^*(G)$ is the direct sum of $C^*$-algebras of transitive groupoids (see Lemma~\ref{directsum}). For a transitive groupoid we can realise $C^*(G)$ as matrices over the $C^*$-algebra $C^*(G_x)$ of an isotropy group, and then traces on $C^*(G_x)$ give traces on $C^*(G)$ (see Corollary~\ref{tracesonCG}). So there are always at least two traces on each $C^*(G)$. We check that in our special cases (that is, for traditional self-similar groups and for graph algebras), the resulting KMS states are indeed the ones previously found in~\cite{lrrw} and~\cite{hlrs}.

In \S\ref{sec:crit}, we prove that $\TT(G,E)$ always has a KMS$_{\ln\rho(B)}$ state, and that this state factors through $\OO(G,E)$. Under an easily verified hypothesis on $(G,E)$, this is the only KMS$_{\ln\rho(B)}$ state on $\TT(G,E)$ (Theorem~\ref{KMSatcritical}). The proof of uniqueness in particular is analytically quite delicate. We close with a section on examples. Since the formulas for the KMS states in Theorem~\ref{KMSatcritical} are rather intricate, it is comforting that in concrete examples we can compute particular values of these states, getting some strange and rather wild numbers (see Proposition~\ref{wildnos}).

\section{Preamble: Self-similar group actions and automata}\label{sec:preamble}

We begin by reviewing the process by which automata are used to construct self-similar group actions. This process is standard (see \cite[Chapter~1]{nek_book}, for example), but having the details clear will help us later.

Suppose that $X$ is a finite set, and write $X^k$ for the set of $k$-tuples in $X$, 
with~$X^0=\{\ast\}$, and define $X^*:=\bigsqcup_{k\geq 0} X^k$. 
Consider the (undirected) rooted tree $T=T_X$ with vertex set $T^0=X^*$ 
and edge set $T^1=\left\{ \{\mu, \mu x \}\colon \mu\in X^* \text{ and }x\in X\right\}$.
Note that $X^k$ is here identified with the vertices in $T$ at depth $k$ from the root $*$.
From a traditional graph-theoretic perspective, 
an \emph{automorphism} $\alpha$ of $T$ consists of 
a family of bijections $\alpha_k\colon  X^k \to X^k$ for $k\geq 0$ such that for all $\mu,\nu\in X^*$ 
\begin{equation}
\label{edge condition auto}
\{\alpha_k(\mu), \alpha_{k+1}(\nu)\} \in T^1 \Leftrightarrow
\{\mu, \nu\} \in T^1.
\end{equation}
It is sometimes convenient to use an alternative, but equivalent, definition of an automorphism of $T$ that is better suited to properties of self-similar actions. Notice that if $\beta=\{\beta_k\}$ is an automorphism, then each $\{\beta_k(\mu),\beta_{k+1}(\mu x)\}$ is an edge in $T$, and hence $\beta_{k+1}(\mu x)\in \beta_k(\mu)X$. So an automorphism satisfies \eqref{ATdef} below, and the property \eqref{ATdef} is ostensibly weaker. 

\begin{lemma}
\label{self-similar defn of auto}
Suppose $X$ is a finite set and $T$ is the rooted tree with 
vertex set $T^0=X^*$ and 
edge set $T^1=\left\{ \{\mu, \mu x \}\colon \mu\in X^* \text{ and }x\in X\right\}$.
Suppose $\alpha \colon T^0\to T^0$ is a bijection satisfying 
\begin{equation}
\label{ATdef}
\alpha(X^k)= X^k \text{ for all }k, \text{ and }
\alpha(\mu x)\in \alpha(\mu)X \text{ for all $\mu\in X^k$ and $x\in X$.}
\end{equation}
Define $\alpha_k:= \alpha|_{X^k}$. 
Then $\{\alpha_k\}$ is an automorphism $\alpha$ of $T$. The inverse is also an automorphism of $T$, and also satisfies \eqref{ATdef}.
\end{lemma}

\begin{proof}
Since $\alpha$ is a bijection and $X^k$ is a finite set, $\alpha(X^k)= X^k$ implies that $\alpha_k$ is a bijection.
We need to show that~\eqref{edge condition auto} is satisfied.
Suppose $\mu\in X^k$ and $\nu\in X^{k+1}$.
If  $\{\mu,\nu\}$ is an edge then $\nu=\mu x$ for some $x\in X$ and 
$\alpha(\nu)=\alpha_{k+1}(\mu x) \in \alpha_k(\mu) X$, 
so that $\{\alpha(\mu),\alpha(\nu)\}$ is an edge. 
Now suppose $\{\alpha(\mu),\alpha(\nu)\}$ is an edge, 
say $\alpha(\nu)=\alpha(\mu)y$ for $y\in X$. 
Since $\nu\in X^{k+1}$ for $k\geq 0$ 
we can write $\nu=\eta x$ for some $\eta\in X^k$ and $x\in X$.
Then $\alpha_{k+1}(\nu)\in\alpha(\eta)X$, $\alpha(\eta)=\alpha(\mu)$, and $\eta=\mu$.
So $\nu=\mu x$ and $\{\mu,\nu\}$ is an edge. Thus we have~\eqref{edge condition auto}, and $\{\alpha_k\}$ is an automorphism of~$T$.

For the assertion about $\alpha^{-1}$, note that \eqref{edge condition auto} implies that this inverse is a graph isomorphism in the usual sense, and hence satisfies the weaker condition \eqref{ATdef} by the comment before the lemma.
\end{proof}

Suppose again that $X$ is a finite set. 
An \emph{automaton over $X$} is a finite set $A$ together with 
a map $(a,x)\mapsto (a\cdot x, a|_x)$  from $A\times X$ to $X\times A$ such that, 
for each fixed $a\in A$, $x\mapsto a\cdot x$ is a bijection.
For $\mu\in X^k$ and $a\in A$ we define $a|_\mu$ inductively 
by $a|_{\mu x} = (a|_\mu)|_x$. 
For $k\geq 0$ we define $f_{a,k}\colon  X^k\to X^k$ by 
\[
f_{a,k}(\mu) = (a\cdot \mu_1)(a|_{\mu_1}\cdot\mu_2)\cdots(a|_{\mu_1\cdots\mu_{k-1}}\cdot\mu_k)
\]
for $\mu=\mu_1\cdots\mu_k\in X^k$.
Then we have
\[
f_{a,k+1}(\mu x) = 
(a\cdot \mu_1)(a|_{\mu_1}\cdot\mu_2)\cdots(a|_{\mu_1\cdots\mu_{k-1}}\cdot\mu_k)
(a|_\mu\cdot x)
= f_{a,k}(\mu)(a|_\mu\cdot x)\in f_{a,k}(\mu) X
\]
from which it follows that $f_a:=\{f_{a,k}\}$ is 
an automorphism of $T_X$ for $a\in A$.

The set of automorphisms of $T_X$ form a group $\Aut T_X$. We let $G_A$  be the subgroup of $\Aut T_X$ generated by $\{f_a\colon a\in A\}$.

Recall that a faithful action of a group $G$ on $X^*$ is \emph{self-similar} if for all $g\in G, x\in X$ there exist $y\in X, h\in G$ such that 
\begin{equation}
\label{selfsimilar action}
g\cdot(xw)=y(h\cdot w) \text{ for all } w\in X^*.
\end{equation}

\begin{prop}
\label{self similar action of G_A}
There is a faithful self-similar action of $G_A$ on $T_X$ such that 
$
f_a\cdot \mu := f_a(\mu)
$
for $a\in A$ and $\mu\in X^*$.
\end{prop}
\begin{proof}
We first show that each $f_a$ satisfies condition~\eqref{selfsimilar action}, then that each $f_a^{-1}$ does, and then that the composition of two automorphisms satisfying~\eqref{selfsimilar action} also satisfies~\eqref{selfsimilar action}. 

First consider $f_a$ for $a\in A$. For $x\in X$ and $\mu\in X^k$ 
\begin{align*}
f_a(x\mu) &=
f_{a,k+1}(x\mu) = 
(a\cdot x)(a|_x\cdot\mu_1)(a|_{x\mu_1}\cdot\mu_2)\cdots (a|_{x\mu_1\cdots\mu_{k-1}}\cdot\mu_k) \\
&= (a\cdot x)(a|_x\cdot\mu_1)((a|_x)|_{\mu_1}\cdot\mu_2)\cdots ((a|_x)|_{\mu_1\cdots\mu_{k-1}}\cdot \mu_k) \\
&= (a\cdot x)f_{a|_x}(\mu_1\cdots \mu_k)
\end{align*}
and hence we can take  $y=a\cdot x$ and $h=f_{a|_x}$ in~\eqref{selfsimilar action}.

Now consider the automorphism $f_a^{-1}$ for $a\in A$.
\begin{align*}
f_a^{-1}(x\mu) = y\nu &\Leftrightarrow x\mu = f_a(y\nu) \\
&\Leftrightarrow x\mu = (a\cdot y) f_{a|_y}(\nu) \\
&\Leftrightarrow x=a\cdot y \text{ and } \mu= f_{a|_y}(\nu) \\
&\Leftrightarrow y = f_a^{-1}(x)  \text{ and } \nu=f_{a|_y}^{-1}(\nu) \\
&\Leftrightarrow y = f_a^{-1}(x)  \text{ and } \nu=f_{a|_{f_a^{-1}(x)}}^{-1}(\nu). 
\end{align*}
So $f_a^{-1}(x\mu)= f_a^{-1}(x)f_{a|_{f_a^{-1}(x)}}^{-1}(\nu)$,
and we can take $y = f_a^{-1}(x)$ and $h=f_{a|_{f_a^{-1}(x)}}^{-1}$ in~\eqref{selfsimilar action}.

Finally, suppose $g_i\in G_A$ and for each $x\in X$ we have 
$y_{i,x}\in X$ and $h_{i,x}\in G_A$ such that 
\[
g_i(x\mu)= y_{i,x} h_{i,x}(\mu) \text{ for all } \mu.
\] 
Then for each $\mu\in X^*$
\begin{align*}
g_1g_2(x\mu) & = g_1(y_{2,x} h_{2,x}(\mu)) \\
&= y_{1,y_{2,x}} h_{1,y_{2,x}}(h_{2,x}(\mu)) \\
& = y_{1,y_{2,x}} (h_{1,y_{2,x}}h_{2,x})(\mu). 
\end{align*}
Thus $y= y_{1,y_{2,x}}$ and $h=h_{1,y_{2,x}}h_{2,x}$ satisfy the property in~\eqref{selfsimilar action} for~$g_1g_2$.

Finally, since every element of $G_A=\langle f_a\colon a\in A\rangle$ 
is a product of elements of the form $f_a$ and $f_b^{-1}$ for $a,b\in A$, 
we can  for every $g\in G_A$ 
construct elements $y\in X$ and $h\in G_A$ with the required properties.
\end{proof}

\begin{example}
Suppose $X=\{x,y\}$, and an automaton has alphabet $A=\{a,b,e\}$.  The actions of $a, b\in A$ on $X=\{x,y\}$ satisfy $a\cdot x =y$, $a|_x=b$, $a\cdot y=x$, $a|_y=e$, 
$b\cdot x =x$, $b|_x=a$, $b\cdot y=y$, and $b|_y=e$; $e\in A$ acts and restricts trivially, so $e\cdot x=x$ and $e|_x=e$, and similarly on $y$.
This is typically presented as a recursive definition on~$X^*$ by  
\begin{alignat}{2}
\label{basilica_action}
a \cdot(xw)&= y (b\cdot w) & \qquad   a \cdot (yw)&= x w\quad\text{ and}  \\
b \cdot(xw)&= x (a \cdot w)& \qquad   b \cdot (yw)&= y w\quad\text{for $w\in X^*$.}\notag
\end{alignat}
\end{example}

We illustrate an automaton by drawing a \emph{Moore diagram}.
In a Moore diagram the arrow 
\begin{center}
\begin{tikzpicture}
\node[vertex] (vertexa) at (-2,0)   {$g$};	
\node[vertex] (vertexb) at (0,0)  {$h$}
	edge [<-,>=latex,out=180,in=0,thick] node[auto,swap,pos=0.5]{$\scriptstyle(x,y)$}(vertexa);
\end{tikzpicture}
\end{center}
means that $g\cdot x=y$ and  $g|_x=h$. 
Thus, for example, the automaton representing the generators $A$ for the basilica group 
 is illustrated by the Moore diagram in Figure~\ref{fig:gens Basilica}. The identity $e$ is often not included in the alphabet $A$, 
but is implicit in the recursive definition~\cite[(2.6)]{lrrw}.
\begin{figure}
\begin{tikzpicture}
\node at (0,0) {$\scriptstyle e$};
\node[vertex] (vertexe) at (0,0)   {$\,$}
	edge [<-,>=latex,out=10,in=60,loop,thick] node[right]{$\scriptstyle(x,x)$} (vertexe)
	edge [<-,>=latex,out=350,in=300,loop,thick] node[right]{$\scriptstyle(y,y)$} (vertexe);
\node at (-2,0.75) {$\scriptstyle b$};
\node[vertex] (vertexb) at (-2,0.75)   {$\,$}
	edge [->,>=latex,out=360,in=140,thick] node[auto,xshift=-0.2cm,pos=0.5]{$\scriptstyle(y,y)$} (vertexe);
\node at (-2,-0.75) {$\scriptstyle a$};
\node[vertex] (vertexa) at (-2,-0.75)  {$\,$}
	edge [->,>=latex,out=360,in=220,thick] node[auto,swap,yshift=0.1cm,pos=0.5]{$\scriptstyle(y,x)$} (vertexe)
	edge [->,>=latex,out=140,in=220,thick] node[auto,xshift=0.1cm,pos=0.5]{$\scriptstyle(x,y)$}(vertexb)
	edge [<-,>=latex,out=380,in=340,thick] node[auto,swap,xshift=-0.1cm,pos=0.5]{$\scriptstyle(x,x)$}(vertexb);
\end{tikzpicture}
\caption{The Moore diagram for the generators of the basilica group.}
\label{fig:gens Basilica}
\end{figure}
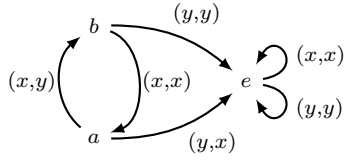

\section{Self-similar groupoids and automata over graphs}\label{sec:SSGpoids}

We now generalise the construction of \S\ref{sec:preamble}, replacing the set $X$ by the set of edges $E^1$ in a finite directed graph $E$. 

Suppose $E=(E^0,E^1,r,s)$ is a directed graph with vertex set $E^0$, edge set $E^1$, and range and source maps $r,s\colon E^1\to E^0$.  
Write 
\[
E^k =\{ \mu=\mu_1\cdots\mu_k \colon \mu_i\in E^1, s(\mu_i)=r(\mu_{i+1})\}
\]
for the set of paths of length $k$ in $E$, keep $E^0$ for the set of vertices, and define $E^*=\bigsqcup_{k\geq 0} E^k$. 
We recover the classical situation of the previous section by taking $E$ to be the graph $(\{\ast\},X,r,s)$ in which  $r(x)=r(y)=s(x)=s(y)=\ast$ for all $x,y\in X=E^1$ and $E^*=X^*$.

The analogue of the tree $T_X$ in our situation is the (undirected) graph $T_E$ with vertex set $T^0=E^*$ and edge set  
\[
T^1=\left\{ \{\mu, \mu e \}\colon \mu\in E^*,e\in E^1, \text{ and } s(\mu)=r(e)\right\}.
\]
The subgraph 
$
vE^*=\{\mu\in E^*\colon r(\mu)=v\} 
$
is a rooted tree with root $v\in E^0$, and $T_E=\bigsqcup_{v\in E^0} vE^*$ is a disjoint union of trees, or \emph{forest}. 
An example of a graph $E$ and the corresponding forest $T_E$ 
is given in Figure~\ref{forest of trees}.

\begin{figure}
\begin{tikzpicture}[scale=1.0]
\node[vertex] (p) at (-8.0,0.5) {$E$};
\begin{scope}[xshift=-6.5cm,yshift=-1.25cm]
\node at (0,0) {$w$};
\node[vertex] (vertexe) at (0,0)   {$\quad$};
\node at (-3,0) {$v$};
\node[vertex] (vertex-a) at (-3,0)   {$\quad$}
	edge [->,>=latex,out=35,in=145] node[below,swap,pos=0.5]{$3$} (vertexe)
	edge [->,>=latex,out=50,in=130] node[above,swap,pos=0.5]{$4$} (vertexe)
	edge [->,>=latex,out=210,in=150,loop] node[left,pos=0.5]{$1$} (vertexe)
	edge [<-,>=latex,out=310,in=230] node[below,swap,pos=0.5]{$2$} (vertexe);
\end{scope}

\node[vertex] (p) at (0,0.5) {$T_E$};
\begin{scope}[xshift=-2.3cm,yshift=0cm]
   \node[vertex] (v) at (0,0) {$v$};
    \node[vertex] (x) at (-1,-1) {$1$};
    \node[vertex] (y) at (1,-1) {$2$};
    \node[vertex] (xx) at (-1.5,-2.2) {$11$};
    \node[vertex] (xy) at (-0.5,-2.2) {$12$};
    \node[vertex] (yx) at (0.5,-2.2) {$23$};
    \node[vertex] (yy) at (1.5,-2.2) {$24$};
    \node[vertex] (31) at (-1.5-0.25,-3.2) {};
    \node[vertex] (32) at (-1.5+0.25,-3.2) {};
    \node[vertex] (33) at (-0.5-0.25,-3.2) {};
    \node[vertex] (34) at (-0.5+0.25,-3.2) {};
    \node[vertex] (35) at (0.5-0.25,-3.2) {};
    \node[vertex] (36) at (0.5+0.25,-3.2) {};
    \node[vertex] (37) at (1.5-0.25,-3.2) {};
    \node[vertex] (38) at (1.5+0.25,-3.2) {};
    \draw (x)--(v);
    \draw (y)--(v);
    \draw (xx)--(x);
    \draw (xy)--(x);
    \draw (yx)--(y);
    \draw (yy)--(y);
    \draw (31)--(xx);
    \draw (32)--(xx);
    \draw (33)--(xy);
    \draw (34)--(xy);
    \draw (35)--(yx);
    \draw (36)--(yx);
    \draw (37)--(yy);
    \draw (38)--(yy);
\end{scope}

\begin{scope}[xshift=2.3cm,yshift=0cm]
    \node[vertex] (v) at (0,0) {$w$};
    \node[vertex] (x) at (-1,-1) {$3$};
    \node[vertex] (y) at (1,-1) {$4$};
    \node[vertex] (xx) at (-1.5,-2.2) {$31$};
    \node[vertex] (xy) at (-0.5,-2.2) {$32$};
    \node[vertex] (yx) at (0.5,-2.2) {$41$};
    \node[vertex] (yy) at (1.5,-2.2) {$42$};
    \node[vertex] (31) at (-1.5-0.25,-3.2) {};
    \node[vertex] (32) at (-1.5+0.25,-3.2) {};
    \node[vertex] (33) at (-0.5-0.25,-3.2) {};
    \node[vertex] (34) at (-0.5+0.25,-3.2) {};
    \node[vertex] (35) at (0.5-0.25,-3.2) {};
    \node[vertex] (36) at (0.5+0.25,-3.2) {};
    \node[vertex] (37) at (1.5-0.25,-3.2) {};
    \node[vertex] (38) at (1.5+0.25,-3.2) {};
    \draw (x)--(v);
    \draw (y)--(v);
    \draw (xx)--(x);
    \draw (xy)--(x);
    \draw (yx)--(y);
    \draw (yy)--(y);
    \draw (31)--(xx);
    \draw (32)--(xx);
    \draw (33)--(xy);
    \draw (34)--(xy);
    \draw (35)--(yx);
    \draw (36)--(yx);
    \draw (37)--(yy);
    \draw (38)--(yy);
\end{scope}
\end{tikzpicture}
\caption{The forest of trees $T_E$ for the graph $E$}
\label{forest of trees}
\end{figure}
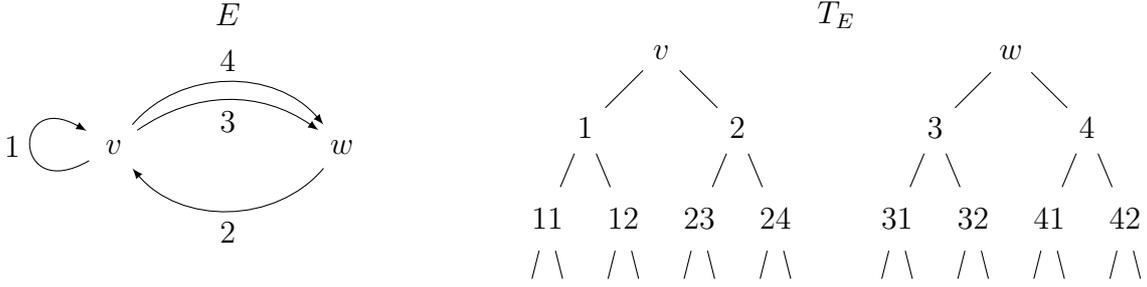

We will be using range and source maps in the context of both graphs and, later,  automata. If there is any room for confusion or if we simply want to be clear for later reference, we will use the notation $r_E$ and $s_E$ for the range and source maps associated to~$E$.

Difficulties in defining restrictions consistently in this new context mean that we need to deal with partial isomorphisms rather than automorphisms. Partial isomorphisms of $T_E$ are only defined on a subtree of $T_E$ but are still isomorphisms of their domain onto their range.

\begin{definition}\label{defn: partial iso}
Suppose $E=(E^0,E^1,r,s)$ is a directed graph. A \emph{partial isomorphism} of $T_E$ consists of two vertices $v,w \in E^0$ and a bijection $g:vE^* \to wE^*$ such that
\begin{equation}
\label{defnpi}
g|_{vE^k}\text{ is a bijection onto }wE^k \text{ for }k \in \N, \text{ and } 
g(\mu e) \in g(\mu)E^1 \text{ for all } \mu e \in vE^*. 
\end{equation}
For $v\in E^0$, we write $\id_{v}:vE^* \to vE^*$ for the partial isomorphism given by 
$\id_{v}(\mu)=\mu$ for all $\mu \in vE^*$. 
We denote the set of all partial isomorphisms of $T_E$ by $\PI(E^*)$.
We define domain and codomain maps $d,c\colon \PI(E^*) \to E^0$ so that $g:d(g)E^* \to c(g)E^*$; 
thus in~\eqref{defnpi} we have $d(g):=v$ and $c(g):=w$.
\end{definition}

Because we work with partial isomorphisms instead of automorphisms, 
we work with groupoids rather than groups. 
A groupoid differs from a group in two main ways: 
the product in a groupoid is only partially defined, and a groupoid typically has more than one unit.

A \emph{groupoid} is a small category with inverses. 
Thus a groupoid consists of a set $G^0$ of objects, a set $G$ of morphisms, 
two functions $c,d:G\to G^0$, and 
a partially defined product $(g,h)\mapsto gh$ from $G^2:=\{(g,h): d(g)=c(h)\}$ to $G$ 
such that $(G,G^0,c,d)$ is a category and such that each $g\in G$ has an inverse $g^{-1}$. 
(We then typically write $G$ to denote the groupoid, and call $G^0$ the unit space of the groupoid. 
If $|G^0|=1$, then $G$ is simply a group in the usual sense.) 
There are many formulations of this definition; 
we have chosen this one because it emphasises the objects and morphisms as distinct sets, 
and neatly summarises the axioms. 
The next result describes an important example, and its proof illustrates the efficiency of the definition.

\begin{prop}\label{prop:alltreeisos}
Suppose that $E=(E^0,E^1,r,s)$ is a directed graph with associated forest $T_E$, and resume the notation of Definition~\ref{defn: partial iso}. Then $(\PI(E^*),E^0,c,d)$ is a groupoid in which: the product is given by composition of functions, the identity isomorphism at $v\in E^0$ is $\id_v:vE^*\to vE^*$, and the inverse of $g\in \PI(E^*)$ is the inverse of the function $g:d(g)E^*\to c(g)E^*$.
\end{prop}

\begin{proof}
For $g,h\in \PI(E^*)^2$, the composition $gh$ of $h:d(h)E^*\to c(h)E^*$ with $g:d(g)E^*=c(h)E^*\to c(g)E^*$ is 
a bijection of $d(h)E^*$ onto $c(g)E^*$ which maps each $d(h)E^k$ onto $c(g)E^k$. 
To see that $gh$ has the second property in \eqref{defnpi}, we take $\mu e\in d(h)E^*$. 
Since $h$ satisfies \eqref{defnpi}, there exists $f\in E^1$ such that $h(\mu e)=h(\mu)f$, and then
\[
(gh)(\mu e)=g(h(\mu e))=g(h(\mu)f)
\]
belongs to $g(h(\mu))E^1=(gh)(\mu)E^1$ because $g$ satisfies \eqref{defnpi}. 
So $gh$ is a partial isomorphism with $d(gh)=d(h)$ and $c(gh)=c(g)$. 
Associativity of the multiplication follows from the associativity of composition. 
Thus $(\PI(E^*),E^0,c,d)$ is a category.

The usual properties of composition also imply that $\id_{c(g)}g=g=g\id_{d(g)}$. 
Each $g:d(g)E^*\to c(g)E^*$ is a bijection, and hence has a set-theoretic inverse $g^{-1}:c(g)E^*\to d(g)E^*$. 
That this inverse is indeed a partial isomorphism follows from 
the last assertion in Lemma~\ref{self-similar defn of auto} 
(or rather from its extension to partial isomorphisms).
\end{proof}

Suppose that $E$ is a directed graph and $G$ is a groupoid with unit space $E^0$. 
An \emph{action} of $G$ on the path space $E^*$ is 
a (unit-preserving) groupoid homomorphism $\phi:G\to \PI(E^*)$; 
the action is \emph{faithful} if $\phi$ is one-to-one. 
If the homomorphism is fixed, we usually write $g\cdot \mu$ for $\phi_g(\mu)$. 
This applies in particular when $G$ arises as a subgroupoid of $\PI(E^*)$.

\begin{definition}
\label{faithful self-similar groupoid action}
Suppose $E=(E^0,E^1,r,s)$ is a directed graph and $G$ is a groupoid with unit space $E^0$ which acts faithfully on $T_E$. The action is \emph{self-similar} if 
for every $g\in G$ and $e\in d(g)  E^1$, there exists $h\in G$ such that
\begin{equation}
\label{selfsimilar groupoid defn}
g\cdot(e\mu)=(g\cdot e)(h\cdot \mu) \text{ for all } \mu\in s(e)E^*.
\end{equation}
Since the action is faithful, there is then exactly one such $h\in G$, and we write $g|_e:=h$.
\end{definition}

Definition~\eqref{faithful self-similar groupoid action} has some immediate and important consequences.

\begin{lemma}
\label{source and range observations}
Suppose $E=(E^0,E^1,r,s)$ is a directed graph and 
$G$ is a groupoid with unit space $E^0$ acting self-similarly on $T_E$. 
Then for $g,h\in G$ with $d(h)=c(g)$ and $e\in d(g) E^1$, we have
\begin{enumerate}
\item 
\label{source and range in SSG}
$d(g|_e)=s(e)$ and $c(g|_e)=s(g\cdot e)$,   
\item 
\label{equivariance}
$r(g\cdot e) = g\cdot r(e)$ and $s(g\cdot e) = g|_e\cdot s(e)$,
\item\label{resid}
if $g=\id_{r(e)}$, then $g|_e=\id_{s(e)}$, and
\item\label{multonedges}
$(hg)|_{e}=(h|_{g\cdot e})(g|_e)$.
\end{enumerate}
\end{lemma}

\begin{proof}
For every $\mu\in s(e)E^*$, $g\cdot(e\mu)=(g\cdot e)(g|_e\cdot \mu)$ is a path in $E$. 
This says, first, that $\mu$ is in the domain $d(g|_e)E^*$ of $g|_e$, so that $d(g|_e)=s(e)$, and, 
second, that $c(g|_e)=r(g|_e\cdot \mu)=s(g\cdot e)$. This gives \eqref{source and range in SSG}. 

For \eqref{equivariance}, we observe that $\mu\mapsto g\cdot\mu$ is an isomorphism of 
the tree $d(g)E^*$ onto $c(g)E^*$, and in particular we have $g\cdot d(g)=c(g)$. 
Thus $r(g\cdot e)=c(g)=g\cdot d(g)=g\cdot r(e)$. 
On the other hand, $g|_e$ is an isomorphism of $d(g|_e)E^*=s(e)E^*$ onto $c(g|_e)E^*=s(g\cdot e)E^*$. 
So $g|_e\cdot s(e)$ is $s(g\cdot e)$.

If $g=\id_{r(e)}$, then $g\cdot (e\mu)=e\mu=e(\id_{r(\mu)}\cdot\mu)=e(\id_{s(e)}\cdot\mu)$, 
and the uniqueness of restrictions gives \eqref{resid}. 
For \eqref{multonedges}, we take $\mu\in d(g)E^*$ and compute:
\begin{align*}
((hg)\cdot e)((hg)|_e\cdot\mu)&=(hg)\cdot(e\mu)=h\cdot((g\cdot e)(g|_e\cdot\mu))\\
&=(h\cdot(g\cdot e))(h|_{g\cdot e}\cdot(g|_e\cdot \mu))\\
&=((hg)\cdot e)((h|_{g\cdot e}g|_e)\cdot \mu).
\end{align*}
Thus we have $(hg)|_e\cdot\mu=(h|_{g\cdot e}g|_e)\cdot \mu$ for all $\mu\in d(g)E^*$, 
and \eqref{multonedges} follows because the action is faithful. 
\end{proof}

\begin{remark}\label{EP remark}
Lemma~\ref{source and range observations}\eqref{equivariance} implies that 
the source map may not be equivariant: $s(g\cdot e)\not=g\cdot s(e)$ in general. 
Indeed, $g\cdot s(e)$ will often not make sense: $g$ maps $d(g)E^*$ onto $c(g)E^*$, 
and $s(e)$ is not in $d(g)E^*$ unless $s(e)=d(g)$.
This non-equivariance of the source map is one of the main points of difference 
between our work and that of Exel--Pardo~\cite{EP2} and Bedos--Kaliszewski--Quigg~\cite{BKQ}; see Appendix \ref{App:Exel-Pardo Actions} for further details.
\end{remark}

Suppose $\mu=\mu_1\cdots\mu_m\in d(g)E^m$ and $\nu\in s(\mu)E^*$.
Then $\mu\nu\in d(g)E^*$ and
\begin{align*}
g\cdot(\mu\nu) &= (g\cdot \mu_1)(g|_{\mu_1}\cdot\mu_2)\cdots 
(\cdots (g|_{\mu_1})|_{\mu_2})\cdots)|_{\mu_{m-1}}\cdot \mu_m)
((\cdots (g|_{\mu_1})|_{\mu_2})\cdots)|_{\mu_m}\cdot\nu) \\
& = (g\cdot\mu)((\cdots (g|_{\mu_1})|_{\mu_2})\cdots)|_{\mu_m}\cdot\nu);
\end{align*}
thus with 
$
g|_\mu:= (\cdots (g|_{\mu_1})|_{\mu_2})\cdots)|_{\mu_m}
$
we can generalise~\eqref{selfsimilar groupoid defn} to
\begin{equation}
\label{selfsimilar groupoid action}
g\cdot(\mu\nu)=(g\cdot \mu)(g|_\mu\cdot \nu) \text{ for all } \mu\in d(g)E^*, \nu\in s(\mu)E^*.
\end{equation}
By iterating parts \eqref{resid} and \eqref{multonedges} of Lemma~\ref{source and range observations} we arrive at the following  more general versions:

\begin{prop}~\label{SSA:extension to paths}
\label{properties of faithful self-similar groupoid action}
Suppose $E=(E^0,E^1,r,s)$ is a directed graph and $G$ is a groupoid with unit space $E^0$ acting self-similarly on $T_E$. Then for all $g,h\in G$, $\mu\in d(g) E^*$, and $\nu\in s(\mu)E^*$ we have:
\begin{enumerate}
\item\label{path_rest_rest} $g|_{\mu\nu}=(g|_\mu)|_\nu$, 
\item $\id_{r(\mu)}|_\mu=\id_{s(\mu)}$,
\item\label{path_rest_nonhomo} $(hg)|_\mu= h|_{g\cdot\mu}g|_\mu$, and
\item\label{path_rest_inv} $g^{-1}|_\mu = (g|_{g^{-1}\cdot\mu})^{-1}$.
\end{enumerate}
\end{prop}

\begin{proof}
The first three of these are straightforward. For the last, observe that on the one hand, we have $gg^{-1}=\id_{r(\mu)}$, so $(gg^{-1})|_{\mu}=\id_{s(\mu)}$, and on the other we have $(gg^{-1})|_{\mu}=(g|_{g^{-1}\cdot\mu})(g^{-1}|_{\mu})$. Now multiplying $\id_{s(\mu)}=(g|_{g^{-1}\cdot\mu})(g^{-1}|_{\mu})$ on the left by $(g|_{g^{-1}\cdot \mu})^{-1}$ gives the result.
\end{proof}

We now generalise the process of constructing faithful self-similar actions from automata. 
The switch to path spaces requires that both the action and restriction maps interact with the graph structure. 
We begin by defining automata in this more general context.

Since the groupoid $\PI(E^*)$ has local identities $\id_v$ associated to each vertex $v$, we may need to add some of the vertices to the alphabet~$A$ of the automaton. 
If our aim were only to ensure that the set $A$ is closed under taking restrictions then 
we may only need to add some of the vertices --- see Example~\ref{s,r clarifying example}, 
where we would need to add $v$ to $A$ but not~$w$.
However, we want groupoids generated from automata associated to $E$ to be subgroupoids of $\PI(E^*)$. 
Since the standard definition of a subgroupoid involves having the same unit space as the ambient groupoid, we assume that $E^0\subset A$.

\begin{definition}
\label{defn: E-automaton}
An \emph{automaton over $E=(E^0,E^1,r_E,s_E)$} is a finite set $A$ containing $E^0$ 
together with functions $r_A,s_A\colon A\to E^0$ such that $r_A(v)=v=s_A(v)$ if $v\in E^0\subset A$, 
and a function 
\begin{equation}
A \tensor[_{s_A}]{\times}{_{r_E}} E^1 \ni (a,e) \mapsto (a\cdot e, a|_e)\in  E^1 \tensor[_{s_E}]{\times}{_{r_A}} A
\end{equation} 
such that:
\begin{itemize}
\item[(A1)] for every $a\in A$, $e\mapsto a\cdot e$ is a bijection of $s_A(a)E^1$ onto $r_A(a)E^1$;
\item[(A2)]\label{second property of automaton} 
$s_A(a|_e)=s_E(e)$ for all $(a,e)\in A \tensor[_{s_A}]{\times}{_{r_E}} E^1 $;
\item[(A3)] $r_E(e)\cdot e= e$ and $r_E(e)|_e=s_E(e)$ for all $e\in E^1$.
\end{itemize}
\end{definition}

Since $s_E(v)=r_E(v)=s_A(v)=r_A(v)=v$ for all $v\in E^0$, 
the range and source maps are consistent whenever they both make sense.
So there should be no confusion in using $s,r$ for both the graph and the automaton. 
However, sometimes it is convenient to distinguish between the two for clarity 
as in Definition~\ref{defn: E-automaton}, in which case
we  write $s_E,r_E$ and $s_A,r_A$ for the range and source maps in $E$ and $A$ respectively.

We can use the property~(A2)  to extend restriction to paths by defining
\[
a|_\mu = (\cdots ((a|_{\mu_1})|_{\mu_2})|_{\mu_3}\cdots )|_{\mu_k}.
\] 
The point is that $s_A(a|_{\mu_1})=s_E(\mu_1)=r_E(\mu_2)$, for example,
and hence $(a|_{\mu_1})|_{\mu_2}$ makes sense. 
Our next result says that we can extend the action of elements of the set $A$ to partial isomorphisms on $E^*$.

\begin{prop}\label{extenda}
Suppose that $E$ is a directed graph and $A$ is an automaton over $E$. We recursively define maps $f_{a,k}:s(a)E^k\to r(a)E^k$ for $a\in A$ and $k\in\N$ by $f_{a,1}(e)=a\cdot e$ and
\begin{equation}\label{deffa}
f_{a,k+1}(e\mu)=(a\cdot e)f_{a|_e,k}(\mu)\quad\text{for $e\mu\in s_A(a)E^{k+1}$.}
\end{equation}
Then for every $a\in A$, $f_a=\{f_{a,k}\}$ is a partial isomorphism of $s(a)E^*$ onto $r(a)E^*$ so that $d(f_a)=s(a)$ and $c(f_a)=r(a)$; for $a=v\in A\cap E^0$, we have $f_a=\id_{v}\colon vE^*\to vE^*$.
\end{prop}

\begin{proof}
We prove by induction on $k$ that the set $\{f_{a,k}:a\in A\}$ consists of bijections $f_{a,k}$ of $s(a)E^k$ onto $r(a)E^k$. The base case $k=1$ is (A1) in the definition of automaton. So we suppose the claim is true for $k\geq 1$. For $e\mu\in s(a)E^{k+1}$, then $s(a|_e)=s(e)=r(\mu)$, so the right-hand side of \eqref{deffa} makes sense and belongs to $r(a)E^{k+1}$. To see that $f_{a,k+1}$ is one-to-one, suppose that $f_{a,k+1}(e\mu)=f_{a,k+1}(f\nu)$. Then we have $a\cdot e=a\cdot f$, which implies $e=f$ by (A1), and then
\[
f_{a|_e,k}(\mu)=f_{a|_f,k}(\nu)=f_{a|_e,k}(\nu)\Longrightarrow \mu=\nu,
\]
and $e\mu=f\nu$. To see that $f_{a,k+1}$ is onto, take $f\nu\in r(a)E^{k+1}$. Then (A1) gives $e\in E^1$ such that $a\cdot e=f$, and surjectivity of $f_{a|_e,k}$ gives $\mu$ such that $f_{a|_e,k}(\mu)=\nu$. Then $f_{a,k+1}(e\mu)=f\nu$, and $f_{a,k+1}$ is onto.

We now have to prove that the maps $f_{a,k}$ satisfy the second condition in \eqref{defnpi}, which we again do simultaneously for all $a$ by induction on $k$. Let $a\in A$, $\mu\in s(a)E^k$, $e\in s(\mu)E^1$, and we want $f_{a,k+1}(\mu e)\in f_{a,k}(\mu)E^1$. For $k=1$, we have $|\mu|=1$ and $\mu=\mu_1\in E^1$. Thus \eqref{deffa} gives
\[
f_{a,2}(\mu e)=(a\cdot \mu_1)f_{a|_{\mu_1}, 1}(e)=(a\cdot \mu_1)(a|_{\mu_1}\cdot e)\in f_{a,1}(\mu_1)E^1.
\]
We now suppose that what we want is true for $\{f_{a,k}:a\in A\}$ for some $k\geq 1$, and take $\mu e\in s(a)E^{k+2}$. Then since $|\mu|=k+1$, we can factor $\mu e=\mu_1\mu' e$, and then
\[
f_{a,k+2}(\mu e)=(a\cdot \mu_1)f_{a|_{\mu_1},k+1}(\mu' e)\in (a\cdot \mu_1)f_{a|_{\mu_1},k}(\mu')E^1=f_{a,k+1}(\mu)E^1,
\]
as required.

Since $r(v)=v=s(v)$, $v\cdot e=e$ and $v|_e=v$ for all $e\in vE^*$, a simple induction argument shows that $f_v$ is the identity on $vE^*$. 
\end{proof}

\begin{thm}\label{defnGA}
Suppose that $E$ is a directed graph and $A$ is an automaton over $E$. For $a\in A$, let $f_a$ be the partial isomorphism of $T_E$ described in Proposition~\ref{extenda}, and let $G_A$ be the subgroupoid of $\PI(E^*)$ generated by $\{f_a:a\in A\}$ (which by convention includes the identity morphisms $\{\id_v:v\in E^0\}$). Then $G_A$ acts faithfully on the path space $E^*$, and this action is self-similar.
\end{thm}

This is the point where it really helps to have worked carefully through the classical case in which $|E^0|=1$ and $E^1$ is just a finite set $X$. Notice that $G_A$ acts faithfully because it is by definition a subgroupoid of $\PI(E^*)$.

\begin{proof}[Proof of Theorem~\ref{defnGA}] The inverses $f_a^{-1}:r(a)E^*\to s(a)E^*$ of the bijections $f_a:s(a)E^*\to r(a)E^*$ are also partial isomorphisms (see Proposition~\ref{prop:alltreeisos}; the crucial argument is that of Lemma~\ref{self-similar defn of auto}, which shows that $f_a^{-1}$ satisfies \eqref{prop:alltreeisos}). Then every element of $G_A$ is a finite product of partial isomorphisms $f_a$ and $f_b^{-1}$ in which the domains and codomains match up. (So, for example, $f_af_b^{-1}$ where $s(a)=c(f_b^{-1})=s(b)$.) Thus it suffices for us to show first, that if a partial automorphism $g\in \PI(E^*)$ has restrictions $g|_e$ for all $e\in d(g)E^1$, then so does $g^{-1}$; and second, that if $g_i\in \PI(E^*)$ have $d(g_1)=c(g_2)$, and $g_i$ have restrictions $g_i|_e$ for all $e\in d(g_i)E^1$, then so does $g_1g_2$. But for these we just need to look at the second and third paragraphs in the proof of Proposition~\ref{self similar action of G_A} to see that they carry over verbatim. So $G_A$ acts self-similarly, as claimed. 
\end{proof}

\begin{example}
\label{s,r clarifying example}

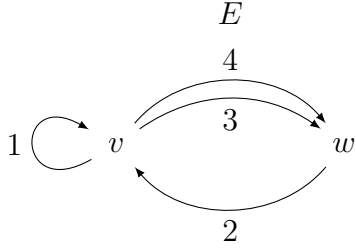
\begin{figure}
\begin{tikzpicture}[scale=1.0]
\node[vertex] (p) at (-8.0,0.5) {$E$};
\begin{scope}[xshift=-6.5cm,yshift=-1.25cm]
\node at (0,0) {$w$};
\node[vertex] (vertexe) at (0,0)   {$\quad$};
\node at (-3,0) {$v$};
\node[vertex] (vertex-a) at (-3,0)   {$\quad$}
	edge [->,>=latex,out=35,in=145] node[below,swap,pos=0.5]{$3$} (vertexe)
	edge [->,>=latex,out=50,in=130] node[above,swap,pos=0.5]{$4$} (vertexe)
	edge [->,>=latex,out=210,in=150,loop] node[left,pos=0.5]{$1$} (vertexe)
	edge [<-,>=latex,out=310,in=230] node[below,swap,pos=0.5]{$2$} (vertexe);
\end{scope}
\end{tikzpicture}
\caption{Graph for Example~\ref{s,r clarifying example}}
\label{asymmetric E}
\end{figure}

We consider the graph $E$ in Figure~\ref{asymmetric E}, 
so that $E^0=\{v,w\}$, $E^1=\{1,2,3,4\}$, 
$v=s_E(1)=r_E(1)=r_E(2)=s_E(3)=s_E(4)$ and $w=s_E(2)=r_E(3)=r_E(4)$. 
We take $A=\{a,b\}\cup E^0$, and define $r_A,s_A:A\to E^0$ 
by $s_A(a)=v=r_A(b)$ and $s_A(b)=w=r_A(a)$. 
We define
\begin{align}\label{defasymm}
a\cdot 1=4,\ a|_1=v;\quad\quad&b\cdot 3=1, \ b|_3=v;\\
a\cdot 2=3,\ a|_2=b;\quad\quad&b\cdot 4=2,\ b|_4=a.\notag
\end{align}
(The restrictions on $E^0\subset A$ are specified in (A3).) 
We verify that $e\mapsto a\cdot e$ is a bijection of $s_A(a)E^1=vE^1=\{1,2\}$ onto 
$r_A(a)E^1=wE^1=\{3,4\}$, and correspondingly for $b$. 
For (A2), we check 
\begin{align*}
s_A(a|_1)=s_A(v)=v=s_E(1)  \quad\quad &   r_A(a|_1)=r_A(v)=v=s_E(a \cdot 1) \\
s_A(a|_2)=s_A(b)=w=s_E(2)  \quad\quad&   r_A(a|_2)=r_A(b)=v=s_E(a \cdot 2) \\
s_A(b|_3)=s_A(v)=v=s_E(3) \quad\quad &    r_A(b|_3)=r_A(v)=v=s_E(b \cdot 3) \\
s_A(b|_4)=s_A(a)=v=s_E(4) \quad\quad &    r_A(b|_4)=r_A(w)=w=s_E(b \cdot 4).
\end{align*}
The data \eqref{defasymm} is often presented as 
\begin{align}
\label{nonsymmetric2_def}
a \cdot 1\mu=4\mu \quad\quad\quad\quad & b \cdot 3\mu=1\mu \\
\nonumber
a \cdot 2\nu=3(b\cdot \nu) \quad\quad &    b \cdot 4\mu=2(a \cdot \mu),
\end{align}
from which we are meant to deduce \eqref{defasymm}. 

To see how this deduction is meant to work, we observe that the formulas in \eqref{nonsymmetric2_def} allow us to 
define partial isomorphisms $f_a$ and $f_b$ recursively, 
using the construction that was formalised in Proposition~\ref{extenda}. So $a$ and $b$ have to be in the alphabet; we take $A=\{a,b\}\cup E^0$.
Consider the action of $a\in A$.
By assuming $1\mu\in E^*$ we are assuming $\mu\in s_E(1)E^*=vE^*$. 
Since $a \cdot 1\mu=4\mu$ for all $\mu\in vE^*$, 
it follows from Proposition~\ref{extenda} that 
$a\cdot 1 = 4$ and $f_{a|_1}(\mu)=\mu$ for all $\mu\in vE^*$. So whatever $a|_1$ is, it must satisfy $f_{a|_1}=\id_v$; since the partial isomorphisms $f_a$, $f_b$, $\id_v$ and $\id_w$ all act differently on the edge set $E^1$, they are distinct partial isomorphisms, and the only choice for $a|_1$ in our current alphabet is $a|_1:=v$. Similarly, we define $a|_2=b$, $b|_3=v$ and $b|_4=a$. We don't need to specify restrictions for $v$ and $w$, because the condition $r(e)|_e=s(e)$ forces $v|_1=v$, $v|_2=w$ and $w|_3=w|_4=v$.

Notice that we made choices in the last paragraph: we could have added extra elements to our alphabet. 
\end{example}

\section{The Toeplitz algebra of a self-similar groupoid action}\label{sec:Toep}

Suppose that $G$ is a (discrete) groupoid. Then the groupoid elements $g\in G$ give point masses $i_g$ in $C_c(G)$, and $C_c(G)=\sp\{i_g:g\in G\}$. For $g,h\in G$, the involution and product described in \cite[Proposition~3.11]{e3} are determined by
\[
i_g^*=i_{g^{-1}}\quad\text{and}\quad i_g*i_h=\begin{cases}i_{gh}&\text{if $d(g)=c(h)$}\\
0&\text{otherwise.}
\end{cases}
\]

A function $U:G\to B(H)$ is a \emph{unitary representation} of $G$ if
\begin{itemize}
\item for $v\in G^0$, $U_v$ is the orthogonal projection on a closed subspace of $H$,
\item for each $g\in G$, $U_g$ is a partial isometry with initial projection $U_{d(g)}$ and final projection $U_{c(g)}$, and
\item for $g,h\in G$, we have
\begin{equation}\label{defmult}
U_gU_h=\begin{cases}U_{gh}&\text{if $d(g)=c(h)$}\\0&\text{otherwise.}\end{cases}
\end{equation}
\end{itemize}
(The multiplication formula \eqref{defmult} applies when $g$ and $h$ are units, and then says that the projections $U_v$ are mutually orthogonal.) We have a similar notion of unitary representation with values in a $C^*$-algebra, and then the map $i:g\mapsto i_g$ is a unitary representation of $G$ in $C_c(G)\subset C^*(G)$. We have chosen the name ``unitary representation'' to emphasise that each $U_g$ is a unitary isomorphism of $U_{d(g)}H$ onto $U_{c(g)}H$.

\begin{prop}\label{univC*G}
Let $G$ be a groupoid. Then the pair $(C^*(G),i)$ is universal for unitary representations of $G$.
\end{prop}

\begin{proof}
Suppose that $U:G\to B(H)$ is a unitary representation. Then extending $U$ linearly gives a $*$-representation $\pi$ of $C_c(G)=\sp\{i_g\}$ in $B(H)$ such that $U=\pi\circ i$. Since this representation is bounded for the norm of $C^*(G)$ (see \cite[page 205]{e3}), $\pi$ extends to a representation $\pi_U$ of $C^*(G)$ on $H$ such that $\pi_U\circ i=U$.
\end{proof}

\begin{example}\label{easyrep}
For any groupoid $G$, there is a unitary representation $U$ of $G$ on $\ell^2(G^0)$ in which  $U_g$ is the matrix unit $e_{c(g),d(g)}$. The corresponding homomorphism $\pi_U:C^*(G)\to B(\ell^2(G^0))$ takes elements of $C_c(G)$ into finite rank operators, and hence has range in the algebra $\KK(\ell^2(G^0))$ of compact operators.

When $G^0$ is finite, $\pi_U$ takes values in $M_{G^0}(\C)$, and composing with the normalised trace gives a tracial state on $C^*(G)$. This is the analogue for groupoids of the trace $\tau_1$ on group algebras used in \cite{lrrw} (see page 6648). If $G$ is not transitive, then the range of $\pi_U$ will be a proper subalgebra of $M_{G^0}(\C)$, and the normalised traces on simple summands will give other tracial states on $C^*(G)$.
\end{example}

Suppose that $E$ is a finite directed graph and $(G,E)$ is a self-similar groupoid action. Consider the graph bimodule $X(E)$ of \cite{fr}, using the conventions of \cite[Chapter~8]{R_book}. The inclusion of $C(E^0)=C_c(G^0)$ in $C_c(G)$ gives a unital homomorphism of $C(E^0)$ into $C^*(G)$, which makes $C^*(G)$ into a right-Hilbert $C(E^0)$-$C^*(G)$ bimodule. We now form the internal tensor product $M:=X(E)\otimes_{C(E^0)}C^*(G)$, which is a right Hilbert $C^*(G)$-module with inner product given on elementary tensors by
\[
\langle x\otimes a,y\otimes b\rangle=\langle\,\langle y,x\rangle a,b\rangle = a^*\langle x,y \rangle b.
\]

Forming the internal tensor product involves taking a completion, which in turn involves modding out by elements with norm $0$. In this process, the tensor product becomes balanced over $C(E^0)$ (see the foot of \cite[page~34]{tfb}, for example). However, the bimodule $X(E)$ is spanned by the point masses $\{e:e\in E^1\}$, and the elements $e\otimes 1$ form a Parseval frame for $M$. In particular, every $m\in M$ is a finite sum
\begin{equation}\label{recon}
m=\sum_{e\in E^1}e\otimes \langle e\otimes 1,m\rangle
\end{equation}
of elementary tensors; we refer to \eqref{recon} as the \emph{reconstruction formula}. Thus $M$ coincides with the algebraic tensor product $X(E)\odot_{C(E^0)}C^*(G)$, and no completion is required. 

\begin{prop}\label{Toepres}
Suppose that $(G,E)$ is a self-similar groupoid action. Then there is a unitary representation $W:G\to \LL(M)$ such that, for $g\in G$, $e\in E^1$ and $a\in C^*(G)$,
\begin{equation}\label{defWg}
W_g(e\otimes a)=\begin{cases}(g\cdot e)\otimes i_{g|_e}a&\text{if $d(g)=r(e)$}\\
0&\text{otherwise.}
\end{cases}
\end{equation}
\end{prop}

\begin{proof}
Fix $g\in G$. Then, motivated by the reconstruction formula \eqref{recon}, we define $W_g:M\to M$ by 
\[
W_gm=\sum_{f\in E^1,\;d(g)=r(f)}(g\cdot f)\otimes (i_{g|_f}\langle f\otimes 1,m\rangle).
\]
Now taking $m=e\otimes a$ and considering only the edges  $f$ with $r(f)=d(g)$ (that is, the only edges which appear in the sum), we have
\[
i_{g|_f}\langle f\otimes 1,e\otimes a\rangle=i_{g|_f}\langle \langle e,f\rangle 1,a\rangle=i_{g|_f}\langle f,e\rangle a =\delta_{e,f}i_{g|_e}i_{s(e)}a,
\]
which is $i_{g|_e}a$ because $g|_e s(e)=g|_e$ in $G$. Thus $W_g$ satisfies \eqref{defWg}. The original formula shows that $W_g$ is right $C^*(G)$-linear, and since the sum is finite, an application of the Cauchy-Schwarz inequality shows that $W_g$ is norm-bounded.

We next show that $W_g$ is adjointable with $W_g^*=W_{g^{-1}}$. We take $e,f\in E^1$ and $a,b\in C^*(G)$, and suppose first that
\[
\langle W_g(e\otimes a),f\otimes b\rangle\not=0.
\]
The formula \eqref{defWg} shows that $d(g)=r(e)$, and then
\begin{equation}
\label{lhs}0\not=\langle W_g(e\otimes a),f\otimes b\rangle=\langle\,\langle f,g\cdot e\rangle \cdot i_{g|_e}a,  b\rangle=a^*i_{g|_e}^*\langle g\cdot e,f\rangle b.
\end{equation}
We deduce that $\langle g\cdot e,f\rangle\not=0$, and hence $f=g\cdot e$. Thus
\[
r(f)=r(g\cdot e)=g\cdot r(e)=g\cdot d(g)=c(g)=s(g^{-1}),
\]
and we have
\begin{equation}
\label{rhs}\langle e\otimes a,W_{g^{-1}}(f\otimes b)\rangle=\langle\,\langle g^{-1}\cdot f, e\rangle \cdot a, i_{g^{-1}|_f} b\rangle=a^*\langle e,g^{-1}\cdot f\rangle i_{g^{-1}|_f}b.
\end{equation}
Then $\langle g\cdot e,f\rangle=i_{s(f)}$ and $\langle e,g^{-1}\cdot f\rangle=i_{s(e)}$, so 
\begin{align*}
i_{g|_e}^*\langle g\cdot e,f\rangle&=i_{g|_e}^*i_{s(f)}=  i_{g|_e^{-1}}i_{s(g\cdot e)}=   i_{g^{-1}|_{g\cdot e}}i_{s(g\cdot e)}\\
&=i_{g^{-1}|_{g\cdot e}}=i_{s(e)}i_{g^{-1}|_{g\cdot e}}\\
&=\langle e,g^{-1}\cdot f\rangle i_{g^{-1}|_f}.
\end{align*}
Thus we have \eqref{lhs}=\eqref{rhs} when $\langle W_g(e\otimes a),f\otimes b\rangle\not=0$. When $\langle e\otimes a,W_{g^{-1}}(f\otimes b)\rangle\not=0$, we can apply the preceding argument to 
\[
\langle e\otimes a,W_{g^{-1}}(f\otimes b)\rangle^*=\langle W_{g^{-1}}(f\otimes b),e\otimes a\rangle,
\]
and deduce that \eqref{rhs}=\eqref{lhs}. Thus we have equality in all cases, and $W_g$ is adjointable with $W_g^*=W_{g^{-1}}$.

Finally, a short calculation using Proposition~\ref{SSA:extension to paths}\eqref{path_rest_nonhomo} shows that $W$ is a unitary representation of $G$.
\end{proof}

The universal property of $(C^*(G),i)$ (Proposition~\ref{univC*G}) gives us a unital homomorphism $\pi_W:C^*(G)\to \LL(M)$ such that $\pi_W(i_g)=W_g$ for $g\in G$. We use this homomorphism to define a left action of $C^*(G)$ on the right-Hilbert $C^*(G)$ module $M=X(E)\otimes_{C(E^0)}C^*G)$, and $M$ thus becomes a Hilbert bimodule over $C^*(G)$ (alternatively known as a $C^*$-correspondence over $C^*(G)$). The \emph{Toeplitz algebra of the self-similar groupoid action $(G,E)$} is then by definition the Toeplitz algebra $\TT(G,E):=\TT(M)$ of the bimodule $M$. As in \cite[\S1]{fr}, we view it as the $C^*$-algebra generated by a universal Toeplitz representation $(i_M,i_{C^*(G)}):M\to\TT(M)$.

We now give a presentation of the Toeplitz algebra $\TT(G,E)$.

\begin{prop}\label{Toeplitz_repn}
Let $E$ be a finite graph without sources and $(G,E)$ a self-similar groupoid action. We define $u:G\to \TT(G,E)=\TT(M)$, $p:E^0\to \TT(M)$ and $s:E^1\to \TT(M)$ by $u=i_{C^*(G)}\circ i$, $p_v=i_{C^*(G)}(i_v)$, and $s_e=i_M(e\otimes 1)$. Then
\begin{enumerate}
\item\label{TR_pi} $u$ is a unitary representation of $G$ with $u_v=p_v$ for $v\in E^0$;
\item\label{TR_TCKfamily} $(p,s)$ is a Toeplitz-Cuntz-Krieger family in $\TT(M)$, and $\sum_{v\in E^0}p_v$ is an identity for $\TT(M)$;
\item\label{TR_SSA} for $g\in G$ and $e\in E^1$, we have
\[
u_gs_e=\begin{cases}
s_{g\cdot e}u_{g|_e}&\text{if $d(g)=r(e)$}\\
0&\text{otherwise;}
\end{cases}
\]
\item\label{TR_SSA_vert}  for $g\in G$ and $v\in E^0$, we have
\[
u_gp_v=\begin{cases}
p_{g\cdot v}u_{g}&\text{if $d(g)=v$}\\
0&\text{otherwise.}
\end{cases}
\]
\end{enumerate}
The $C^*$-algebra $\TT(M)$ is generated by $\{u_g\}\cup\{p_v\}\cup\{s_e\}$, and $(\TT(M), (u,p,s))$ is universal for families $U$, $P$, $S$ satisfying the upper-case analogues of \textnormal{(\ref{TR_pi}--\ref{TR_SSA_vert})}.
\end{prop}

\begin{proof}
Part~\eqref{TR_pi} holds because $i:G\to C^*(G)$ is a unitary representation.

For part~\eqref{TR_TCKfamily}, we begin by observing that the point masses $\{v:v\in E^0\}$ are mutually orthogonal projections in $C(E^0)$, and hence the $\{p_v\}$ are mutually orthogonal projections in $\TT(M)$. In particular $q=\sum_{v\in E^0}p_v$ is also a projection. Next we note that the elements 
\[
\{e\otimes i_g:e\in E^1, g\in G, s(e)=c(g)\}
\]
span a dense subspace of $M$, and that for each such element we have
\[
q(e\otimes i_g)=\sum_{v\in E^0}(v\cdot e)\otimes i_g=(s(e)\cdot e)\otimes i_g=e\otimes i_g.
\]
Thus the bimodule $M$ is essential in the sense of \cite{fr}, and it follows that $q$ is an identity for $\TT(M)$. (See the discussion in the middle of \cite[page~178]{fr}.)
We next verify the Toeplitz-Cuntz-Krieger relations. We take $e,f\in E^1$ and compute
\begin{align*}
s_e^*s_f&=i_{M_e}(e\otimes 1)^*i_{M_e}(f\otimes 1)=i_{C^*(G)}\big(\langle e\otimes 1,f\otimes 1\rangle\big)\\
&=\delta_{e,f}i_{C^*(G)}(s(e)\cdot 1)=\delta_{e,f}i_{C^*(G)}(i_{s(e)})=\delta_{e,f}p_{s(e)};
\end{align*}
this proves that $s_s^*s_e=p_{s(e)}$, and that the elements $\{s_e s_e^* : e \in E^1\}$ are mutually orthogonal projections. We calculate
\begin{align}\label{proj_isom_comp}
p_{v}s_e&=i_{C^*(G)}(i_v)i_{M}(e\otimes 1)=i_{M}(W_v(e\otimes 1)) = i_{M}((v\cdot e)\otimes i_{v|_e}1)\\
&=\delta_{v,r(e)}i_{M}(e\otimes i_{s(e)})=\delta_{v,r(e)} i_{M_e}(e\otimes 1) = \delta_{v,r(e)} s_e; \notag
\end{align}
thus $p_{r(e)}\geq s_e s_e^*$, and, since the $\{s_e s_e^*\}$ are mutually orthogonal, $p_v\geq \sum_{e\in vE^0}s_e s_e^*$. Thus $(p,s)$ is a Toeplitz-Cuntz-Krieger $E$-family.

To establish part \eqref{TR_SSA}, we compute
\begin{align*}
u_g s_e&= i_{C^*(G)}(i_g)i_{M}(e\otimes 1)=i_{M}(i_g \cdot (e\otimes 1))\\
&=\delta_{d(g),r(e)}i_{M}((g\cdot e)\otimes i_{g|_e})
=\delta_{d(g),r(e)}i_{M}(((g\cdot e)\otimes 1)\cdot i_{g|_e})\\
&=\delta_{d(g),r(e)}i_{M}((g\cdot e)\otimes 1)i_{C^*(G)}(i_{g|_e})=\delta_{d(g),r(e)}s_{g \cdot e}u_{g|_e}.
\end{align*}
A similar computation gives \eqref{TR_SSA_vert}.

Since $\TT(M)$ is generated by the images of $i_M=\clsp\{e\otimes i_g\}$ and $i_{C^*(G)}=\clsp\{i_g\}$, the $t_g$, $p_v$ and $s_e$ generate. It remains to verify the universal property, so we take $U_g$, $P_v$ and $S_e$ in a $C^*$-algebra $B$ satisfying the relations (1--4).

The unitary representation $U$ of $G$ induces a homomorphism $\pi_U:C^*(G)\to B$ such that $\pi_U(i_g)=U_g$. Since $\sum_{v\in E^0}i_v$ is the identity in $C^*(G)$ and \eqref{TR_TCKfamily} says that $\sum_{v\in E^0}P_v$ is the identity in $B$, $\pi_U$ is unital. Define $\psi:M\to B$ by
\[
\psi(m)=\sum_{e\in E^1}S_e\pi_U(\langle e\otimes 1,m\rangle).
\]
We will prove that $(\psi,\pi_U)$ is a Toeplitz representation of $M$.

An inspection of the formula for $\psi(m)$ shows that $\psi(m\cdot a)=\psi(m)\pi_U(a)$. Next, for $m,n\in M$ we have
\[
\psi(m)^*\psi(n)=\sum_{e,f\in E^1}\pi_U(\langle e\otimes 1,m\rangle^*S_e^*S_f\pi_U(\langle f\otimes 1,n\rangle).
\]
The summands with $e\not=f$ vanish, and 
\begin{align*}
S_e^*S_e\pi_U(\langle e\otimes 1,n\rangle)&=P_{s(e)}\pi_U(\langle e\otimes 1,n\rangle)=\pi_U(i_{s(e)}\langle e\otimes 1,n\rangle)=\pi_U(\langle e\otimes 1,n\rangle).
\end{align*}
Thus 
\begin{align*}
\psi(m)^*\psi(n)&=
\sum_{e\in E^1}\pi_U\big(\langle e\otimes 1,m\rangle^*\langle e\otimes 1, n\rangle\big) \\
&=\pi_U\Big(\sum_{e\in E^1}\langle m, e\otimes 1\rangle\langle e\otimes 1, n\rangle\Big)\\
&=\pi_U\Big(\Big\langle m,\sum_{e\in E^1}(e\otimes 1)\langle e\otimes 1, n\rangle\Big\rangle\Big),
\end{align*}
which by the reconstruction formula is $\pi_U(\langle m,n\rangle)$.

Next we have to verify that $\psi(b\cdot m)=\pi_U(b)\psi(m)$ for $b\in C^*(G)$ and $m\in M$. For this we take $b=i_g$, $m =f\otimes a$ and compute both sides. On one hand, we have
\[
\psi(i_g(f\otimes a))=\psi(W_g(f\otimes a)),
\]
which is $0$ unless $d(g)=r(f)$. If $d(g)=r(f)$, then
\begin{align*}
\psi(i_g(f\otimes a))&=\psi((g\cdot f)\otimes i_{g|_f}a)\\
&= \sum_{e\in E^1}S_e \langle e\otimes 1 , gf \otimes i_{g|_f}a\rangle\\
&=\sum_{e\in E^1}S_e\pi_U\big(\big\langle\,\langle g\cdot f,e\rangle,i_{g|_f}a\big\rangle\big)\\
&=s_{g\cdot f}\pi_U(i_{g|_f}a).
\end{align*}
On the other hand, we have
\[
\pi_U(i_g)\psi(f\otimes a)=\sum_{e\in E^1}U_gS_e\pi_U\big(\langle e\otimes 1,f\otimes a\rangle\big)=U_gS_f\pi_U(a).
\]
The formula in \eqref{TR_SSA} says that this too is $0$ unless $d(g)=r(f)$, and then gives
\[
\pi_U(i_g)\psi(m)=S_{g\cdot f}U_{g|_f}\pi_U(a)=S_{g\cdot f}\pi_U(i_{g|_f}a),
\]
which is $\psi(i_g(f\otimes a))$.

Now the universal property of $\TT(M)$ give a homomorphism $\psi\times \pi_U:\TT(M)\to B$ such that $\psi=(\psi\times \pi_U)\circ i_M$ and $\pi_U=(\psi\times \pi_U)\circ i_{C^*(G)}$. To finish, we need to check that $\psi\times \pi_U$ maps $(u,p,s)$ to $(U,P,S)$. This is straightforward for $u$ and $p$. For $f\in E^1$, we have
\begin{align*}
(\psi\times \pi_U)(s_f)&=(\psi\times \pi_U)(i_M(f\otimes 1))=\psi(f\otimes 1)\\
&=\sum_{e\in E^1}S_e\pi_U(\langle e\otimes 1,f\otimes 1\rangle),
\end{align*}
which collapses to $S_f\pi_U(i_{s(f)})=S_fP_{s(f)}=S_f$.
 \end{proof}

\begin{prop}\label{Toeplitz_spanning}
Suppose $(G,E)$ is a faithful self-similar groupoid action and take $(u,p,s)$ as in Proposition~\ref{Toeplitz_repn}. Then
\[
\TT(G,E)=\clsp\{s_\mu u_g s_\nu^* : \mu,\nu \in E^*, g \in G \text{ and } s(\mu)=g\cdot s(\nu)\}.
\]
\end{prop}

The condition $s(\mu)=g\cdot s(\nu)$ is just there to exclude elements which are certainly $0$: 
\[
s_\mu u_g s_\nu^*\not=0\Longrightarrow s_\mu p_{s(\mu)}u_g p_{s(\nu)}s_\nu^*\not=0\Longrightarrow p_{s(\mu)}u_g p_{s(\nu)}\not=0\Longrightarrow p_{s(\mu)}p_{g\cdot s(\nu)}\not=0.
\]
The closed linear span on the right-hand side contains all the generators, and hence the result follows from the following lemma, which is very similar to \cite[Lemma 3.4]{lrrw}.

\begin{lemma}\label{prodform}
Suppose $\kappa, \lambda, \mu,\nu \in E^*$ and $g,h \in S$ satisfy $s(\kappa)=g \cdot s(\lambda)$ and $s(\mu)=h\cdot s(\nu)$ we have
\begin{equation}\label{spanning_mult}
(s_\kappa u_g s_\lambda^*)(s_\mu u_h s_\nu^*)=\begin{cases}
s_{\kappa(g \cdot\mu')}u_{g|_{\mu'} h} s_\nu^* & \text{ if } \mu=\lambda \mu' \\
s_{\kappa} u_{gh|_{h^{-1}\cdot\lambda'}}s_{\nu(h^{-1}\cdot \lambda')}^* & \text{ if } \lambda=\mu\lambda' \\
0 & \text{ otherwise.} \\
\end{cases}
\end{equation}
\end{lemma}

\begin{prop}\label{Cuntz_repn}
Suppose $(G,E)$ is a self-similar groupoid action and let $(u,p,s)$ be the universal representation from  \proref{Toeplitz_repn}. Then $\OO(G,E)$ 
is the quotient of $\TT(G,E)$ by the ideal generated by 
\[
\Big\{p_v - \sum_{\{e\in vE^1\}} s_e s_e^*: v \in E^0\Big\}. 
\]
\end{prop}
\begin{proof}
Let $\phi:C^*(G)\to \LL(M)$ be the homomorphism implementing the left action, and let $(\psi,\pi)$ be a Toeplitz representation of $M$. Then the reconstruction formula implies that
\[
\phi(a)=\sum_{e\in E^1}\Theta_{e\otimes 1, a^*(e\otimes 1)}\quad\text{for all $a\in C^*(G)$.}
\]
Thus for every $a\in C^*(G)$ we have
\begin{equation}\label{CPcov}
(\psi,\pi)^{(1)}(\phi(a))=\sum_{e\in E^1}\psi(e\otimes 1)\psi(a^*\cdot (e\otimes 1))^*=\sum_{e\in E^1}s_es_e^*\pi(a).
\end{equation}
If $(\psi,\pi)$ is Cuntz-Pimsner covariant, then taking $a=i_v$ in \eqref{CPcov} gives 
\[
\pi(p_v)=(\psi,\pi)^{(1)}(\phi(p_v))=\sum_{r(e)=v}\psi(s_e)\psi(s_e)^*.
\]
Conversely, if we have the Cuntz Krieger relation at every $v$, then for every $g\in G$ we have
\begin{align*}
\pi(i_g)&=\pi(i_{c(g)}i_g)=\sum_{r(e)=c(g)}\psi(s_e)\psi(s_e)^*\pi(i_g)\\
&=\sum_{e\in E^1}\psi(s_e)\psi(s_e)^*\pi(i_g)\\
&=(\psi,\pi)^{(1)}(\phi(i_g)).\qedhere
\end{align*}
\end{proof}

\section{An algebraic characterisation of KMS states}\label{sec:algchar}

Suppose that $E$ is a finite graph with no sources and that $(G,E)$ is a self-similar groupoid action over $E$.
There is a strongly continuous gauge action $\gamma: \T \to \Aut \TT(G,E)$ such that $\gamma_z(i_{C^*(G)}(a))=i_{C^*(G)}(a)$ and $\gamma_z(i_M(m))=z i_M(m)$ for $a \in {C^*(G)}$ and $m \in M$ \cite[Proposition 1.3 (c)]{fr}. 
The gauge action gives rise to a periodic action $\sigma$ of the real line (a dynamics) by the formula $\sigma_t=\gamma_{e^{i t}}$. This dynamics satisfies
\begin{equation}\label{T_spanning_sigma}
\sigma_t(s_\mu u_g s_\nu^*)=e^{it(|\mu|-|\nu|)}s_\mu u_g s_\nu^*.
\end{equation}

Suppose that $\sigma:\R \to \Aut A$ is a strongly continuous action on a $C^*$-algebra $A$. An element $a \in A$ is analytic if the function $t \mapsto \sigma_t(a)$ extends to an entire function on $\C$. For $\beta \in \R$ a state $\phi$ on $A$ is a KMS$_\beta$ state for $(A,\sigma)$ if it satisfies the KMS$_\beta$ condition
\begin{equation}\label{KMS_condition}
\phi(ab)=\phi(b \sigma_{i\beta}(a)),
\end{equation}
for all $a,b$ in a set of analytic elements that spans a dense $\sigma$-invariant subspace of $A$.
For $\beta \neq 0$ all KMS$_\beta$ states are $\sigma$-invariant, see \cite[Proposition 5.3.3]{bra-rob}; for $\beta=0$ it is customary to make $\sigma$-invariance part of the definition, so that the KMS$_0$ states are the $\sigma$-invariant traces.

We are interested in the KMS$_\beta$ states of $(\TT(G,E), \sigma)$. The spanning elements $s_\mu u_g s_\nu^*$ are analytic in $\TT(G,E)$, because the function $t \in\R  \mapsto e^{it(|\mu|-\nu|)}$ has an entire extension to $\C$.
Therefore it is enough to check the KMS$_\beta$ condition on spanning elements.

Since the ideal generated by the elements $p_v - \sum_{e \in vE^1} s_e s_e^*$ 
is $\sigma$-invariant, there is a compatible time evolution (also denoted by $\sigma$) on  the quotient $\OO(G,E)$. The
KMS states of $(\OO(G,E), \sigma)$ are given by KMS states of $(\TT(G,E), \sigma)$ that factor through $\OO(G,E)$.
The following alternative characterisation of the KMS$_\beta$ states of  $\TT(G,E)$ and $\OO(G,E)$ is motivated by \cite[\S2]{hlrs} and \cite[\S4]{lrrw}. 

\begin{prop}\label{KMS_algebraic}
Let  $E$ be a finite graph with no sources and vertex matrix B, and let $\rho(B)$ be the spectral radius of $B$.
Suppose that $(G,E)$ is a self-similar groupoid action. 
Let $\sigma: \R \to \Aut\TT(G,E)$ be the dynamics given by  \eqref{T_spanning_sigma}.
\begin{enumerate}
\item\label{no_states} For $\beta < \ln \rho(B)$, there are no KMS$_{\beta}$-states for $\sigma$.
\item\label{states_iff} For $\beta \geq \ln \rho(B)$, a state $\phi$ is a KMS$_{\beta}$-state for $\sigma$ if and only if $\phi \circ i_{C^*(G)}$ is a trace on $C^*(G)$ and 
\begin{equation}\label{char:spanning}
\phi(s_\mu u_g s_\nu^*)=
\delta_{\mu,\nu} \delta_{s(\mu),c(g)}\delta_{s(\nu),d(g)}e^{-\beta |\mu|}\phi( u_g) \qquad \text{ for } g \in S \text{ and } \mu,\nu \in E^*.
\end{equation}
\end{enumerate}
\end{prop}

In principle, we should be able to deduce part~\eqref{states_iff} from the general result in \cite[Proposition~3.1]{aahr}, which characterises the KMS states on an arbitrary Toeplitz-Cuntz-Pimsner algebra $\TT(X)$. However, the description there is in terms of elementary tensors in the tensor powers $X^{\otimes n}$, and since our module $M$ is itself a tensor product, the tensor powers get quite complicated. So it seems more straightforward to reason directly in terms of our spanning family.

\begin{proof} Suppose first that $\phi$ is a KMS$_{\beta}$-state of $(\TT(G,E), \sigma)$.
The Toeplitz-Cuntz-Krieger family $(p,s)$ in $\TT(G,E)$
induces a homomorphism $\pi_{p,s}:\TT (E)\to \TT(G,E)$, and $\pi_{p,s}$ is equivariant for $\sigma$ and the periodic dynamics arising from the gauge action on $\TT (E)$.
Thus the composition $\phi\circ\pi_{p,s}$ is a KMS$_\beta$-state on $\TT (E)$, and \cite[Proposition 4.3~(c)]{hlrs} implies that $\beta \geq \ln \rho(B)$. This proves part \eqref{no_states}.

For part~\eqref{states_iff}, suppose again that $\phi$ is a KMS$_\beta$ state. Since $\gamma$ fixes $u_g$ and $u_h$, the KMS relation implies that $\phi(u_gu_h)=\phi(u_hu_g)$, and hence that $\phi\circ i_{C^*(G)}$ is a trace. The second and third delta functions on the right-hand side of \eqref{char:spanning} just reflect that $s_\mu u_g s_\nu^*=0$ unless $s(\mu)=c(g)$ and $d(g)=s(\nu)$. Since $\phi$ is gauge-invariant, $\phi(s_\mu u_g s_\nu^*)=0$ unless $|\mu|=|\nu|$, and then using the KMS commutation relation to pull $s_\mu$ past $u_gs_\nu^*$ gives \eqref{char:spanning}.

For the reverse implication, suppose that $\phi\circ i_{C^*(G)}$ is a trace and that $\phi$ satisfies \eqref{char:spanning}. To see that $\phi$ is a KMS$_\beta$ state, it suffices to verify the KMS condition
\begin{equation}\label{KMScondbc}
\phi(bc)=\phi(c\alpha_{i\beta}(b))\quad \text{for $b=s_\kappa u_g s_\lambda^*$ and $c=s_\mu u_h s_\nu^*$.}
\end{equation}
There is symmetry in \eqref{KMScondbc} which we can use to simplify the argument. First, note that both sides vanish unless $|\kappa|+|\mu|=|\lambda|+|\nu|$. The right-hand side is $e^{-\beta(|\kappa|-|\lambda|)}\phi(cb)$, and \eqref{KMScondbc} is equivalent to
\[
\phi(cb)=e^{\beta(|\kappa|-|\lambda|)}\phi(bc)=e^{-\beta(|\mu|-|\nu|)}\phi(bc).
\]
Thus it suffices to prove \eqref{KMScondbc} for $\phi(bc)\not=0$, because we can then swap $b$ and $c$ and deduce it for $\phi(cb)\not=0$. Next we observe that it suffices to prove \eqref{KMScondbc} when $|\mu|\geq |\lambda|$, because taking complex conjugates reduces the other case to this one (see the end of the proof of \cite[Proposition~4.1]{lrrw}).

So we suppose that $\phi(bc)\not=0$ and $|\mu|\geq |\lambda|$. Then since $\phi(bc)\not=0$, \eqref{spanning_mult} implies that there exists $\mu'$ such that $\mu=\lambda\mu'$. Then the first option in \eqref{spanning_mult} shows that
\begin{align}
\phi(bc)&=\phi(s_{\kappa(g\cdot \mu')} u_{g|_{\mu'}}u_hs_\nu^*)\notag\\
&=\delta_{\kappa(g\cdot \mu'),\nu}\delta_{s(\mu),r(g|_{\mu'})}
\delta_{s(\nu),d(h)}e^{-\beta(|\kappa|+|\mu'|)}\phi(u_{g|_{\mu'}}u_h).\label{phibc}
\end{align}
Thus $\phi(bc)\not=0$ implies that $\nu=\kappa(g\cdot \mu')$, $s(\mu)=r(g|_{\mu'})=d(h)$ and 
\[
\phi(bc)=e^{-\beta(|\kappa|+|\mu'|)}\phi(u_{g|_{\mu'}}u_h);
\]
since then $u_{g|_{\mu'}}u_h\not =0$, we also have $d(g|_{\mu'})=c(h)$. 

Now we remember that $\nu=\kappa(g\cdot \mu')$, and use the second option in \eqref{spanning_mult} to compute
\begin{align*}
\phi(c\alpha_{i\beta}(b))&=e^{-\beta(|\kappa|-|\lambda|)}\phi(cb)\notag\\
&=e^{-\beta(|\kappa|-|\lambda|)}\phi\big((s_\mu u_h s_\nu^*)(s_\kappa u_g s_\lambda^*)\big)\notag\\
&=e^{-\beta(|\kappa|-|\lambda|)}\phi\big(s_\mu u_{hg|_{g^{-1}\cdot(g\cdot\mu')}} s^*_{\lambda(g^{-1}\cdot(g\cdot\mu'))}\big)\notag\\
&=e^{-\beta(|\kappa|-|\lambda|)}\phi\big(s_\mu u_{hg|_{\mu'}} s^*_{\lambda\mu'}\big)\notag\\
&=e^{-\beta(|\kappa|-|\lambda|)}e^{-\beta|\mu|}\delta_{\mu,\lambda\mu'}\delta_{s(\mu),c(h)}
\delta_{s(\mu),d(g|_{\mu'})}\phi(u_h u_{g|_{\mu'}}).
\end{align*}
All the delta functions are $1$, and we deduce that 
\begin{align}
\phi(c\alpha_{i\beta}(b))&=e^{-\beta(|\kappa|-|\lambda|+|\mu|)}\phi(u_h u_{g|_{\mu'}})\notag\\
&=e^{-\beta(|\kappa|-|\lambda|+|\mu|)}\phi(u_{g|_{\mu'}}u_h),\label{phicb}
\end{align}
because $\phi\circ i_{C^*G)}$ is a trace. Since
\[
|\kappa|-|\lambda|+|\mu|=|\kappa|-|\lambda|+(|\lambda|+|\mu'|)=|\kappa|+|\mu'|,
\]
we deduce that \eqref{phibc}\,=\,\eqref{phicb}, and we are done.
\end{proof}

If $\tau$ is a trace on  a groupoid algebra $C^*(G)$, then
\[
\tau(i_g)=\tau(i_{c(g)} i_g i_{d(g)}) = \tau( i_g i_{d(g)}i_{c(g)}), \qquad g\in G,
\]
so  $\tau(i_g)\not=0$ implies $d(g) = c(g)$. So in particular we can make the following simple and important observation:

\begin{cor}
Suppose that $\phi$ is a KMS$_\beta$ state of $(\TT (G,E),\sigma)$. Then
\[
\phi(u_g)\not=0\Longrightarrow d(g)=c(g).
\]
\end{cor}

\section{KMS states above the critical temperature}\label{sec:largetemp}

\begin{thm}\label{Thm:KMS_beta_tau}
Suppose that $E$ is a finite graph with no sources and $G$ is a groupoid which acts self-similarly on $E$. Let $B$ be the vertex matrix of $E$, and let $\sigma$ be the dynamics on $\TT(G,E)$ such that
\[
\sigma_t(s_\mu u_h s_\nu^*)=e^{it(|\mu|-|\nu|)}s_\mu u_h s_\nu^*.
\]
Suppose that  $\beta > \ln\rho(B)$, where $\rho(B)$ is the spectral radius of $B$.
 \begin{enumerate}
 \item\label{Zconverges} For each normalised trace $\tau$ on $C^*(G)$, the series
\begin{equation}\label{treblesumcvges}
\sum_{j=0}^\infty \sum_{\mu \in E^j} e^{-\beta j} \tau(i_{s(\mu)})
\end{equation}
converges with sum $Z(\beta,\tau)$, say.
 \item\label{propertytaubeta} For each normalised trace $\tau$ on $C^*(G)$  there is a KMS$_\beta$-state $\psi_{\beta,\tau}$ of $(\TT(G,E),\sigma)$ such that
\begin{equation}\label{KMS_tau_formula}
\psi_{\beta,\tau}(s_\kappa u_g s_\lambda^*)=\delta_{\kappa,\lambda}Z(\beta,\tau)^{-1} \sum_{j=|\kappa|}^\infty e^{-\beta j} \sum_{\{\nu \in s(\kappa)E^{j-|\kappa|}: \, g \cdot \nu=\nu\}} \tau\big(i_{g|_\nu}\big) .
\end{equation}
\item\label{KMS_tau_affine} 
The map $\tau \mapsto \psi_{\beta,\tau}$ 
is an affine homeomorphism of the simplex of tracial states of $C^*(G)$ onto the simplex of KMS$_\beta$ states of $(\TT(G,E),\sigma)$.
\end{enumerate}
\end{thm}

Our strategy for part~\eqref{propertytaubeta} is based on a construction in \cite{ln} (see the proof of Theorem~2.1 in \cite{ln}). Given $\tau$, we consider the GNS represntation $\pi_\tau$ of $C^*(G)$ on $K_\tau$, say. We then consider the Fock module $\FF(M)=\bigoplus_{j=0}^\infty M^{\otimes n}$, which is also a bimodule over $C^*(G)$, and the  induced representation of $C^*(G)$ on $\FF(M)\otimes_{C^*(G)}K_\tau$. We then extend this induced representation to a representation of $\TT(G,E)$, and construct our KMS states as sums of vector states in this representation. As in previous applications \cite{lr,lrr,lrrw}, we find it easier to work in a concretely defined realisation of this induced representation. The key observation is that the tensor powers $X(E)^{\otimes j}$ of the graph bimodule $X(E)$ used in the construction of $M$ can be realised as functions on the path spaces $E^j$ (see \cite[Proposition~9.7]{R_book}).  

For $j\in \N$, we consider the graph bimodule $X(E^j)$ of the graph $(E^0,E^j,r,s)$, which is a Hilbert bimodule over $C(E^0)$. Thus we can form the Hilbert space $H_j:=X(E^j)\otimes_{C(E^0)} K_\tau$. For a path $\mu\in E^j$, we also write $\mu$ for the characteristic  function of the set $\{\mu\}$. Then every element $x$ of $C(E^j)$, which is the underlying space of $X(E^j)$, is the finite linear combination $\sum_{\mu\in X^j} x(\mu)\mu$. Because of the balancing in $H_j=X(E^j)\otimes_{C(E^0)} K_\tau$, we have 
\[
\mu\otimes k=(\mu\cdot s(\mu))\otimes k=\mu\otimes \pi_\tau(i_{s(\mu)})k \text{ for $\mu\in E^j$, $k\in K_\tau$.}
\]
The Hilbert space of our representation is then $H:=\bigoplus_{j=0}^\infty H_j$. We observe that
\[
H=\clsp\{\mu\otimes k:\mu\in E^*, k\in K_\tau\text{ and }\pi_\tau(i_{s(\mu)})k=k\}.
\]
We observe for future use that these spanning vectors satisfy
\begin{equation}\label{Fock_ip}
\big(\mu\otimes k\mid \nu\otimes l\big)=\big( \pi_\tau(\langle\nu,\mu\rangle)k\mid l\big)=\delta_{\mu,\nu}(\pi_\tau(i_{s(\mu)})k\mid l).
\end{equation}
Thus if we choose an orthonormal basis $\{e_{v,i}\}$ for each $\pi_\tau(i_{v})K_\tau$, then $\{\mu\otimes e_{s(\mu),i}\}$ is an orthonormal basis for $H$.

\begin{lemma}\label{representT}
Suppose that $E$ is a finite graph with no sources, and $G$ is a groupoid with unit space $E^0$ which acts self-similarly on $E$. Let $\tau$ be a tracial state of $C^*(G)$, and take $\pi_\tau$ and $H$ as above. Then there is 
a Toeplitz-Cuntz-Krieger family $(P,S)$ on $H$ 
such that 
\begin{equation}\label{Fock_TSP}
P_{v}(\mu \otimes k)=\delta_{v,r(\mu)}\mu\otimes k \qquad \text{and} \qquad S_e(\mu\otimes k)=\delta_{s(e),r(\mu)}e\mu\otimes k;
\end{equation}
and there is a unitary representation $U: G\to U(H)$ 
such that for each $g\in G$
\begin{equation}\label{Fock_T}
U_g(\mu \otimes k)= \begin{cases}(g \cdot \mu )\otimes \pi_\tau (i_{g|_\mu})k & \text {if } r(\mu) = d(g)\\
0 &\text{ otherwise.}
\end{cases}
\end{equation}
Moreover, the triple $(U, P, S)$ satisfies the relations \textnormal{(1--4)} of Proposition~\ref{Toeplitz_repn}.
\end{lemma}

\begin{proof}
Looking at their effects on the orthonormal basis $\{\mu\otimes e_{s(\mu),i}\}$ shows that there are  projections $P_v$ and a partial isometries $U_e$ satisfying \eqref{Fock_TSP}, and that they form a Toeplitz-Cuntz-Krieger family. 

Regarding now the vertex $v$ as a unit of the groupoid $G$, we define $U_v:=P_v$. We now need to define $U_g$ for more general $g\in G$. Each vector $h$ in the algebraic linear span
\[
H':=\sp\{\mu\otimes k:\mu\in E^*, k\in K_\tau\text{ and }\pi_\tau(i_{s(\mu)})k=k\}
\]
has a unique sum decomposition $h=\sum_{\mu\in F}\mu\otimes k_\mu$ for some finite subset $F$ of $E^*$. So we can define $U_g$ as a function on $P_{d(g)}H'$ by
\[
U_g\Big(\sum_{\mu\in F} \mu \otimes k_\mu\Big) := \sum_{\mu\in F} (g \cdot \mu)\otimes \pi_\tau(i_{g|_{\mu}})k_\mu. 
\]
We aim to prove that $U_g$ is isometric on $P_{d(g)}H'$, and hence extends to an isometry on $\overline{P_{d(g)}H'}=P_{d(g)}H$. We have
 \begin{equation*}\label{compnormT}
\Big\|U_g\Big(\sum_{\mu\in F} \mu \otimes k_\mu\Big)\Big\|^2 
=\Big(\sum_{\mu\in F} (g \cdot \mu) \otimes \pi_\tau(i_{g|_{\mu}})k_\mu \Big| \sum_{\nu\in F} (g \cdot \nu)\otimes \pi_\tau(i_{g|_{\nu}})k_\nu  \Big) 
\end{equation*}
Since $\langle \mu,\nu\rangle=\delta_{\mu,\nu}i_{s(\mu)}$, the $\mu,\nu$ summand on the right-hand side vanishes unless $g\cdot \mu=g\cdot \nu$, and hence unless $\mu=\nu$. Thus
\[
\Big\|U_g\Big(\sum_{\mu\in F} \mu \otimes k_\mu\Big)\Big\|^2
= \sum_{\mu\in F} \big(\pi_\tau( i_{g|_{\mu}})k_\mu\,\big|\, \pi_\tau(i_{g|_{\mu}} )k_\mu  \big).
\]
Since $\pi_\tau(i_{g|_{\mu}})$ is an isomorphism of $d(g|_\mu)E^*=s(\mu)E^*$ onto $c(g|_\mu)E^*$, it follows that
\begin{align*}
\Big\|U_g\Big(\sum_{\mu\in F} \mu \otimes k_\mu\Big)\Big\|^2=\sum_{\mu\in F}(k_\mu\,|\,k_\mu)=\Big\|\sum_{\mu\in F}\mu\otimes k_\mu\Big\|^2.
\end{align*}
Thus the map $U_g$ is well-defined and isometric on a dense subset of $d(g) H = U_{d(g)}H $, and hence extends to an isometry of $d(g) H$ into $U_{d(g)}H $. Since $\mu\mapsto g\cdot \mu$ is onto $c(g)E^*$ and $\pi_\tau(i_{g|_\mu})$ is onto $\pi_\tau(i_{c(g|_\mu)})K_\tau=\pi_\tau(i_{s(g\cdot\mu)})K_\tau$ 
(because $\pi_\tau$ is a unitary representation), every elementary tensor in $P_{c(g)}H'$ belongs to the range of $U_g$. Thus $U_g$ is a unitary isomorphism of $P_{d(g)}H$ onto $P_{c(g)}H$.

Now take $g,h\in G$. If $d(g) \neq c(h)$, then the final subspace $P_{c(h)} H$ of $U_h$ is orthogonal to the initial subspace $P_{d(g)} H$ of $U_g$, and hence $U_gU_h=0$.
So we assume that $d(g) = c(h)$. Then the range projection of $U_h$ is the initial projection of $U_g$, and thus $U_gU_h$ is a partial isometry whose initial and final projections are the same as those of $U_{gh}$.
To prove that $U_gU_h = U_{gh}$, we compute on an elementary tensor  $\mu \otimes k $ with $r(\mu) = d(g)$:
\begin{align*}
U_gU_h(\mu \otimes k)&=U_g(h \cdot \mu \otimes \pi (i_{h|_\mu})k)=g\cdot (h \cdot \mu) \otimes \pi (i_{g|_{h \cdot \mu}}) \pi (i_{h|_\mu})k \\
&=gh \cdot \mu \otimes \pi (i_{gh|_\mu})= U_{gh}(\mu \otimes k).
\end{align*}
Thus $U$ is a unitary representation of the groupoid $G$ on $H$.
 
Next we take $g\in G$ and $e\in E^1$, and verify relation \eqref{TR_SSA} in Proposition~\ref{Toeplitz_repn}. If $d(g) \neq r(e)$, then  $S_eH =0=P_{r(e)}H$ is orthogonal to the initial space $P_{d(g)}H$ of $U_g$, and hence $U_gS_e=0$.
So suppose that $d(g) = r(e) $, and take an elementary tensor  $\mu \otimes k \in H$ with $s(e) = r(\mu)$. Then we compute
\begin{align*}
U_gS_e(\mu \otimes k) &=U_g(e\mu \otimes k)=(g\cdot e\mu) \otimes \pi (i_{g|_{e\mu}})k \\
&=(g\cdot e)(g|_{e} \cdot \mu) \otimes \pi (i_{g|_{e\mu}})k =S_{g\cdot e}(g|_{e} \cdot \mu \otimes \pi (i_{g|_{e}|_{\mu}})k)\\
& = S_{g\cdot e}U_{g|_e}(\mu \otimes k).
\end{align*}
Finally, we check that $U_g P_v = P_{g \cdot v}U_g$. For $\mu \otimes k \in H$, we have
\begin{align*}
U_g P_v(\mu \otimes k)&=U_g (\delta_{v,r(\mu)} \mu \otimes k)=\delta_{v,r(\mu)} g \cdot \mu \otimes \pi (i_{g|_{\mu}}) k \\
&=\delta_{g \cdot v,r(g \cdot \mu)} g \cdot \mu \otimes \pi (i_{g|_{\mu}}) k = P_{g \cdot v} (g \cdot \mu \otimes \pi (i_{g|_{\mu}}) k)\\
&= P_{g \cdot v} U_g (\mu \otimes k).
\end{align*}
Thus $(U,P,S)$ satisfies all the relations in Proposition~\ref{Toeplitz_repn}, as required.
\end{proof}

\begin{proof}[Proof of Theorem~\ref{Thm:KMS_beta_tau}\eqref{Zconverges}]
Suppose that $\tau$ is a normalised trace on $C^*(G)$, and consider  the corresponding GNS representation $(K_\tau, \pi_\tau, \xi_\tau)$ of $C^*(G)$. Construct the Hilbert space $H$ as above. Let $\pi_{U,S,P}:\TT(G,E)\to B(H)$ be the representation arising from the family $(U,P,S)$ of Lemma~\ref{representT} by the universal property from \proref{Toeplitz_repn}.
 
Suppose that $\beta > \ln \rho(B)$.  Then $\sum_{j=0}^{\infty}\sum_{\mu\in E^jv}e^{-\beta j}$ 
converges for each $v\in E^0$ by \cite[Theorem~3.1(a)]{hlrs}. Since $0\leq \tau(i_{s(\mu)})\leq 1$ for every $\mu\in E^*$, it follows that 
\begin{equation}\label{treblesumcvges_a}
\sum_{j=0}^\infty \sum_{\mu \in E^j} e^{-\beta j} \tau(i_{s(\mu)})=\sum_{v\in E^0}\sum_{j=0}^{\infty}\sum_{\mu\in E^jv}e^{-\beta j}\tau(i_{s(\mu)})
\end{equation}
converges. 
\end{proof} 

\begin{proof}[Proof of Theorem \ref{Thm:KMS_beta_tau}\eqref{propertytaubeta}] 
We begin by defining $\psi_{\beta,\tau}$ spatially using the representation on $H$. As in the previous paragraph, for every $a\in \TT(G,E)$, the series
\begin{equation}\label{psi_a}
\psi_{\beta,\tau}(a):={Z(\beta,\tau)\inv} \sum_{j=0}^\infty e^{-\beta j} \sum_{\mu \in E^j} \big( \pi_{U,P,S}(a)(\mu \otimes \xi_\tau) \mid \mu \otimes \xi_\tau\big)
\end{equation}
converges absolutely with sum $\psi_{\beta,\tau}(a)\in \C$. By standard properties of convergent series, the assignment $\psi_{\beta,\tau}:a\mapsto \psi_{\beta,\tau}(a)$ is a positive linear functional on $\TT(G,E)$. In fact, we have 
\begin{align*}
\psi_{\beta,\tau}(1)&=Z(\beta,\tau)\inv \sum_{j=0}^\infty e^{-\beta j} \sum_{\mu \in E^j} (\mu \otimes \xi_\tau \mid \mu \otimes \xi_\tau) \\
&=Z(\beta,\tau)\inv \sum_{j=0}^\infty e^{-\beta j}\sum_{\mu \in E^j}\tau(i_{s(\mu)})\\
&=1, 
\end{align*}
and hence $\psi_{\beta,\tau}$ is a state.

To  prove  that $\psi_{\beta,\tau}$ satisfies \eqref{KMS_tau_formula}, we fix a spanning element $s_\kappa u_g s_\lambda^*$. Then
\begin{equation}\label{psi_a_spanning}
\psi_{\beta,\tau}(s_\kappa u_g s_\lambda^*)=Z(\beta,\tau)\inv \sum_{j=0}^\infty e^{-\beta j} \sum_{\mu \in E^j} \big( S_\kappa U_g S_\lambda^*(\mu \otimes \xi_\tau) \,\big|\, \mu \otimes \xi_\tau\big).
\end{equation}
Suppose there exists $\mu$ such that
\[
\big( S_{\kappa} U_g S_\lambda^*(\mu \otimes \xi_\tau) \mid \mu \otimes \xi_\tau\big)=\big( U_g S_\lambda^*(\mu \otimes \xi_\tau) \mid S_{\kappa}^* (\mu \otimes \xi_\tau)\big)\not=0.
\]
Since $S_\kappa^*(\mu\otimes\xi_\tau)\not=0$, there exists $\mu'$ such that $\mu=\kappa \mu'$, and similarly there exists $\mu''$ such that $\mu=\lambda \mu''$. Then
\begin{align}\label{ip_eqn}
0\neq \big( U_g S_\lambda^*(\mu \otimes \xi_\tau) \mid S_{\kappa}^* (\mu \otimes \xi_\tau)\big)
&=\big( U_g (\mu'' \otimes \xi_\tau) \mid \mu' \otimes \xi_\tau\big)\\
&=\big((g \cdot \mu'' )\otimes \pi_\tau(i_{g|_{\mu''}})\xi_\tau \mid \mu' \otimes \xi_\tau\big),\notag
\end{align}
which implies that $g \cdot \mu''=\mu'$ and $|\mu''| = |g \cdot \mu''| = |\mu'|$. Now we have $\kappa \mu'=\mu=\lambda \mu''$. From all this we deduce: first, that $\kappa=\lambda$, which gives the $\delta_{\kappa,\lambda}$ in \eqref{KMS_tau_formula}; second, that  only the terms with $\mu=\kappa\mu'$ contribute non-zero terms to the sum on the right of \eqref{psi_a_spanning}; third, that only the terms with $g\cdot\mu'=\mu'$ survive. For these terms, we have
\begin{align*}
\big( S_{\kappa} U_g S_{\lambda}^*(\mu \otimes \xi_\tau) \mid \mu \otimes \xi_\tau\big)
&=\big(g\cdot\mu' \otimes \pi_\tau(i_{g|_{\mu'}})\xi_\tau \mid \mu' \otimes \xi_\tau\big)\\
&=\big(\pi_\tau(i_{s(\mu)}i_{g|_{\mu'}})\xi_\tau) \mid \mu' \otimes \xi_\tau\big)\\
&=\big(\pi_\tau(i_{g|_{\mu'}})\xi_\tau) \mid \xi_\tau\big). 
\end{align*}
Now summing over $\mu'\in s(\kappa)E^*$ satisfying $g\cdot\mu'=\mu'$, and writing $\nu$ for $\mu'$, gives \eqref{KMS_tau_formula}.

To show that $\psi_{\beta,\tau}$ is a  KMS$_\beta$ state, we verify the conditions in Proposition \ref{KMS_algebraic}. To see that $\psi_{\beta,\tau} \circ i_{C^*(G)}$ is a trace on $C^*(G)$, consider  two generators $i_g$ and $i_h$  in $C^*(G)$. Then \eqref{KMS_tau_formula} gives
\begin{align}
\psi_{\beta,\tau} \circ i_{C^*(G)}(i_gi_h)
&=\psi_{\beta,\tau}(u_g  u_h)= \delta_{d(g),c(h)}\psi_{\beta,\tau}(u_{gh}) \notag \\
\label{Toeplitz_trace1}
&= \delta_{d(g),c(h)}\delta_{c(g), d(h)}  \, Z(\beta,\tau)\inv \sum_{j=0}^\infty e^{-\beta j} \sum_{\{\nu \in d(h)E^{j} : gh \cdot \nu = \nu\}} \tau\big( i_{(gh)|_{\nu}}\big).
\end{align}
On the other hand, we have
\begin{align}
\psi_{\beta,\tau} \circ i_{C^*(G)}(i_hi_g)
&=\psi_{\beta,\tau}(u_h  u_g)= \delta_{d(h), c(g)}\psi_{\beta,\tau}(u_{hg}) \notag \\
\label{Toeplitz_trace2}
&= \delta_{d(h), c(g)}\delta_{c(h), d(g)}  \, Z(\beta,\tau)\inv \sum_{j=0}^\infty e^{-\beta j} \sum_{\{\mu \in d(g)E^{j} : hg \cdot \mu=\mu\}} \tau\big( i_{(hg)|_{\mu}}\big).
\end{align}

We need to prove that $\eqref{Toeplitz_trace1}=\eqref{Toeplitz_trace2}$. Both sides vanish unless $d(h)=c(g)$ and $c(h)=d(g)$, so we suppose that both are true. We claim that the map $\theta:\mu \mapsto g \cdot \mu$ is a bijection of the index set $\{\mu \in d(g)E^j : (hg) \cdot \mu=\mu\}$ in the sum \eqref{Toeplitz_trace2} onto the index set  $\{\nu \in d(h)E^j : (gh) \cdot \nu =\nu\}$ in \eqref{Toeplitz_trace1}. Suppose $\mu \in d(g)E^j$ satisfies  $ hg ( \mu)=\mu$. Then $r(\mu)=d(g)$ gives $r(g\cdot \mu)=c(g)=d(h)$, so $g\cdot \mu \in d(h)E^j$. Now $(gh) \cdot(g \cdot \mu)=g \cdot (hg \cdot \mu)=g\cdot \mu$, and hence 
$g\cdot \mu$ belongs to the set $\{\nu \in d(h)E^j : gh\cdot \nu=\nu\}$. 
The map $\nu \mapsto h\cdot \nu$ is an inverse for $\theta$, and the claim follows. 

Now we show that the bijection $\theta$ matches up the summands as well as their labels. Suppose that $\nu\in d(h)E^j$ satisfies $gh \cdot \nu =\nu$. Then
\begin{align*}
\tau\big( i_{(gh)|_{\nu}}\big)&=\tau\big(i_{g|_{h \cdot \nu}}i_{h|_{\nu}}\big)\quad \text{by \proref{SSA:extension to paths}} \\
&=\tau\big(i_{h|_{\nu}} i_{g|_{h \cdot \nu}}\big) \quad \text{ since $\tau$ is a trace} \\
&=\tau\big( i_{h|_{gh \cdot \nu}} i_{g|_{h \cdot \nu}}\big) \quad \text{ because $gh \cdot \nu = \nu$}\\
&=\tau\big( i_{(hg)|_{h \cdot \nu}}\big).
\end{align*}
Thus the sums in  \eqref{Toeplitz_trace1} and  \eqref{Toeplitz_trace2} coincide, and $\psi_{\beta,\tau} \circ i_{C^*(G)}$ is a trace on $C^*(G)$.

We now prove that $\psi_{\beta,\tau}$ satisfies \eqref{char:spanning}. Let $s_{\kappa} u_g s_\lambda^* \in \TT(G,E)$ and assume that $s(\kappa)=g\cdot s(\lambda)$, for otherwise the product vanishes. On one hand, \eqref{KMS_tau_formula} gives
\begin{align*}
\psi_{\beta,\tau}(s_{\kappa} u_g s_\lambda^*)&=\delta_{\kappa,\lambda}Z(\beta,\tau)\inv \sum_{j=|\kappa|}^\infty e^{-\beta j} \sum_{\{\nu \in s(\kappa)E^{j-|\kappa|} : g \cdot \nu=\nu\}} \tau\big( i_{g|_{\nu}}\big) \\
&=\delta_{\kappa,\lambda}Z(\beta,\tau)\inv \sum_{k=0}^\infty e^{-\beta |\kappa|} e^{-\beta k} \sum_{\{\nu \in s(\kappa)E^{k} : g \cdot \nu=\nu\}} \tau\big( i_{g|_{\nu}}\big).
\end{align*}
On the other hand, since $\psi_{\beta,\tau} \circ i_{C^*(G)}$ is a trace, the right hand side of \eqref{char:spanning} gives
\begin{align*}
\delta_{\kappa,\lambda} e^{-\beta |\kappa|} \psi_{\beta,\tau}(p_{s(\kappa)}u_g)
&=\delta_{\kappa,\lambda} e^{-\beta |\kappa|} \psi_{\beta,\tau}(p_{s(\kappa)}u_g p_{s(\kappa)})\quad \\
&=\delta_{\kappa,\lambda} e^{-\beta |\kappa|} Z(\beta,\tau)\inv  \sum_{k=0}^\infty e^{-\beta k} \sum_{\{\nu \in s(\kappa)E^{k} : g \cdot \nu=\nu\}} \tau\big(i_{s(\nu)} i_{g|_{\nu}}\big) \\
&=\delta_{\kappa,\lambda} Z(\beta,\tau)\inv
\sum_{k=0}^\infty e^{-\beta |\kappa|} e^{-\beta k} \sum_{\{\nu \in s(\kappa)E^{k} : g \cdot \nu=\nu\}} \tau\big(i_{s(\nu)} i_{g|_{\nu}}\big). 
\end{align*}
Since  $g\cdot\nu=\nu$ implies that $i_{s(\nu)} i_{g|_{\nu}} =  i_{g|_{\nu}}$ , we deduce that $\psi_{\beta,\tau}$ satisfies  \eqref{char:spanning}. Now Proposition~\ref{KMS_algebraic} implies that $\psi_{\beta,\tau}$ is a KMS$_\beta$-state. This finishes the proof  of Theorem \ref{Thm:KMS_beta_tau}\eqref{propertytaubeta}.
\end{proof}

The key step in our proof of part \eqref{KMS_tau_affine} of Theorem \ref{Thm:KMS_beta_tau}
is proving that every KMS state has the form $\psi_{\beta,\tau}$ for exactly one trace $\tau$ on $C^*(G)$. The proof of uniqueness follows previous ones (for example, in \cite[\S10]{lr} and \cite[\S6]{lrrw}) in which a KMS$_\beta$ state is reconstructed from its conditioning by a projection in  $\TT(G,E)$.

\begin{lemma}\label{Lem:phi_P}
Suppose that $\beta >\ln\rho(B)$ and $\phi$ is a KMS$_\beta$ state of $\TT(G,E)$. Define $m^\phi\in[0,\infty)^{E^0}$ by
$m^\phi_v :=\phi(p_v)$ in $[0,\infty)^{E^0}$, and  a projection $P$ in $\TT(G,E)$ by
 \[
P:= 1 - \sum_{e\in E^1} s_es_e^* = \sum_{v\in E^0} \Big(p_v - \sum_{e\in vE^1}  s_es_e^*\Big).
\]
Then 
\begin{equation}\label{P=sum}
\phi(P)= \sum_{v\in E^0}\big((1 - e^{-\beta}B)m^\phi\big)_v    > 0,
\end{equation}
and $\phi_P:a \mapsto \phi(P)\inv \phi(PaP)$ is a state 
on $\TT(G,E)$ such that $\phi_P\circ i_{C^*(G)}$ is a trace on $C^*(G)$.
\end{lemma}
\begin{proof}
A computation like that of \cite[Equation (2.5)]{hlrs} yields
\begin{align*} \phi(P) &=  \sum_{v\in E^0} \Big(\phi(p_v) -\sum_{e\in vE^1}  \phi(s_es_e^*)\Big)\\ 
&=  \sum_{v\in E^0} \Big(\phi(p_v) -\sum_{e\in vE^1}  e^{-\beta}\phi(s_e^*s_e)\Big) \\
&= \sum_{v\in E^0} \Big(\phi(p_v) - \sum_{w\in E^0}  |vE^1 w| e^{-\beta}\phi(p_w)\Big) \\
&=  \sum_{v\in E^0}\big((1 - e^{-\beta}B)m^\phi\big)_v  .
\end{align*}
Since $1= \phi(1)=\sum_{v}\phi(p_v)$, the vector $ m^\phi $ is nonzero,
and the matrix $1 - e^{-\beta}B$ is invertible because $e^\beta >\rho(B)$. 
Thus $\phi(P) > 0$, and $\phi_P$ is a state.

With a view to proving that  $\phi_P$ is a trace on the copy of $C^*(G)$, we first claim that if $d(g) = r(e)$, then  $u_g s_e s_e^* =  s_{g \cdot e} s_{g \cdot e}^* u_g$. Theorem~\ref{Toeplitz_repn} imples that $u_gs_e=s_{g\cdot e}u_{g|_e}$. Now we use the properties of Proposition~\ref{properties of faithful self-similar groupoid action} to compute
\begin{align*}
u_{g|_e}s_e^*&=(s_eu_{g|e}^*)^*=(s_eu|_{g|_e^{-1}})^*=(s_eu_{g^{-1}|_{g\cdot e}})^*\\
&=(s_{g^{-1}\cdot(g\cdot e)}u_{g^{-1}|_{g\cdot e}})^*=(s_{g^{-1}\cdot(g\cdot e)}u_{g^{-1}|_{g\cdot e}})^*\\
&=(u_{g^{-1}}s_{g\cdot e})^*=s_{g\cdot e}^*u_g.
\end{align*}
Thus $u_gs_es_e^*=s_{g\cdot e}u_{g|_e}s_e^*=s_{g\cdot e}s_{g\cdot e}^*u_g$, as claimed.
Now for $g\in G$ we have 
\begin{align*}
u_g P&=u_g\Big(1-\sum_{e\in E^1}s_es_e^*\Big)=u_g - \sum_{e \in g(g)E^1} u_gs_e s_{e}^*\\
&=u_g - \sum_{e \in d(g)E^1} s_{g\cdot e}s_{g\cdot e}^*u_g.
\end{align*}
Since $e\mapsto g\cdot e$ is a bijection of $d(g)E^1$ onto $c(g)E^1$, we have
\[
u_g P=\Big(1-\sum_{f\in c(g)E^1}s_fs_f^*\Big)u_g=\Big(1-\sum_{f\in E^1}s_fs_f^*\Big)u_g=Pu_g.
\]
Suppose that $g,h\in G$ satisfy $d(g)=c(h)$. Then $u_g$ and $u_gP$ are fixed under the action $\sigma$, so the KMS condition implies that
\[
\phi(Pu_gu_hP)=\phi(u_gP^2u_h)=\phi(Pu_h\sigma_{i\beta}(u_{g}P))=\phi(Pu_hu_gP).
\]
Thus $a\mapsto \phi(PaP)$ is a trace on $C^*(G)$, and so is the multiple $\phi_P$.
\end{proof}

The next lemma is the main technical ingredient in our reconstruction formula.

\begin{lemma}\label{reconstr_lemma}
Suppose that $\beta>\ln\rho(B)$ and $\phi$ is a KMS$_\beta$ state on $(\TT(G,E),\sigma)$. Then for $a \in \TT(G,E)$, we have 
\begin{align}\label{KMS_reconstr_formula}
\phi(a)&=\phi(P) \sum_{j=0}^\infty \sum_{\mu \in E^j} e^{-\beta j} \phi_P(s_\mu^* a s_\mu)\\
&=\Big(\sum_{v\in E^0}\big((1-e^{-\beta}B)m^\phi\big)_v \Big)\sum_{j=0}^\infty \sum_{\mu \in E^j} e^{-\beta j} \phi_P(s_\mu^* a s_\mu).\notag
\end{align}
\end{lemma}

\begin{proof}
A computation similar to the one in the proof of \cite[Lemma 6.3]{lrrw} shows that for each $n \in \N$
\[ 
p_n := \sum_{j=0}^n \sum_{\mu \in E^j} s_\mu P s_\mu^*
\]
is a projection in $\TT(G,E)$. 

We next show that $\phi(p_n) \to 1$ as $n \to \infty$. To do this, we compute:
\begin{align*}
\phi(p_n)&=\sum_{j=0}^n \sum_{\mu \in E^j} \phi(s_\mu P s_\mu^*)
= \sum_{j=0}^n \sum_{\mu \in E^j} e^{-\beta j}\phi(P s_\mu^*s_\mu) \notag \\
&= \sum_{j=0}^n \sum_{\mu \in E^j} e^{-\beta j}\Big(\phi(p_{s(\mu)}) -\sum_{r(e)=s(\mu)} \phi(s_e s_e^*)\Big)\\
&= \sum_{j=0}^n \sum_{\mu \in E^j} e^{-\beta j}\Big(m_{s(\mu)}^\phi -\sum_{r(e)=s(\mu)} e^{-\beta} m_{s(e)}^\phi \Big).
\end{align*}
Now we rewrite this in terms of the vertex matrix $B$, and continue the computation, using some tricks from \cite[\S3]{hlrs}:
\begin{align*}
\phi(p_n)&= \sum_{j=0}^n \sum_{\mu \in E^j} e^{-\beta j}\Big(m_{s(\mu)}^\phi -\sum_{w \in E^0} e^{-\beta} B(s(\mu),w) m_{w}^\phi \Big)\\
&= \sum_{j=0}^n \sum_{\mu \in E^j} e^{-\beta j}\big((1-e^{-\beta} B)m^\phi \big)_{s(\mu)} \notag\\
&= \sum_{j=0}^n \sum_{v \in E^0} \sum_{w \in E^0} e^{-\beta j} B^j(v,w) \big((1-e^{-\beta} B)m^\phi \big)_{w}\\
&= \sum_{j=0}^n \sum_{v \in E^0} e^{-\beta j} \big(B^j(1-e^{-\beta} B)m^\phi \big)_{v} \notag\\
&= \sum_{v \in E^0}\Big(\sum_{j=0}^n e^{-\beta j} B^j(1-e^{-\beta} B)m^\phi \Big)_{v}. 
\end{align*}
Since $\beta>\ln\rho(B)$, we have $e^{\beta}>\rho(B)$, and the series $\sum_je^{-\beta j}B^j$ converges with sum $(1-e^{-\beta}B)^{-1}$. Thus as $n\to \infty$ 
\[
\phi(p_n)\to \sum_{v \in E^0}\big((1-e^{\beta}B)^{-1} (1-e^{-\beta} B)m^\phi \big)_{v} = \sum_{v \in E^0} m_v^\phi =1.
\]

Since $\phi(p_n)\to 1$, we have  $\phi(p_n a p_n) \to \phi(a)$ as $n\to \infty$ (see, for example, \cite[Lemma 7.3]{lrr}). For $\mu,\nu\in E^j$, the elements $s_\mu Ps_\mu^*$ and $s_\nu Ps_\nu^*$  belong to the fixed-point algebra for $\sigma$, and hence since $\phi$ is a KMS state, we have
\[
\phi\big((s_\mu Ps_\mu^*)a(s_\nu Ps_\nu^*)\big)=0\quad\text{for $a\in \TT(G,E)$ and $\mu\not=\nu$.}
\]
We now use the KMS condition again to compute
\begin{align*}
\phi(a)&=\sum_{j=0}^\infty \sum_{\mu \in E^j} \phi(s_\mu P s_\mu^* a s_\mu P s_\mu^*)\\
&=\sum_{j=0}^\infty  \sum_{\mu \in E^j} e^{-\beta j} \phi(P s_\mu^* a s_\mu Ps_\mu^*s_\mu ) \\
&=\sum_{j=0}^\infty  \sum_{\mu \in E^j} e^{-\beta j} \phi(P s_\mu^* a s_\mu s_\mu^*s_\mu P)\\
&=\sum_{j=0}^\infty  \sum_{\mu \in E^j} e^{-\beta j} \phi(P) \phi_P(s_\mu^* a s_\mu).
\end{align*}
The second formula follows from this and \eqref{P=sum}.
\end{proof}

\begin{proof}[Proof of  \thmref{Thm:KMS_beta_tau}\eqref{KMS_tau_affine}]
The proof follows that of \cite[Theorem 6.1]{lrrw}. An application of the monotone convergence theorem shows that $\tau \mapsto \psi_{\beta,\tau}$ is affine and weak$^*$ continuous. Since the set of tracial states and the set of KMS$_\beta$ states are both weak$^*$ compact, it suffices to show that the map $\tau \mapsto \psi_{\beta,\tau}$ is one-to-one and onto.

The equations \eqref{Fock_TSP} in \lemref{representT} imply that the  representation $\pi=\pi_{U,P,S}$ of $\TT(G,E)$
satisfies $\pi(P)(\mu \otimes \xi_\tau) = 0 $ unless $\mu = v \in E^0$, 
in which case $\pi(P) v\otimes \xi_\tau = v \otimes \xi_\tau$.
We now take $b\in C^*(G)$, and set $a = Pu(b)P $ in \eqref{psi_a}. Then all the terms with $j>0$ vanish, yielding
\begin{equation}\label{taufrompsi}
\psi_{\beta,\tau}(Pu(b)P) = Z(\beta, \tau)\inv \sum_v \tau(p_v b p_v) = Z(\beta, \tau)\inv \tau(b).
\end{equation}
Taking $b = 1$ shows that $\psi_{\beta,\tau}(P) = Z(\beta, \tau)\inv$, and then dividing \eqref{taufrompsi} through by $\psi_{\beta,\tau}(P)$ gives 
\begin{equation}\label{condpsi}
(\psi_{\beta,\tau})_P(u(b)) = \psi_{\beta,\tau}(P)\inv\psi_{\beta,\tau}(Pu(b)P) = \tau(b).
\end{equation}
Thus $\tau = (\psi_{\beta,\tau})_P \circ u$, and  the map $\tau \mapsto \psi_{\beta,\tau}$ is one-to-one. 

It remains to show that the map $\tau \mapsto \psi_{\beta,\tau}$ is onto. Suppose that $\phi$ is a KMS$_\beta$-state on $\TT(G,E)$. By \lemref{Lem:phi_P}, $\tau:=\phi_P\circ u$ is a tracial state on $C^*(G)$,
and by the above argument, $\phi_P \circ u= \tau = (\psi_{\beta,\tau})_P u$. Thus \lemref{reconstr_lemma} implies that $\phi(u_g) = \psi_{\beta,\tau}(u_g)$ for every $g\in G$, and in \proref{KMS_algebraic}\eqref{states_iff} we have $\phi =\psi_{\beta,\tau}$. 
\end{proof}

\section{Traces on groupoid algebras}\label{sec:traces}

Theorem~\ref{Thm:KMS_beta_tau} tells us that the KMS states at large inverse temperatures are parametrised by traces on the groupoid algebra $C^*(G)$. In general finding these traces is a hard problem, even for group algebras --- indeed, one unexpected outcome of our previous analysis of KMS states in  was the discovery of new tracial states on group algebras \cite[Corollary~7.5]{lrrw}. The exception is when $G$ is an abelian group, in which case $C^*(G)$ is the algebra $C(\widehat G)$ of continuous functions on the compact dual group $\widehat G$, and the tracial states are given by probability measures on $\widehat G$. Fortunately, for the groupoids of interest to us, we can often reduce to this case.

The existence of the isomorphism in the following result is a special case of Theorem~3.1 of \cite{mrw}, but it will be helpful to have a concrete description of the isomorphism.  As a local convention to avoid complicated subscripts, if $u$ is a unitary representation of a group or groupoid, we write $u(g)$ rather than $u_g$.

\begin{lemma}\label{mrwiso}
Suppose that $G$ is a groupoid with finite unit space $G^0$ and that $G$ acts transitively on $G^0$. Fix $x\in G^0$, let $G_x:=xGx$ be the isotropy group, and write $i_{G_x}$ for the canonical unitary representation of $G_x$ in its group algebra $C^*(G_x)$. For each $y\in G^0$ choose $k_y\in G$ such that $d(k_y)=x$, $c(k_y)=y$ and $k_x=x$. Then there is an isomorphism $\psi$ of $M_{G^0}(C^*(G_x))$ onto $C^*(G)$ such that
\begin{equation}
\psi\big(\,(a_{yz})\,\big)=\sum_{y,z\in G^0} i(k_y)\pi_i(a_{yz})i(k_z)^{-1}.
\end{equation}\label{form1}
\!The inverse $\psi^{-1}$ satisfies
\begin{equation}\label{form-1}
\psi^{-1}(i(g))_{yz}=\begin{cases}u\big(k_{c(g)}^{-1}gk_{d(g)}\big)&\text{if $y=c(g)$ and $z=d(g)$}\\
0&\text{otherwise}.
\end{cases}
\end{equation}
\end{lemma}

\begin{proof}
We verify that $\{i(k_yk_z^{-1}):y,z\in G^0\}$ is a set of non-zero matrix units in $C^*(G)$, and hence there is an injective homomorphism $\pi:M_{G^0}(\C)\to C^*(G)$ which takes the usual matrix units $e_{yz}$ to $i(k_yk_z^{-1})$. 

The restriction $i|_{G_x}$ is a unitary representation (in the usual group-theoretic sense) of $G_x$ in the corner $i(x)C^*(G)i(x)$, and hence there is a unital homomorphism $\rho:C^*(G_x)\to i(x)C^*(G)i(x)$ such that $\rho\circ u=i|_{G_x}$. To see that $\rho$ is injective, we show that every unitary representation $U:G_x\to U(H)$ is the compression to $G_x$ of a unitary representation $W$ of the groupoid $G$. Indeed, we take $H_x=H$, and for $y\in G^0$ let $H_y$ be copies of $H$, with $V_y:H\to H_y$ the identity maps. Then for $g\in G$, 
\[
W(g):=V_{c(g)}U(k_{c(g)}^{-1}gk_{d(g)})V_{d(g)}^*:H_{d(g)}\to H_{c(g)}
\]
is a unitary isomorphism. For $g,h\in G$ with $d(g)=c(h)$ we have $V_{d(g)}=V_{c(h)}$ and $k_{d(g)}=k_{c(h)}$. Thus
\begin{align*}
W(g)W(h)&=\big(V_{c(g)}U(k_{c(g)}^{-1}gk_{d(g)})V_{d(g)}^*\big)\big(V_{c(h)}U(k_{c(h)}^{-1}hk_{d(h)})V_{d(h)}^*\\
&=V_{c(g)}U(k_{c(g)}^{-1}ghk_{d(h)})V_{d(h)}.
\end{align*}
Now $c(gh)=c(g)$ and $d(gh)=d(h)$. Thus
\[
W(g)W(h)=V_{c(gh)}U(k_{c(gh)}^{-1}ghk_{d(gh)})V_{d(gh)}=W(gh),
\]
and $W$ is a unitary representation of the groupoid $G$ on $\bigoplus_{y\in E^0}H_y$, with $W(x)$ the projection on the Hilbert space $H$ of $U$. Since $k_x=x$ and $V_x$ is the identity, we have $W(x)W(g)W(x)^*=U(g)$ for $g\in G_x$. Thus the homomorphism $\rho$ is isometric for the enveloping $C^*$-norms on $C_c(G_x)$ and $C_c(G)$, and in particular is injective. It follows that $\theta:u(g)\mapsto \sum_{y\in G^0}i(k_yg k_y^{-1})$ is also injective from $C^*(G_x)$ to $C^*(G)$. 

Now we show that $\theta$ and $\pi$ have commuting ranges. Indeed, for $g\in G_x$ and a matrix unit $e_{zw}\in M_n(\C)$,  
\[
\theta(i_{G_x}(g))\pi(e_{zw})=\sum_{y\in G^0}i(k_yg k_y^{-1})i(k_zg k_w^{-1})
\]
vanishes unless $y=z$, and then equals $i(k_zg k_w^{-1})$; a similar computation on the other side shows that only the $y=w$ summand in $\pi(e_{zw})\theta(i_{G_x}(g))$ survives, and is again equal to $i(k_zg k_w^{-1})$. Thus $\theta$ and $\pi$ have commuting ranges, and there is an injection $\theta\otimes_{\max}\pi$ of the maximal tensor product $C^*(G_x)\otimes_{\max} M_{G^0}(\C)$ into $C^*(G)$ such that 
\begin{equation}\label{thetaotimespi}
\theta\otimes_{\max}\pi(i_{G_x}(g)\otimes e_{yz})=\theta(i_{G_x}(g))\pi(e_{yz})=i(k_yg k_z^{-1})
\end{equation} (see \cite[Theorem~B.27]{tfb}). Since each 
\[
i(h)=i\big(k_{c(h)}(k^{-1}_{c(h)}hk_{d(h)})k_{d(h)}^{-1})=\theta\otimes_{\max}\pi(u(k^{-1}_{c(h)}hk_{d(h)})\otimes e_{c(h)d(h)})
\]
belongs to the range of $\theta\otimes_{\max}\pi$, it is an isomorphism of $C^*(G_x)\otimes_{\max} M_{G^0}(\C)$ onto $C^*(G)$.

When we view $C^*(G_x)\otimes_{\max} M_{G^0}(\C)$ as $M_{G^0}(C^*(G_x))$, we identify the matrix $(a_{yz})$ with the sum $\sum_{y,z}a_{yz}\otimes i(e_{yz})$ (see \cite[Proposition~B.18]{tfb}). Hence the formulas \eqref{form1} and \eqref{form-1} follow from \eqref{thetaotimespi}.
\end{proof}

\begin{prop}\label{extendtrace}
Suppose that $G$ is a groupoid with finite unit space $G^0$ and that $G$ acts transitively on $G^0$. Fix $x\in G^0$, let $G_x$ be the isotropy group at $x$, and choose elements $k_y\in G$ such that $d(k_y)=x$, $c(k_y)=y$ and $k_x=x$. Then for every trace $\tau$ on $C^*(G_x)$, there is a trace $\tau'$ on $C^*(G)$ such that
\begin{equation}\label{tau'trans}
\tau'(i(g))=
\begin{cases}
\tau\big(u(k_{d(g)}^{-1}gk_{d(g)})\big)&\text{if $d(g)=c(g)$}\\
0&\text{otherwise}.
\end{cases}
\end{equation}
The trace $\tau'$ does not depend on the choice of orbit representatives $\{k_y\}$.
\end{prop}

\begin{proof}
There is a trace $\sigma$ on $M_{G^0}(C^*(G_x))$ such that
\[
\sigma\big(\,(a_{yz})\,\big)=\sum_{y\in G^0} \tau(a_{yy}).
\]
Pulling this over to $C^*(G)$ via the isomorphism of Lemma~\ref{mrwiso} gives a trace $\tau'$ on $C^*(G)$ satisfying
\begin{align*}
\tau'(i(g))&=\sigma\big(\psi^{-1}(i(g))\big)=\sum_{y\in G^0}\tau\big(\psi^{-1}(i(g))_{yy}\big)\\
&=\begin{cases}
\tau(u(k_{d(g)}^{-1}gk_{d(g)}))&\text{if $d(g)=c(g)$}\\
0&\text{otherwise,}
\end{cases}
\end{align*}
as required.

If $\{h_y:y\in G^0\}$ is another set of orbit representatives, then $h_y^{-1}k_y\in G_x$ for all $y\in G^0$. Thus if $d(g)=c(g)$, say $d(g)=y=c(g)$, we have
\begin{align*}
\tau(u(h_y^{-1}gh_y))&=\tau\big(u(h_y^{-1}k_yk_y^{-1}gk_yk_y^{-1}h_y)\big)\\
&=\tau\big(u(h_y^{-1}k_y)u(k_y^{-1}gk_y)u(h_y^{-1}k_y)^{-1}\big)\\
&=\tau\big(u(k_y^{-1}gk_y)u(h_y^{-1}k_y)^{-1}u(h_y^{-1}k_y)\big)\quad\text{because $\tau$ is a trace}\\
&=\tau\big(u(k_y^{-1}gk_y)).\qedhere
\end{align*}
\end{proof}

\begin{cor}
In the situation of Proposition~\ref{extendtrace}, there are traces $\tau_e'$ and $\tau_1'$ on $C^*(G)$ such that
\begin{align*}
\tau_e'(i(g))&=\begin{cases}
1&\text{if $g\in G^0$}\\
0&\text{otherwise, and}\end{cases}\\
\tau_1'(i(g))&=\begin{cases}
1&\text{if $d(g)=c(g)$}\\
0&\text{otherwise.}\end{cases}
\end{align*} 
\end{cor}

\begin{proof}
The isotropy group algebra $C^*(G_x)$ has (at least) two traces $\tau_e$ and $\tau_1$ (see the discussion on page~6648 of \cite{lrrw}, for example), and we can apply Proposition~\ref{extendtrace} to both of them. The elements $k_{c(g)}^{-1}gk_{d(g)}$ of $G$ all belong to $G_x$, and hence the formulas for $\tau_e'$ and $\tau_1'$ follow from \eqref{tau'trans} and the corresponding properties of $\tau_e$ and $\tau_1$: $\tau_e(i(g))=0$ for $g\not=\id_x=e_{G_x}$, and $\tau_1(g)=1$ for all $g\in G_x$. 
\end{proof}

\begin{remark}
Both traces $\tau_e'$ and $\tau_1'$ can be spatially implemented. For $\tau_e'$, take the regular unitary representation $\lambda$ of $G$ on $\ell^2(G)$, defined in terms of the usual orthonormal basis $\{\xi_g:g\in G\}$ by
\[
\lambda_g\xi_h=\begin{cases}\xi_{gh}&\text{if $d(g)=c(h)$}\\
0&\text{if $d(g)\not=c(h)$.}
\end{cases}
\]
Then
\[
\tau_e'(a)=\sum_{x\in G^0}(\pi_\lambda(a)\xi_x\,|\,\xi_x)\quad\text{for $a\in C^*(G)$.}
\]
For $\tau_1'$, we take the unitary representation $U$ of $G$ on $\ell^2(G^0)$ from Example~\ref{easyrep}. Then in terms of the usual basis $\{h_x:x\in G^0\}$, we have
\[
\tau_1'(a)=\sum_{x\in G^0} (\pi_U(a)h_x\,|\,h_x)\quad\text{for $a\in C^*(G)$.}
\]
\end{remark}

Now suppose that $G$ is a groupoid with finite unit space $G^0$. We define a relation $\sim$ on $G^0$ by $x\sim y\Longleftrightarrow xGy\not=\emptyset$. Then $\sim$ is an equivalence relation, and we write $G^0/\!\!\sim$ for the set of equivalence classes. For each component $C\in G^0/\!\!\sim$, we let 
\[
G_C:=\{g\in G: d(g)\in C\text{ and }c(g)\in C\}
\] be the reduction of $G$ to $C$.

\begin{lemma}\label{directsum}
Let $i:G\to C^*(G)$ be the canonical unitary representation of $G$. Then $i|_{G_C}$ is a unitary representation of $G_C$; let $\pi_C$ be the corresponding homomorphism from $C^*(G_C)$ into $C^*(G)$. Then $\bigoplus_{C\in G^0/\!\sim}\pi_C:(a_C)\mapsto\sum_C\pi_C(a_C)$ is an isomorphism of $\bigoplus_{C\in G^0/\!\sim}C^*(G_C)$ onto $C^*(G)$.
\end{lemma}

\begin{proof}
Since $G_C$ is closed under all the operations of $G$, including the domain and codomain maps\footnote{Normally one would say that $G_C$ is a subobject of $G$. But we are using ``subgroupoid'' to mean a subobject with the same unit space.}, the restriction $i|_{G_C}$ is a unitary representation of $G_C$. Every unitary representation $U$ of $G_C$ extends to a unitary representation of $G$, simply by taking $U_g=0$ for $g\notin G_C$, and hence factors through the homomorphism $\pi_C$; thus $\pi_C$ is injective.  For $D\in G^0/\!\!\sim$ and $D\not=C$, we have $i_D(h)i_C(g)=(i_D(h)i(d(h)))(i(c(g))i_D(g))=0$, and hence the ranges of $\pi_D$ and $\pi_C$ satisfy $\pi_D(C^*(G_D))\pi_C(C^*(G_C))=0$. Since every $g\in G$ belongs to some $G_C$ (the one for which $d(g)\in C$ and $c(G)\in C$), it follows from \cite[Proposition~A.6]{R_book}, for example, that $\bigoplus\pi_C$ is an isomorphism of $\bigoplus C^*(G_C)$ onto $C^*(G)$.
\end{proof}

\begin{cor}\label{tracesonCG}
Suppose that $G$ is a groupoid with finite unit space $G^0$, and $C\in G^0/\!\!\sim$. Fix $x\in C$ and choose representatives $k_y$ for each $xCy$ with $y\in C$. Then for every trace $\tau_C$ on $C^*(G_x)$, there is a trace $\tau_C'$ on $C^*(G)$ such that
\begin{equation}\label{formfortauC}
\tau_C'(i(g))=
\begin{cases}
\tau_C(u(k_{c(g)}^{-1}gk_{d(g)}))&\text{if $d(g)=c(g)\in C$}\\
0&\text{otherwise}.
\end{cases}
\end{equation}
Every trace on $C^*(G)$ is a sum of traces of the form $\tau_C'$.
\end{cor}

\begin{proof}
The groupoid $G_C$ is transitive, so Proposition~\ref{extendtrace} gives us a trace $\tau_C'$ on $C^*(G_C)$. We can extend this trace to $\bigoplus_D C^*(G_D)$ by taking it to be zero on the other summands, and then use the isomorphism of Lemma~\ref{directsum} to pull it over to a trace $\tau'$ on $C^*(G)$. Since the isomorphism on the summand $C^*(G_C)$ comes from $i|_{G_C}$, the formula \eqref{formfortauC} follows from \eqref{tau'trans}. For the last remark, note that if $\tau$ is a trace on $C^*(G)$, then $\tau\circ \pi_C$ is a trace on $C^*(G_C)$, and we can recover $\tau$ as $\sum_C (\tau\circ\pi_C)'$.
\end{proof}

At this point, we discuss an important class of examples that motivated the previous work of Exel and Pardo \cite{EP2}. The Cuntz-Pimsner  algebras of the following self-similar groupoids are a family of algebras constructed by Katsura \cite{k2} to provide models of Kirchberg algebras. These examples also fit the theory of \cite{EP2}, and in particular have the property $s(g\cdot e)=s(e)=g\cdot s(e)$.

\begin{example}
\label{Katsura example}
Suppose that $N\in \N$, and consider matrices $A=(a_{ij})$, $B=(b_{ij})$ in $M_N(\N)$ such that $A$ has no nonzero rows and $a_{ij}=0\Longrightarrow b_{ij}=0$. Let $E$ be the graph
with
\[
E^0=\{1,2,\cdots,N\},\ E^1=\{e_{i,j,m}:0\leq m<a_{ij}\},\ r(e_{i,j,m})=i,\
s(e_{i,j,m})=j.
\]
We define an automaton $\AA$ over $E$ as follows.  Let\footnote{We have used the calligraphic $\AA$ to distinguish the automaton from the matrix $A$ in the definition of the Katsura action, thereby avoiding a clash with the established notation in \cite{k2} and \cite{EP2}.} $\AA:=\{a_i : 1 \leq i \leq N\}$ with $s(a_i)=i$ and $r(a_i)=i$. We then define, for $\mu \in jE^*$,
\begin{equation}
\label{Katsura_action}
a_i \cdot e_{i,j,m}\mu=e_{i,j,n}(a_j^l \cdot \mu) \quad \text{where} \quad
b_{ij}+m=l{a_{ij}}+n \text{ and }0\leq n< a_i.
\end{equation}
We consider the faithful self-similar groupoid action $(G_{\AA},E)$ of Theorem~\ref{defnGA}.

Each generator $f_{a_i}$ of $G_{\AA}$ acts trivially on the vertex set $E^0$, and hence so does every element of $G_{\AA}$. Thus the connected components $C\in E^0/\!\!\sim$ are the singletons $C_i=\{i\}$. The reductions $(G_{\AA})_{\{i\}}$ each contain one generator $f_{a_i}$, and hence they act on different subtrees of $T_E$. Thus every element of $G_{\AA}$ is a power of one generator $f_{a_i}$, and the reductions $(G_{\AA})_{\{i\}}$ are cyclic groups generated by $f_{a_i}$. For each $i$, there is either an integer $n_i$ such that $(G_{\AA})_{\{i\}}=\{i, f_{a_i},f_{a_i}^2,\cdots,f_{a_i}^{n_i-1}\}$, or $(G_{\AA})_{\{i\}}=\{f_{a_i}^k:k\in \Z\}$. In the first case, the normalised traces on $C^*((G_{\AA})_{\{i\}})=\C^n$ form a simplex $\Delta_i$ of dimension $n_i-1$, in the second case they form a copy $\Delta_i$ of the probability measures on the circle $\T$. Corollary~\ref{tracesonCG} then says that for each $i$ and each $\mu_i\in \Delta_i$, there is a distinct normalised trace $\tau'_{i,\mu_i}$ on $C^*(G_{\AA})$.  Theorem~\ref{Thm:KMS_beta_tau} says that each gives a distinct KMS$_\beta$ state on $\TT(G_{\AA},E)$.
\end{example}

\begin{example}\label{s,r clarifying example-part two}
We return to the self-similar action described in Example~\ref{s,r clarifying example}, and the associated subgroupoid $G=G_A$ of $\PI(E^*)$. With $E$ listed as $\{v,w\}$, the vertex matrix is
\[
B=\begin{pmatrix}1&1\\2&0\end{pmatrix},
\]
which has spectrum $\sigma(B)=\{-1,2\}$ and $\rho(B)=2$. So Theorem~\ref{Thm:KMS_beta_tau} describes the KMS$_\beta$ states for $\beta>\ln 2$ in terms of the normalised traces on $C^*(G)$. We aim to apply our results to find these.

Since $a\cdot v=w$, the equivalence $\sim$ on $E^0=\{v,w\}$ has a single orbit. Thus Proposition~\ref{extendtrace} gives us a bijection between traces $\tau$ on the group algebra $C^*(G_v)$ and traces $\tau'$ on $C^*(G)$. To get a specific formula for $\tau'$, we take $k_v=\id_v$, as instructed for the base point $v$, and $k_w=f_a$. Then we have
\begin{align*}
\tau'(i(\id_v))&=\tau(u(\id_v^{-1}\id_v)=\tau(u(\id_v)),\text{ and}\\
\tau'(i(\id_w))&=\tau(u(f_a^{-1}\id_w f_a)=\tau(u(\id_v)).
\end{align*}
So each trace on $C^*(G_v)$ such that $\tau(\id_v)=\frac{1}{2}$ gives a normalised trace $\tau'$ on $C^*(G)$.  

To describe the KMS states, we need to understand $C^*(G_v)$. Recall that $G$ is the subgroupoid of $\PI(E^*)$ generated by two partial isomorphisms $f_a$, $f_b$ with $d(f_a)=v=c(f_b)$ and $c(f_a)=w=d(f_b)$. Every $g\in G_v$ can be writen $g=g_1g_2\cdots g_n$ with $g_i\in S:=\{f_a,f_b,f_a^{-1},f_b^{-1}\}$, $g_{i+1}\not=g_i^{-1}$ and $c(g_1)=v=d(g_n)$. Since each element of $S$ switches the vertices $v$ and $w$, $n$ has to be even, say $n=2k$. Then $c(g)=c(g_1)=v$ forces $g_1\in \{f_b,f_a^{-1}\}$. First suppose that $g_1=f_b$. Then $g_2$ has to be in $\{f_a,f_b^{-1}\}$, and since $g_2\not= g_1^{-1}$, we must have $g_2=f_a$. Continuing this way shows that $g=(f_bf_a)^k$. On the other hand, if $g_1=f_a^{-1}$, we have $g_2=f_b^{-1}$, and continuing gives $g=(f_a^{-1}f_b^{-1})^k=(f_bf_a)^k$. Thus
\[
G_v=\{(f_bf_a)^{k}:k\in \Z\}.
\]
Then $k\mapsto (f_bf_a)^{k}$ is either an isomorphism of $\Z$ onto $G_v$ or a quotient map, in which case there exists $k\in \N$ such that $(f_bf_a)^k=e_{G_v}=\id_v$. We will show that there is no such $k$, and hence $G_v$ is a copy of $\Z$.

So we suppose that $k\in \N$, and have to prove that $(f_bf_a)^k\not=\id_v$. First we use the defining relations
\[
(f_bf_a)(1)=f_b(f_a(1))=f_b(a\cdot 1)=f_b(4)=b\cdot 4=2,
\]
and similarly $(f_bf_a)(2)=1$. Thus $(f_bf_a)^k$ is not the identity for $k$ odd. So we suppose that $k=2n$ is even. Next we do some calculations with the restriction map:
\begin{align*}
(f_bf_a)|_1&=f_{b|_{a\cdot 1}}f_{a|_1}=f_{b|_4}\id_v=f_{b|_4}=f_a, \text{ and }\\
(f_bf_a)|_2&=f_{b|_3}f_{a|_2}=\id_vf_b=f_b,\\
\intertext{which imply}
(f_bf_a)^2|_1&=(f_bf_a)|_{(f_bf_a)(1)}(f_bf_a)|_1=(f_bf_a)|_2f_a=f_bf_a.
\end{align*}
Continuing this way, we find that $(f_bf_a)^{2n}|_1=(f_bf_a)^n$. Now by looking at the action of $(f_bf_a)^{2n}$ on sufficiently long words of the form $\mu=111\cdots 1$ we can reduce $2n$ repeatedly by factors of $2$, and arrive at $(f_bf_a)^{2n}|_\mu=(f_bf_a)^m$ with $m$ odd, so that $(f_bf_a)^{2n}\cdot\mu 1=\mu (f_bf_a)^m(1)=\mu 2$. For example,
\begin{align*}
(f_bf_a)^4(111)&= (f_bf_a)^4(1)(f_bf_a)^4|_1(11)=(f_bf_a)^4(f_bf_a)^2(11)\\
&=1(f_bf_a)^2(1)(f_bf_a)^2|_1(1)=11(f_bf_a)(1)\\
&=112.
\end{align*}
So whatever $k$ is, $(f_bf_a)^k$ acts nontrivially on $1E^*$.

Thus $G_v$ is isomorphic to $\Z$, and the traces of norm $\frac{1}{2}$ on $C^*(G_v)$ are in one-to-one correspondence (via multiplication by $2$) with the set of probability measures on the unit circle $\T$. Thus so is the simplex of KMS$_\beta$ states on $(\TT(G,E),\sigma)$ for every $\beta>\ln 2$.
\end{example}

There are two families of self-similar groupoids for which complete classifications of KMS states have previously been found. The first is when $|E^0|=1$, in which case our dynamical system is the one studied in \cite{lrrw}. The second concerns the groupoid $G=C(E^0)$, in which every morphism is the identity morphism at some vertex. For this groupoid, for $v\in E^0$ and for $g=\id_v$, Proposition 6.4(1) says that $u_g=p_v$, and hence $\TT(G,E)$ is universal for Toeplitz-Cuntz-Kreiger $E$-families. Thus $\TT(G,E)=\TT C^*(E)$, and the dynamical system is the system $(\TT C^*(E),\alpha)$ studied in \cite{hlrs}. As a reality check we reconcile our new results with those of \cite{hlrs}.

We consider a finite directed graph $E$ with vertex matrix $B$, and consider a state $\phi_\epsilon$ associated to a vector $\epsilon\in \Sigma_\beta\subset [0,\infty)^{E^0}$, as in \cite[Theorem~3.1]{hlrs}. Equation (3.1) in \cite{hlrs} says that
\begin{equation}\label{fromHLRS}
\phi_\epsilon(s_\kappa s_\lambda^*)=\delta_{\kappa,\lambda}e^{-\beta|\kappa|}m_{s(\kappa)}=\delta_{\kappa,\lambda}e^{-\beta|\kappa|}\big((1-e^{-\beta |\kappa|}B)\epsilon\big)_{s(\kappa)}.
\end{equation}

We look for a normalised trace $\tau$ on $C^*(G)=C(E^0)$ such that $\phi_\epsilon=\psi_{\beta,\tau}$. The Banach-space dual of $C(E^0)$ is $\ell^1(E^0)$, and the states on $C(E^0)$ are given by vectors $\tau=(\,\tau(p_v)\,)$ in $[0,\infty)^{E^0}$ with $\|\tau\|_1=1$. Equation~\eqref{KMS_tau_formula} says that
\begin{align*}
\psi_{\beta,\tau}(s_\kappa s_\lambda^*)&=\delta_{\kappa, \lambda} Z(\beta,\tau)^{-1} \sum_{j=|\mu|}^\infty e^{-\beta j} \Big(\sum_{\{\mu \in s(\kappa)E^{j-|\kappa|}: \, g \cdot \mu=\mu\}} \tau\big(i_{g|_\mu}\big)\Big)\\
&=\delta_{\kappa, \lambda} Z(\beta,\tau)^{-1} e^{-\beta|\kappa|} \Big(\sum_{k=0}^\infty e^{-\beta k} \sum_{\{\mu \in s(\kappa)E^{k}: \, g \cdot \mu=\mu\}} \tau\big(i_{g|_\mu}\big)\Big).
\end{align*}
So we suppose $\kappa=\lambda$. Then $g$ is the vertex $s(\kappa)$, $g\cdot \mu=\mu$ for all $\mu\in s(\kappa)E^k$, and $s(\kappa)|_{\mu}=s(\mu)$. Thus
\begin{align*}
\psi_{\beta,\tau}(s_\kappa s_\lambda^*)&=\delta_{\kappa, \lambda} Z(\beta,\tau)^{-1} e^{-\beta|\kappa|}\Big(\sum_{k=0}^\infty e^{-\beta k}\sum_{\mu\in s(\kappa)E^k}\tau(p_{s(\mu)})\Big)\\
&=\delta_{\kappa, \lambda} Z(\beta,\tau)^{-1} e^{-\beta|\kappa|}\Big(\sum_{k=0}^\infty e^{-\beta k}\sum_{v\in E^0}|s(\kappa)E^k|\tau(p_{v})\Big)\\
&=\delta_{\kappa, \lambda} Z(\beta,\tau)^{-1} e^{-\beta|\kappa|}\Big(\sum_{k=0}^\infty e^{-\beta k}\sum_{v\in E^0}B^k(s(\kappa),v)\tau(p_{v})\Big).
\end{align*}
Viewing $\tau$ as the vector with entries $\tau(p_v)$, we recognise the inner sum as the value of the matrix product $B^k\tau$ at the vertex $s(\kappa)$. Because $\beta>\ln\rho(B)$, the series $\sum_{k=0}^\infty e^{-\beta k}B^k\tau$ converges with sum $(1-e^{-\beta}B)^{-1}\tau$, and hence 
\begin{equation}\label{fromLRRW2}
\psi_{\beta,\tau}(s_\kappa s_\lambda^*)=\delta_{\kappa, \lambda} Z(\beta,\tau)^{-1} e^{-\beta|\kappa|}\big((1-e^{-\beta}B)^{-1}\tau\big)_{s(\kappa)}.
\end{equation}

When we compare \eqref{fromHLRS} and \eqref{fromLRRW2}, we see that, apart from the factor  $Z(\beta,\tau)^{-1}$, they look similar. However, if they are to be exactly the same, then $\epsilon$ has to be a scalar multiple of $\tau$. The key observation is that the only positive scalar multiple of $\epsilon$ which gives a state is $\|\epsilon\|_1^{-1}\epsilon$. More formally, we have:

\begin{prop}\label{hlrsvslrrw2}
Suppose that $E$ is a finite graph with vertex matrix $B$ and that $\beta>\ln\rho(B)$. Let $y^\beta$ be the vector in $[1,\infty)^{E^0}$ with entries $y^\beta_v=\sum_{\mu\in E^0v}e^{-\beta|\mu|}$. Suppose that $\epsilon\in [0,\infty)^{E^0}$ satisfies $y^\beta\cdot \epsilon=1$. Then $\tau:=\|\epsilon\|_1^{-1}\epsilon$ is a normalised trace on $C(E^0)$ satisfying $Z(\beta,\tau)=\|\epsilon\|_1^{-1}$, and $\phi_{\epsilon}=\psi_{\beta,\tau}$. Conversely, if $\tau$ is a normalised trace on $C(E^0)$, then $\epsilon:=Z(\beta,\tau)^{-1}\tau$ satisfies $y^\beta\cdot \epsilon=1$, and $\phi_{\epsilon}=\psi_{\beta,\tau}$.
\end{prop}

The key to the proof is the following computation.

\begin{lemma}\label{compZ} 
Resume the notation of Proposition~\ref{hlrsvslrrw2}. Then for every $\epsilon\in [0,\infty)^{E^0}$, we have $Z(\beta,\epsilon)=y^\beta\cdot \epsilon$.
\end{lemma}

\begin{proof}
We compute
\begin{align*}
Z(\beta,\epsilon)&=\sum_{k=0}^\infty\sum_{\mu\in E^k}e^{-\beta k}\epsilon_{s(\mu)}\\
&=\sum_{k=0}^\infty\sum_{w\in E^0}\sum_{v\in E^0}e^{-\beta k}B^k(w,v)\epsilon_{v}\\
&=\sum_{k=0}^\infty\sum_{w\in E^0}e^{-\beta k}(B^k\epsilon)_w\\
&=\sum_{w\in E^0}\Big(\sum_{k=0}^\infty e^{-\beta k}(B^k\epsilon)\Big)_w\\
&=\sum_{w\in E^0}\big((1-e^{-\beta} B)^{-1}\epsilon\big)_w.
\end{align*}
But $\big((1-e^{-\beta} B)^{-1}\epsilon\big)_w$ is the number denoted by $m_w$ in \cite[Theorem~3.1]{hlrs}, and hence the result follows from Equation (3.3) in the proof of \cite[Theorem~3.1]{hlrs}.
\end{proof}

\begin{proof}[Proof of Proposition~\ref{hlrsvslrrw2}]
Suppose that $y^\beta\cdot \epsilon=1$. Then $\|\tau\|_1=\|\,\|\epsilon\|_1^{-1}\epsilon\|_1=1$, so $\tau$ is a normalised trace. Lemma~\ref{compZ} implies that
\[
Z(\beta,\tau)=\|\epsilon\|_1^{-1}Z(\beta,\epsilon)=\|\epsilon\|_1^{-1}(y^\beta\cdot\epsilon)=\|\epsilon\|_1^{-1}.
\]
Comparing \eqref{fromHLRS} and \eqref{fromLRRW2} shows that $\phi_\epsilon$ and $\psi_{\beta,\tau}$ agree on all elements $s_{\kappa}s_{\lambda}^*$, and hence by linearity and continuity agree on all of $\TT(G,E)=\TT C^*(E)$.

Suppose next that $\|\tau\|_1=1$ and $\epsilon= Z(\beta,\tau)^{-1}\tau$. Then $\|\epsilon\|_1=\|Z(\beta,\tau)^{-1}\tau\|_1=Z(\beta,\tau)^{-1}$, so $\tau=\|\epsilon\|_1^{-1}\epsilon$, and Lemma~\ref{compZ} gives
\[
y^\beta\cdot \epsilon=Z(\beta,\tau)^{-1}(y^\beta\cdot \tau)=1.
\]
The previous part now gives $\phi_\epsilon=\psi_{\beta,\tau}$.
\end{proof}

\section{KMS states at the critical inverse temperature}\label{sec:crit}

Suppose that $E$ is a finite directed graph with no sources and that $G$ is a groupoid which acts self-similarly on $E$. We consider the ideal $I$ in $\TT(G,E)$ generated by the gap projections
\begin{equation}\label{gapprojs}
P:=\Big\{p_v-\sum_{e\in vE^1}s_es_e^*:v\in E^0\Big\},
\end{equation}
and the quotient $\OO(G,E):=\TT(G,E)/I$, which we call the Cuntz-Pimsner algebra of $(G,E)$. The gap projections are all fixed by the action $\sigma$ of $\R$ on $\TT(G,E)$, and hence $\sigma$ induces a dynamics on $\OO(G,E)$, which we also denote by $\sigma$.

\begin{prop}\label{prop:factors}
Suppose that $E$ is a finite graph with no sources, that $(G,E)$ is a self-similar groupoid action, and that $\phi$ is a KMS$_\beta$ state of $(\TT(G,E),\sigma)$ for some $\beta\geq \ln\rho(B)$. Define $m^\phi\in [0,1]^{E^0}$ by $m_v^\phi :=\phi(p_v)$ for $v\in E^0$. Then $\phi$ factors through $\OO(G,E)$ if and only if $B m^\phi=e^\beta m^\phi$.
\end{prop}

\begin{proof}
Suppose that $\phi$ factors through the quotient map $q$ of $\TT(G,E)$ onto $\OO(G,E)$. Then $(q(s_e), q(p_v))$ is a Cuntz-Krieger $E$-family in $\OO(G,E)$, and hence there is a homomorphism $\pi_{q\circ s,q\circ p}:C^*(E)\to \OO(G,E)$. This homomorphism is equivariant for the dynamics $\alpha:\R\to \Aut C^*(E)$ studied in \cite{hlrs} and the dynamics $\sigma$ on $\OO(G,E)$, and thus $\phi\circ\pi_{q\circ s,q\circ p}$ is a KMS$_{\beta}$ state on $(C^*(E),\gamma)$. Thus \cite[Proposition 2.1]{hlrs} implies that $B m^\phi=e^\beta m^\phi$.

Conversely, suppose that $B m^\phi=e^\beta m^\phi$. Then by \cite[Proposition~2.1]{hlrs}, the composition $\phi\circ\pi_{q\circ s,q\circ p}$ factors through a state of $C^*(E)$, and hence vanishes on the gap projections in $\TT C^*(E)$. Thus $\phi$ vanishes on the set $P$ in \eqref{gapprojs}. Since the elements of $P$  are fixed by the action $\sigma$, we can apply \cite[Lemma 2.2]{hlrs} to the set of analytic elements 
\[
\FF=\big\{s_{\mu}u_gs_{\nu}^* : \mu,\nu \in E^*, g \in G \text{ and } s(\mu)=g\cdot s(\nu)\big\},
\]
and deduce that $\phi$ factors through a state of $\OO(G,E)$.
\end{proof}

The inspiration for our next result comes from \cite[Proposition~7.2]{lrrw}. 

\begin{prop}\label{Fng}
Suppose that $E$ is a strongly connected finite graph, and that $(G,E)$ is a self-similar groupoid action. Then the vertex matrix $B$ is irreducible, and has a unique unimodular Perron-Frobenius eigenvector $x\in (0,\infty)^{E^0}$. For $g \in G \setminus E^0$, $v\in E^0$ and $k\geq 0$, we define
\begin{align*}
F_g^k(v) &:= \{\mu \in d(g)E^k v : g \cdot \mu = \mu \text{ and } g|_{\mu}=v \},\text{ and}\\
c_{g,k}&:=\rho(B)^{-k} \sum_{v\in E^0} |F_g^k(v)|x_v.
\end{align*}
Then for each $g\in G\setminus E^0$, the sequence $\{c_{g,k}:k\in \N\}$ is increasing and converges. The limit $c_g$ belongs to $[0,x_{d(g)})$. 
\end{prop}

\begin{proof}
That $E$ is strongly connected says precisely that $B$ is irreducible in the sense that for all $v,w\in E^0$, there exists $n$ such that $B^n(v,w)\not=0$. Then the Perron-Frobenius theorem says that $B$ has a unique eigenvector $x$ with strictly positive entries such that $\|x\|_1=1$ (see \cite[Theorem~2.6]{DZ} or \cite[Theorem~1.6]{seneta}). 

Now we fix $g\in G\setminus E^0$ and show that $\{c_{g,k}\}$ is increasing. For $\mu \in F_g^k(v)$ and $e \in vE^1$, we have
\[
g \cdot (\mu e) = \mu(g|_\mu \cdot e) =\mu(v \cdot e) = \mu e \quad\text{ and}\quad g|_{\mu e} = (g|_\mu)|_e = v|_e = s(e),
\]
so $\mu e \in F_g^{k+1}(s(e))$. Thus $|F_g^{k+1}(v)|\geq \sum_{w \in E^0} |F_g^{k}(w)|B(w,v)$, and we compute
\begin{align*}
c_{g,k+1}&=\rho(B)^{-(k+1)} \sum_{v\in E^0}|F_g^{k+1}(v)|x_v \\
&\geq \rho(B)^{-(k+1)} \sum_{v \in E^0} \sum_{w \in E^0} |F_g^{k}(w)|B(w,v)x_v \\
&= \rho(B)^{-(k+1)} \sum_{w \in E^0} |F_g^{k}(w)|\rho(B)x_w\\
&= c_{g,k}.
\end{align*}

Next we find an upper bound for the sequence $\{c_{g,k}\}$. Since $g\notin E^0$ and $G$ acts faithfully on $T_E$, the action of $g$ on $d(g)E^*$ is nontrivial, and there exist $j \in \N$ and $\mu \in d(g)E^j$ such that $g\cdot \mu\neq \mu$. So $\mu \notin F_g^j(s(\mu))$, and for every $\mu' \in s(\mu)E^{k-j}$ with $k\geq j$, the path $\mu\mu'$ is not in $F_g^k(d(g))$. Thus for $k \geq j$ and $v\in E^0$, we have
\[
|F_g^k(v)|\leq |d(g)E^kv|-|s(\mu)E^{k-j}v|=B^k(d(g),v)-B^{k-j}(s(\mu),v).
\]
The vector $x$ is also the Perron-Frobenius eigenvector of the powers of $B$, and we have $\rho(B^k)=\rho(B)^k$. Thus for $k\geq j$ we have
\begin{align}\label{estcgk}
c_{g,k} &\leq \rho(B)^{-k}\Big(\sum_{v\in E^0}B^k(d(g),v)x_v-\sum_{v\in E^0}B^{k-j}({s(\mu)},v)x_v\Big) \\
&= \rho(B)^{-k}\big((B^kx)_{d(g)}-(B^{k-j}x)_{s(\mu)}\big)\notag\\
&=\rho(B)^{-k}\big(\rho(B^k)x_{d(g)}-\rho(B^{k-j})x_{s(\mu)}\big)\notag\\
&=x_{d(g)}-\rho(B)^{-j}x_{s(\mu)}.\notag
\end{align}
The inequality \eqref{estcgk} shows, first, that the  increasing sequence $\{c_{g,k}:k\in \N\}$ is bounded above, and hence converges, say to $c_g$. Since the Perron-Frobenius eigenvector has strictly positive entries, \eqref{estcgk} also shows that $c_{g,k}$ is bounded away from $x_{d(g)}$, and hence the limit $c_g$ belongs to $[0,x_{d(g)})$.
\end{proof}

We can now state our main result about KMS states at the critical inverse temperature. 

\begin{thm}\label{KMSatcritical}
Suppose that $E$ is a strongly connected graph with vertex matrix $B$, and $(G,E)$ is a self-similar groupoid action. Let $\{c_g:g\in G\}$ be the numbers described in Proposition~\ref{Fng}.
\begin{enumerate}
\item\label{existcritKMS}  There is a KMS$_{\ln \rho(B)}$ state of $(\OO(G,E),\sigma)$ such that
\begin{equation}\label{formcritstate}
\psi(s_\kappa u_g s_\lambda^*)=\begin{cases}
\rho(B)^{-|\kappa|}c_g&\text{if $\kappa=\lambda$ and $d(g)=c(g)=s(\kappa)$}\\
0&\text{otherwise.}
\end{cases}
\end{equation}
\item\label{uniqueKMS} Suppose that for every $g \in G \setminus E^0$, the set $\{g|_\mu :\mu \in d(g)E^*\}$ is finite. Then the state in part \eqref{existcritKMS} is the only KMS state of $(\OO(G,E),\sigma)$.
\end{enumerate}
\end{thm}

To find KMS states at the critical inverse temperature $\beta_c$, we take a sequence of KMS$_\beta$ states for $\beta>\beta_c$ and take limits as $\beta$ decreases to $\beta_c$. We will  use the Perron-Frobenius eigenvector $x$ (which determines the numbers $c_g$) to get a trace $\tau$ on $C^*(G)$, build a KMS state $\psi_{\beta,\tau}$ using Theorem~\ref{Thm:KMS_beta_tau}, and then take limits. However, traces on $C^*(G)$ satisfy $\tau(i_{d(g)})=\tau(i_{g^{-1}}i_g)=\tau(i_gi_{g^{-1}})=\tau(i_{c(g)})$, so to get a trace from $x$ we  need it to satisfy $x_{d(g)}=x_{c(g)}$. Fortunately, for us this is automatic because $G$ acts self-similarly on $E$:

\begin{prop}\label{blockformx}
Suppose that $E$ is a strongly connected finite graph with no sources, and that $(G,E)$ is a self-similar groupoid action. Let $x$ be the unimodular Perron-Frobenius eigenvector of the vertex matrix $B$. Then $x_v=x_w$ whenever $v=g\cdot w$ (or equivalently, whenever $v=d(g)$ and~$w=c(g)$).
\end{prop}

We will show that the existence of the self-similar action on $E$ puts constraints on the vertex matrix $B$. We consider the set $E^0/\!\!\sim$ of equivalence classes for the relation on $E^0$ defined by
\[
v\sim w\Longleftrightarrow\text{ there exists $g\in G$ such that $d(g)=v$ and $c(g)=w$}
\]
(see also \S\ref{sec:traces}). We write $[v]$ for the class of $v\in E^0$ in $E^0/\!\!\sim$.

\begin{lemma}\label{rowsums}
Suppose that $C,D\in E^0/\!\!\sim$. Then for $v_1,v_2\in C$, we have
\[
\sum_{w\in D} B(v_1,w)=\sum_{w\in D} B(v_2,w).
\]
\end{lemma}

\begin{proof}
Since $v_1$ and $v_2$ are in the same component, there exists $g\in G$ such that $v_1=d(g)$ and $v_2=c(g)$. Suppose that $e\in v_1E^1D$. Then $g\cdot e\in v_2E^1D$, because $s(g\cdot e)=g|_e\cdot s(e)\sim s(e)\in D$. So $g:v_1E^1D\to v_2E^1D$. We similarly have $g^{-1}:v_2E^1D\to v_1E^1D$ and $g^{-1}\cdot(g\cdot e)=(g^{-1}g)\cdot e=e$, so $g$ is a bijection of $v_1E^1D$ onto $v_2E^1D$. Thus
\[
\sum_{w\in D} B(v_1,w)=|v_1E^1D|=|v_2E^1D|=\sum_{w\in D} B(v_2,w).\qedhere
\]
\end{proof}

\begin{proof}[Proof of Proposition~\ref{blockformx}]
We define a matrix $R=(r_{CD})$ over $E^0/\!\!\sim$ by $r_{CD}=|vE^1D|$ for $v\in C$; Lemma~\ref{rowsums} shows $r_{CD}$ is well-defined, independent of $v\in C$. Since paths $\mu\in E^*$ give nonzero entries in $R^{|\mu|}$, $R$ is an irreducible nonnegative matrix. Let $z$ be a Perron-Frobenius eigenvector for $R$, and define $y\in [0,\infty)^{E^0}$ by $y_v:=z_{[v]}$. Then for $v\in C$, we have
\begin{align*}
(By)_v&=\sum_{w\in D} B(v,w)y_w=\sum_{D\in E^0/\!\sim}\sum_{w\in D}B(v,w)z_D\\
&=\sum_{D\in E^0/\!\sim}|vE^1D|z_D=\sum_{D\in E^0/\!\sim} r_{CD}z_D\\
&=\rho(R)z_C=\rho(R)y_v.
\end{align*}
In other words, $y$ is an eigenvector of $B$ with eigenvalue $\rho(R)$. Since $z$ is a Perron-Frobenius eigenvector of $R$, we have $z_D>0$ for all $D$, and hence $y_v>0$ for all $v$. Thus $y$ is a strictly positive eigenvector for $B$, and hence must be a scalar multiple $tx$ of the unimodular Perron-Frobenius eigenvector $x$ for $B$. But now if $v=g\cdot w$, we have $[v]=[w]$ and $x_v=tz_{[v]}=tz_{[w]}=x_w$.
\end{proof}

Now for the proof of Theorem~\ref{KMSatcritical}\eqref{existcritKMS}, we take $x$ to be the unimodular Perron-Frobenius eigenvector of $B$, and for $C\in E^0/\!\!\sim$ we  write $x_C$ for the common value of $\{x_v:v\in C\}$. For each $C$, we choose a representative $v\in C$, and apply Corollary~\ref{tracesonCG} to the state $\tau_e$ on $C^*(G_e)$ to get a trace $\tau'_{e,C}$ on $C^*(G)$. Then
\[
\tau_x:=\sum_{C\in E^0/\!\sim}x_C\tau'_{e,C}
\]
is a tracial state on $C^*(G)$. (The normalisation works because $1=\sum_vx_v=\sum_C|C|x_C$.) Now for each $\beta>\ln\rho(B)$, Theorem~\ref{Thm:KMS_beta_tau} gives us a KMS$_\beta$ state $\psi_{\beta,\tau_x}$ satisfying \eqref{KMS_tau_formula}. The formula
\begin{equation}\label{comptaux}
\tau_x(i_g)=\begin{cases}x_C&\text{if $i_g=i_w$ for some $w\in C\subset G^0$}\\
0&\text{if $g\notin G^0$}\end{cases}
\end{equation}
shows that $\tau_x$ is independent of the choice of representatives $v\in C$.

We next look  at the normalising factor $Z(\beta,\tau_x)$ in the formula \eqref{KMS_tau_formula} for $\psi_{\beta,\tau_x}$.

\begin{lemma}
For $\beta>\ln\rho(B)$ and $x$, $\tau_x$ as above, we have
\begin{equation}\label{formforZ}
Z(\beta,\tau_x)=\big(1-e^{-\beta}\rho(B)\big)^{-1}.
\end{equation}
\end{lemma}

\begin{proof}
Writing the inside sum in terms of the vertex matrix gives
\begin{align*}
Z(\beta,\tau_x)&=\sum_{j=0}^\infty \sum_{\mu \in E^j} e^{-\beta j} \tau_x(i_{s(\mu)})\\
&=\sum_{j=0}^\infty\Big(\sum_{v\in E^0}\sum_{w\in E^0}e^{-\beta j}B^j(v,w)\tau_x(i_w)\Big).
\end{align*} 
Since Proposition~\ref{blockformx} gives $x_w=x_{C}$ for all $x\in C$, we deduce from \eqref{comptaux} that
\begin{align*}
Z(\beta,\tau_x)&=\sum_{v\in E^0}\sum_{j=0}^\infty e^{-\beta j}\Big(\sum_{w\in E^0}B^j(v,w)x_w\Big)\\
&=\sum_{v\in E^0}\sum_{j=0}^\infty e^{-\beta j}(B^jx)_v\\
&=\sum_{v\in E^0}\sum_{j=0}^\infty e^{-\beta j}\rho(B)^jx_v\\
&=\sum_{v\in E^0}\big(1-e^{-\beta}\rho(B)\big)^{-1}x_v\quad\text{because $e^\beta>\rho(B)$}\\
&=\big(1-e^{-\beta}\rho(B)\big)^{-1}\quad\text{because $\|x\|_1=1$.}\qedhere
\end{align*}
\end{proof}

\begin{proof}[Proof of Theorem~\ref{KMSatcritical}\eqref{existcritKMS}]
Suppose that $\{\beta_n\}$ is a decreasing sequence such that $\beta_n\to \beta$ as $n\to \infty$. If $\tau$ is any normalised trace on $C^*(G)$ then running the usual weak* compactness argument on the states $\psi_{\beta_n,\tau}$ gives us a KMS$_{\ln\rho(A)}$ state on $(\TT(G,E),\sigma)$ (see \cite[Proposition~7.6]{lrr} or \cite[page~6652]{lrrw}). The point is that, if we take $x$ to be the unimodular Perron-Frobenius eigenvector of $B$ and $\tau=\tau_x$, then we can compute the limit of $\{\psi_{\beta_n,\tau_x}\}$.

We now take a spanning element $s_\kappa u_g s_\lambda^*$ for $\TT(G,E)$, and compute \[
\lim_{n\to\infty}\psi_{\beta_n,\tau_x}(s_\kappa u_g s_\lambda^*).
\] Since everything vanishes otherwise, we assume that $\kappa=\lambda$, $s(\kappa)=d(g)$ and $d(g)=c(g)$.  Putting the formula \eqref{formforZ} into \eqref{KMS_tau_formula} gives
\begin{align*}
\notag
\psi_{\beta_n,\tau_x}(s_\kappa u_g s_\kappa^*)&=\big(1-e^{-\beta_n}\rho(B)\big) \sum_{j=|\kappa|}^\infty e^{-\beta_n j} \Big(\sum_{\{\mu \in s(\kappa)E^{j-|\kappa|}\,:\,g \cdot \mu=\mu\}} \tau_{x}\big(i_{g|_\mu}\big)\Big) \\
\notag
&=e^{-\beta_n|\kappa|}\big(1-e^{-\beta_n}\rho(B)\big) \sum_{k=0}^\infty e^{-\beta_n k}\Big(\sum_{\{\mu \in d(g)E^{k} \,:\, g \cdot \mu=\mu\}} \tau_{x}\big(i_{g|_\mu}\big)\Big).
\end{align*}
Now we recall from \eqref{comptaux} that $\tau_x(i_{g|_\mu})$ vanishes unless $g|_\mu=v$ for some vertex $v$ (strictly speaking, unless $g|_\mu$ is the identity morphism at $v$), and then $\tau_x(i_v)=x_v$. Thus the only $\mu$ which contribute to the sum are those which satisfy $g\cdot \mu=\mu$ and $g|_\mu=v$, which means that $\mu$ belongs to the set $F_g^{|\mu|}(v)$ of Proposition~\ref{Fng}. Thus we have
\begin{align*}
\psi_{\beta_n,\tau_x}(s_\kappa u_g s_\kappa^*)&=e^{-\beta_n|\kappa|}\big(1-e^{-\beta_n}\rho(B)\big) \sum_{k=0}^\infty e^{-\beta_n k} \sum_{v\in E^0}|F_g^k(v)|x_v\\
&=e^{-\beta_n|\kappa|}\sum_{k=0}^\infty \big(1-e^{-\beta_n}\rho(B)\big) \Big(\rho(B)^{-k}\sum_{v \in E^0} |F_g^k(v)| x_v\Big) \big(e^{-\beta_n}\rho(B)\big)^k.
\end{align*}
Now, modulo the factor $e^{-\beta_n|\kappa|}$, which converges to $\rho(B)^{-|\kappa|}$ as $n\to\infty$, we are in the situation of \cite[Lemma~7.4]{lrrw}, with $r=e^{-\beta_n}\rho(B)$ and $c_k=c_{g,k}= \rho(B)^{-k}\sum_{v \in E^0} |F_g^k(v)| x_v$, which by Proposition~\ref{Fng} increases to $c_g$ as $n\to \infty$. So by \cite[Lemma~7.4]{lrrw} we have
\[
\psi_{\beta_n,\tau_x}(s_\kappa u_g s_\kappa^*)\to \rho(B)^{-|\kappa|}c_g \text{ as $n\to \infty$.}
\]
Thus the limiting KMS$_{\ln\rho(B)}$ state $\psi$ satisfies \eqref{formcritstate}.

To see that $\psi$ factors through $\OO(G,E)$, we use again the canonical homomorphism $\pi_{s,p}:\TT C^*(E)\to \TT(G,E)$. Then $\psi\circ \pi_{s,p}$ is a KMS$_{\ln\rho(B)}$ state of $\TT C^*(E)$, and hence by \cite[Theorem~4.3(a)]{hlrs} factors through $C^*(E)$. By \cite[Proposition~4.1]{hlrs}, the vector $m^\psi=m^{\psi\circ\pi_{s,p}}$ satisfies $Bm^\psi=\rho(B)m^\psi=e^{\ln\rho(B)}m^\psi$, and Proposition~\ref{prop:factors} implies that $\psi$ factors through $\OO(G,E)$.
\end{proof}

To start the proof of Theorem~\ref{KMSatcritical}\eqref{uniqueKMS}, we suppose that $\phi$ is a KMS$_{\ln\rho(B)}$ state of $(\OO(G,E),\sigma)$. Then Proposition~\ref{prop:factors} implies that $m^\phi=\big(\phi(p_v)\big)$ satisfies $Bm^\phi=\rho(B)m^\phi$. Since $\sum_vp_v$ is the identity in $\OO(G,E)$, $m^\phi$ is the unimodular Perron-Frobenius eigenvector of $B$. Thus $\phi(p_v)=x_v=\psi(p_v)$ for all $v\in E^0$. Since both $\phi$ and $\psi$ are KMS$_{\ln\rho(B)}$ states, they are determined by their values on the generators $\{u_g:g\in G\}$ (see Proposition~\ref{KMS_algebraic}\eqref{states_iff}). Thus it suffices for us to prove that $\phi(u_g)=c_g=\psi(u_g)$ for all $g\in G\setminus E^0$. So we fix such a $g$.

To motivate our arguments, we start our calculation, following the argument on \cite[page~6654]{lrrw}. Suppose that $n\in \N$. We use the Cuntz relations $p_v=\sum_{\mu\in vE^n}s_\mu s_\mu^*$ and the decomposition $1=\sum_{v\in E^0}p_v$ to get 
\[
\phi(u_g)= \phi\Big( u_g \sum_{v \in E^0}\sum_{\mu \in vE^{n}} s_\mu s_\mu^* \Big)= \sum_{\mu \in d(g)E^{n}}\phi(s_{g\cdot \mu}u_{g|_{\mu}} s_\mu^*).
\]
Proposition~\ref{KMS_algebraic} implies that $\phi(s_{g\cdot \mu}u_{g|_{\mu}} s_\mu^*)=0$ unless $g\cdot \mu=\mu$. So we introduce\
\[
G^n_g(v):=\{\mu\in d(g)E^n v:g\cdot \mu=\mu\},
\]
and then \eqref{char:spanning} gives
\begin{align}\label{expandphiu}
\phi(u_g)&=\sum_{v\in E^0}\sum_{\mu\in G_g^n(v)}\rho(B)^{-n}\phi(u_{g|_\mu})\\
\notag
&=\sum_{v\in E^0} \sum_{\mu \in G^{n}_g(v) \setminus F_g^{n}(v)} \rho(B)^{-n} \phi(u_{g|_\mu}) +  \sum_{v\in E^0}\sum_{\mu \in F_g^{n}(v)} \rho(B)^{-n} \phi(u_{v}) \\
\label{splitsum}&= \rho(B)^{-n}\Big(\sum_{v\in E^0} \sum_{\mu \in G^{n}_g(v) \setminus F_g^{n}(v)}\phi(u_{g|_\mu})\Big)+  \rho(B)^{-n}\Big(\sum_{v\in E^0}|F_g^{n}(v)|x_v\Big).
\end{align}
We are aiming to prove that $\phi(u_g)=c_g$, and the second summand in \eqref{splitsum} goes to $c_g$ as $n\to \infty$. So we need to show that the first summand goes to $0$.

Now we recall that since $\phi$ is a KMS state, $\phi(u_{g|_\mu})=0$ unless  $c(g|_\mu)=d(g|_\mu)=s(\mu)=v$, say. In that case, $u_{g|_\mu}$ belongs to the corner $D=p_v\pi_u(C^*(G))p_v$; since $\phi|_D$ is a positive functional with norm $\phi|_D(1_D)=\phi(p_v)$, we have $|\phi(u_{g|_\mu})|\leq \phi(p_v)$. Thus
\begin{align}\label{estsum1}
\rho(B)^{-n}\sum_{v\in E^0} \sum_{\mu \in G^{n}_g(v) \setminus F_g^{n}(v)}\phi(u_{g|_\mu})
&\leq \rho(B)^{-n}\sum_{v\in E^0}\big|G^{n}_g(v) \setminus F_g^{n}(v)\big|\phi(p_v)\\
\notag&=\rho(B)^{-n}\sum_{v\in E^0}\big|G^{n}_g(v) \setminus F_g^{n}(v)\big|x_v.
\end{align}

At this point we invoke the finite-state hypothesis, which implies that for fixed $g\in G\setminus E^0$, the set 
\[
R:=\big\{g|_\mu:\mu\in d(g)E^*,\ g|_\mu\not= \id_{s(\mu)}\big\}
\]
is finite. Then there exists $j$ such that for all $h\in R$, there exists $\nu^h\in d(h)E^j$ such that $h\cdot\nu^h\not=\nu^h$. Since $E$ is strongly connected and $g|_\mu\cdot(\nu\lambda)=(g|_\mu\cdot\nu)(g|_{\mu\nu}\cdot\lambda)$, we can by making $\nu^h$ longer (and $j$ larger) assume that $\nu^h\in d(h)E^jd(h)$.

\begin{lemma}\label{keyineqlem}
With the preceding notation, we have
\begin{equation}\label{keyinequal}
\sum_{v\in E^0}\big|G^{nj}_g(v) \setminus F_g^{nj}(v)\big|x_v
\leq \sum_{v\in E^0}\big(\rho(B)^j-1\big)^nx_v\quad\text{for every $n\geq 0$.}
\end{equation}
\end{lemma}

\begin{proof}
We prove this by induction on $n$. For $n=0$, we we have $F_g^{nj}(v)=\emptyset$ because $g\notin E^0$, and both sides collapse to $\sum_v x_v=1$. Suppose that \eqref{keyinequal} holds for $n$, and consider $\lambda\in G^{(n+1)j}_g(v) \setminus F_g^{(n+1)j}(v)$. We can factor $\lambda=\mu\mu'$ with $\mu\in d(g)E^{nj}w$ and $\mu'\in wE^jv$ for some $w\in E^0$. Then $g\cdot \lambda=\lambda$, so $g\cdot\mu=\mu$ and $g|_\mu\cdot \mu'=\mu'$. Similarly, $g|_\lambda=(g|_\mu)|_{\mu'}$ is not $\id_{s(\lambda)}=\id_{s(\mu')}=v$ and $g|_\mu\not=\id_{s(\mu)}$. Thus $\mu\in G^{nj}_{g}(w) \setminus F_g^{nj}(w)$ and $\mu'\in G^{j}_{g|_\mu}(v) \setminus F_{g|_\mu}^{j}(v)$. Thus
\[
\sum_{w\in E^0}\big|G^{nj}_g(w) \setminus F_g^{nj}(w)\big|
\leq\sum_{v\in E^0}\sum_{w\in E^0}\big|G^{nj}_g(w) \setminus F_g^{nj}(w)\big|\,\big|G^{j}_{g|_\mu}(v) \setminus F_{g|_\mu}^{j}(v)\big|x_v.
\]
For each $\mu\in G^{nj}_g(w) \setminus F_g^{nj}(w)$, we have $d(g|_\mu)=s(\mu)=w$, and there is a path $\nu\in wE^jw$ which is not in $G^j_{g|_\mu}$, so 
\begin{align*}
\sum_{v\in E^0}\big|G^{(n+1)j}_g(v) \setminus F_g^{(n+1)j}(v)\big|x_v
&\leq \sum_{w\in E^0}\big|G^{nj}_g(w) \setminus F_g^{nj}(w)\big|\Big(\sum_{v\in E^0}B^j(w,v)x_v-x_w\Big)\\
&=\sum_{w\in E^0}\big|G^{nj}_g(w) \setminus F_g^{nj}(w)\big|\big(\rho(B)^jx_w-x_w\big)\\
&=\big(\rho(B)^j-1)\sum_{w\in E^0}\big|G^{nj}_g(w) \setminus F_g^{nj}(w)\big|x_w 
\end{align*}
Now the inductive hypothesis gives
\[
\sum_{v\in E^0}\big|G^{(n+1)j}_g(v) \setminus F_g^{(n+1)j}(v)\big|x_v
\leq \sum_{w\in E^0}\big(\rho(B)^j-1)^{n+1}x_w,
\]
as required.
\end{proof}

\begin{proof}[End of the proof of Theorem~\ref{KMSatcritical}\eqref{uniqueKMS}] 
We start with the formula \eqref{expandphiu}, and estimate:
\begin{align*}
|\phi(u_g)&-\psi(u_g)|=|\phi(u_g)-c_g|\\
&=\Big|\Big(\rho(B)^{-n}\sum_{v\in E^0} \sum_{\mu \in G^{n}_g(v) \setminus F_g^{n}(v)}\phi(u_{g|_\mu})\Big)+  \Big(\rho(B)^{-n}\sum_{v\in E^0}|F_g^{nj}(v)|x_v - c_g\Big)\Big|\\
&\leq\rho(B)^{-n}\sum_{v\in E^0}\big|G^{n}_g(v) \setminus F_g^{n}(v)\big|x_v+\Big|\rho(B)^{-n}\sum_{v\in E^0}|F_g^{nj}(v)|x_v - c_g\Big|\quad\text{ by \eqref{estsum1}}\\
&\leq \sum_{v\in E^0} \big(1-\rho(B)^{-j}\big)^nx_v +\Big|\rho(B)^{-n}\sum_{v\in E^0}|F_g^{nj}(v)|x_v - c_g\Big|\quad\text{by Lemma~\ref{keyineqlem}.}
\end{align*}
Now the first summand goes to $0$ as $n\to \infty$ because $\rho(B)\geq 1$ (see the appendix in \cite{hlrs}), and the second goes to $0$ by Proposition~\ref{Fng}. So $\phi(u_g)=\psi(u_g)$, and $\phi=\psi$ by Proposition~\ref{KMS_algebraic}. This completes the proof of Theorem~\ref{KMSatcritical}\eqref{uniqueKMS}.
\end{proof}

Since we know from Proposition~\ref{KMS_algebraic} that the restrictions of KMS states to the copy of $C^*(G)$ in $\OO(G,E)$ are traces, existence of the KMS$_{\ln\rho(B)}$ state gives the following groupoid version of \cite[Corollary~7.5]{lrrw}.

\begin{cor}
Suppose that $G$ is a finite groupoid that acts self-similarly on a finite graph $E$ with no sources and vertex set $G^0$. Then there is a trace $\tau$ on $C^*(G)$ such  that $\tau(u_g)=c_g$ for all $g\in G\setminus G^0$.
\end{cor}

\section{Computing KMS states at the critical inverse temperature}\label{sec:exscomp}

Suppose that $E$ is a strongly connected finite directed graph, that $(G,E)$ is a self-similar groupoid action, and that for every $g \in G \setminus E^0$, the set $\{g|_\mu :\mu \in d(g)E^*\}$ is finite. Then Theorem \ref{KMSatcritical}\eqref{uniqueKMS} implies that $(\OO(G,E),\sigma)$ has a unique KMS state given by \eqref{formcritstate}.

We say that $(G,E)$ \emph{contracts} to a finite subset $\NN$ of $G$ if  for every $g \in G$ there exists $n \in \N$ such that $\{g|_{\mu} : \mu \in d(g)E^*, |\mu| \geq n\} \subset \NN$. Then the \emph{Moore diagram} for $\NN$ is the labelled directed graph with vertex set $\NN$ and, for each $g \in \NN$ and $e \in d(g)E^1$, an edge from $g \in \NN$ to $g|_e \in \NN$ labelled $(e,g \cdot e)$. For $g \in \NN$ and $e \in d(g)E^1$, a typical edge in a Moore diagram looks like
\begin{center}
\begin{tikzpicture}
\node[vertex] (vertexa) at (-3,0)   {$g$};	
\node[vertex] (vertexb) at (0,0)  {$g|_e$}
	edge [<-,>=latex,out=180,in=0,thick] node[auto,swap,pos=0.5]{$\scriptstyle (e,g\cdot e)$}(vertexa);
\end{tikzpicture}
\end{center}
The self-similarity relations for the set $\NN$ can be read off the Moore diagram: the edge above encodes the relation $g\cdot (e \mu)=(g\cdot e)(g|_e\cdot \mu)$ for $\mu \in s(e)E^*$. 

\begin{remark}
We do not assume that the contracting set $\NN$ is minimal, as is often done in the literature on self-similar groups. However, as in \cite[\S2]{lrrw}, we can decide whether such a set is minimal by looking for cycles in the Moore diagram of $G$: if $g$ lies on a cycle, then $g|_{\mu}\in \NN$ for all $\mu\in s(g)E^*$.
\end{remark}

Now suppose that $\psi$ is the KMS$_{\ln\rho(B)}$ state of $(\OO(G,E),\sigma)$ from Theorem \ref{KMSatcritical}\eqref{existcritKMS}. The definition of the sets $G^n_g(v)$ in \eqref{expandphiu} shows that
\begin{align}\label{sum c_g formula}
\phi(u_g)&=\sum_{v\in E^0}\sum_{\mu\in G_g^n(v)}\rho(B)^{-n}\phi(u_{g|_\mu})\\
&=\sum_{\{\mu\in d(g)E^n\,:\,g\cdot \mu=\mu\}}\rho(B)^{-n}\phi(u_{g|_{\mu}}).\notag
\end{align}
Since for fixed $g$ and large enough $n$, every $g|_{\mu}\in \NN$, the values $\{\psi(u_g):g\in G\}$ are determined by the values $\{\psi(u_g):g\in \NN\}$. 

 For $g \in \NN \setminus E^0$, the set $F_g^k(v)$ consists of \emph{stationary paths} $\mu \in d(g)E^k v$ in the Moore diagram of the form
\begin{center}
\begin{tikzpicture}
\node[vertex] (vertexe) at (-5,0)   {$g$};
\node[vertex] (vertexa) at (-2.5,0)   {$g|_{\mu_1}$}	
	edge [<-,>=latex,out=180,in=0,thick] node[auto,swap,pos=0.5]{$\scriptstyle(\mu_1,\mu_1)$} (vertexe);
\node[vertex] (vertexb) at (0,0)  {$g|_{\mu_1\mu_2}$}
	edge [<-,>=latex,out=180,in=0,thick] node[auto,swap,pos=0.5]{$\scriptstyle(\mu_2,\mu_2)$}(vertexa);
\node[vertex] (vertexc) at (2.5,0)   {$\cdots$}
	edge [<-,>=latex,out=180,in=0,thick] node[auto,swap,pos=0.5]{$\scriptstyle(\mu_3,\mu_3)$} (vertexb);
\node[vertex] (vertexd) at (5,0)  {$g|_{\mu}=\id_v$.}
	edge [<-,>=latex,out=180,in=0,thick] node[auto,swap,pos=0.5]{$\scriptstyle(\mu_k,\mu_k)$}(vertexc);
\end{tikzpicture}
\end{center}
So we can compute the values of $|F_g^k(v)|$ by counting the stationary paths of length $k$ from $g$ to $\id_v$. Once we have the values $|F_g^k(v)|$, we compute $\psi(u_g)=c_g$ by evaluating the limit $\lim_{k\to \infty} c_{g,k}$ of Proposition \ref{Fng}.

We return to the graph $E$ of Example~\ref{s,r clarifying example}, and the  faithful self-similar groupoid action $(G,E)$ obtained by applying Theorem~\ref{defnGA} to the partial automorphisms $f_a$ and $f_b$ descibed in that example. Looking at the drawing of $E$ in Figure~\ref{asymmetric E} shows that the graph $E$ is strongly connected.

\begin{prop}\label{prop:nonsymm2}
The self-similar groupoid action $(G,E)$ of Example~\ref{s,r clarifying example} contracts to 
\begin{equation}\label{defNoldex}
\NN:=\{\id_v,\id_w,f_a,f_{a^{-1}},f_b,f_{b^{-1}}\}.
\end{equation}
\end{prop}

\begin{proof}
All reduced elements of length two in $G$ belong to either
\[
R_v:=\{f_b f_a,f_{a^{-1}}f_{b^{-1}}\} \quad \text{or} \quad R_w:=\{f_a f_b,f_{a^{-1}}f_{b^{-1}}\}.
\]
We compute all of their restrictions by length one paths:
\begin{equation}\label{ordertwo_nonsymmetric_a}
\begin{tabular}{lclclcl}
$(f_bf_a)|_{1}=f_a$ & \qquad & $(f_{a^{-1}}f_{b^{-1}})|_{1}=f_{b^{-1}}$ & \qquad & $(f_af_b)|_{3}=\id_v$ & \qquad & $(f_{b^{-1}}f_{a^{-1}})|_{3}=f_{a^{-1}}f_{b^{-1}}$ \\
$(f_bf_a)|_{2}=f_b$ & \qquad & $(f_{a^{-1}}f_{b^{-1}})|_{2}=f_{a^{-1}}$ & \qquad & $(f_af_b)|_{4}=f_bf_a$ & \qquad & $(f_{b^{-1}}f_{a^{-1}})|_{4}=\id_v$ \\
\end{tabular}
\end{equation}
The restriction of each element of $R_v$ belongs to $\{f_a,f_b,f_{a^{-1}},f_{b^{-1}}\}$ and the restriction of each element of $R_w$ belongs to $R_v \cup E^0$. Thus for any $g \in R_v \cup R_w$ and $\mu \in d(g)E^2$, we have $g|_{\mu} \in \NN$. So if $g \in G$ is the product of $n$ generators, then $g|_{\nu}$ belongs to $\NN$ for all $\nu \in d(g)E^{2(n-1)}$. Thus  $(G,E)$ contracts to $\NN$. 
\end{proof}

We draw the Moore diagram for the contracting set $\NN$ of $(G,E)$ in Figure~\ref{Moore:nonsymmetric2}.

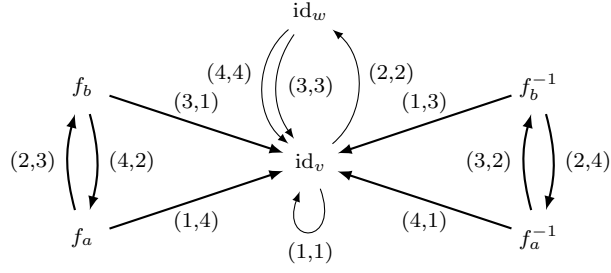
\begin{figure}
\begin{center}
\begin{tikzpicture}
\node at (0,2) {$\scriptstyle \id_w$};
\node[vertex] (vertexew) at (0,2)   {$\quad$};
\node at (0,0) {$\scriptstyle \id_v$};
\node[vertex] (vertexev) at (0,0)   {$\quad$}
	edge [<-,>=latex,out=35+90,in=145+90] node[right,swap,pos=0.5]{$\scriptstyle (3,3)$} (vertexew)
	edge [<-,>=latex,out=50+90,in=130+90] node[left,swap,pos=0.6]{$\scriptstyle (4,4)$} (vertexew)
	edge [<-,>=latex,out=200+90,in=160+90,loop] node[below,pos=0.5]{$\scriptstyle (1,1)$} (vertexev)
	edge [->,>=latex,out=310+90,in=230+90] node[right,swap,pos=0.6]{$\scriptstyle (2,2)$} (vertexew);
\node at (-3,-1) {$\scriptstyle f_a$};
\node[vertex] (vertexa) at (-3,-1)  {$\quad$}
	edge [->,>=latex,out=288.4+90,in=108.4+90,thick] node[below,pos=0.5]{$\scriptstyle(1,4)$} (vertexev);
\node at (-3,1) {$\scriptstyle f_b$};
\node[vertex] (vertexb) at (-3,1)  {$\quad$}
	edge [<-,>=latex,out=165+90,in=15+90,thick] node[left,pos=0.5]{$\scriptstyle (2,3)$} (vertexa)
	edge [->,>=latex,out=195+90,in=-15+90,thick] node[right,pos=0.5]{$\scriptstyle (4,2)$} (vertexa)
	edge [->,>=latex,out=251.6+90,in=71.6+90,thick] node[above,pos=0.5]{$\scriptstyle (3,1)$} (vertexev);
\node at (3,-1) {$\scriptstyle f_{a}^{-1}$};
\node[vertex] (vertex-a) at (3,-1)  {$\quad$}
	edge [->,>=latex,out=71.6+90,in=251.6+90,thick] node[below,pos=0.5]{$\scriptstyle(4,1)$} (vertexev);
\node at (3,1) {$\scriptstyle f_{b}^{-1}$};
\node[vertex] (vertex-b) at (3,1)  {$\quad$}
	edge [<-,>=latex,out=165+90,in=15+90,thick] node[left,pos=0.5]{$\scriptstyle (3,2)$} (vertex-a)
	edge [->,>=latex,out=195+90,in=-15+90,thick] node[right,pos=0.5]{$\scriptstyle (2,4)$} (vertex-a)
	edge [->,>=latex,in=288.4+90,out=108.4+90,thick] node[above,pos=0.5]{$\scriptstyle (1,3)$} (vertexev);
\end{tikzpicture}
\end{center}
\caption{The Moore diagram for the contracting set $\NN$ of \eqref{defNoldex}.}
\label{Moore:nonsymmetric2}
\end{figure}

\begin{prop}
The Cuntz-Pimsner algebra $(\OO(G,E),\sigma)$ has a unique KMS$_{\ln 2}$ state $\psi$ which is given on $\NN$ by
\[
\psi(u_g)=\begin{cases}
0 & \text{ for } g \in \{f_a,f_b,f_{a^{-1}},f_{b^{-1}}\}, \\
1/2 & \text{ for } g \in \{\id_v,\id_w\}. \\
\end{cases}
\]
\end{prop}

\begin{proof}
The spectral radius of the vertex matrix $B$ of $E$ is $\rho(B)=2$ with normalised Perron-Frobenius eigenvector $(1/2,1/2)$. Theorem \ref{KMSatcritical}\eqref{uniqueKMS} implies that there is a unique KMS$_{\ln 2}$ state $\psi$ and Proposition \ref{prop:factors} implies that $\psi(u_g)=1/2$ for $g \in \{\id_v,\id_w\}$.

The Moore diagram in Figure \ref{Moore:nonsymmetric2} shows that there are no stationary paths from elements of the set $S=\{f_a,f_b,f_{a^{-1}},f_{b^{-1}}\}$ to $\id_v$ or $\id_w$. Thus $F_g^k(v)=F_g^k(w)=\emptyset$ for all $k \in \N$ and $g \in S$, and $\psi(u_g)=0$ for $g \in S$.
\end{proof}

We next examine  one of the self-similar groupoid actions $(G_{\AA},E)$ of Example~\ref{Katsura example}, for $\AA=\{a_1,a_2\}$ and the matrices 
\begin{equation}\label{Katex}
A = \begin{pmatrix}
2 & 1 \\
2 & 2\end{pmatrix}
\quad  \text{ and } \quad 
B = \begin{pmatrix}
1 & 0 \\
2 & 1 \end{pmatrix}.
\end{equation}
We draw the corresponding graph $E$ in Figure~\ref{fig:graph Katsura 1}, and observe that $E$ is strongly connected. 
\begin{figure}
\begin{tikzpicture}
\node at (0,0) {$2$};
\node[vertex] (vertex2) at (0,0)   {$\quad$}
	edge [<-,>=latex,out=20,in=340,loop,thick,looseness=20] node[right,pos=0.5]{\,$e_{2,2,0}$} (vertex2)
	edge [<-,>=latex,out=40,in=320,loop,thick,looseness=15] node[left,pos=0.5]{$e_{2,2,1}$\,} (vertex2);
\node at (-3,0) {$1$};
\node[vertex] (vertex1) at (-3,0)   {$\quad$}
	edge [<-,>=latex,out=160,in=200,loop,thick,looseness=20] node[right,pos=0.5]{$e_{1,1,0}$} (vertex1)
	edge [<-,>=latex,out=140,in=220,loop,thick,looseness=15] node[left,pos=0.5]{$e_{1,1,1}$} (vertex1)
	edge [<-,>=latex,out=30,in=150,thick] node[above,pos=0.5]{$e_{1,2,0}$} (vertex2)
	edge [->,>=latex,out=335,in=205,thick] node[above,pos=0.5]{$e_{2,1,0}$} (vertex2)
	edge [->,>=latex,out=320,in=220,thick] node[below,pos=0.5]{$e_{2,1,1}$} (vertex2);
\end{tikzpicture}
\caption{The graph $E$ defined by the matrix $A$ in \eqref{Katex}.}
\label{fig:graph Katsura 1}
\end{figure}
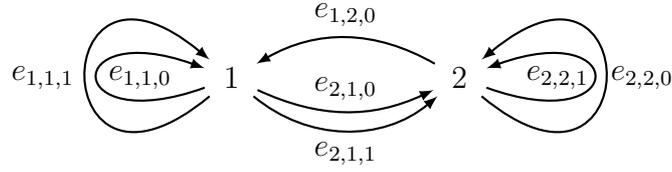
The action of the generating set $\{f_{a_1},f_{a_2}\}$ is given by the relations
\begin{equation*}
\begin{tabular}{lclcl}
$f_{a_1} \cdot e_{1,1,0}\phantom{.} \mu =e_{1,1,1} \phantom{.}\mu $ & \qquad & $f_{a_1} \cdot e_{1,1,1} \phantom{.}\mu =e_{1,1,0} (f_{a_1}\cdot \mu) $ & \qquad & $f_{a_1} \cdot e_{1,2,0}\phantom{.} \nu =e_{1,2,0}\phantom{.} \nu $\\
$f_{a_2} \cdot e_{2,1,0}\phantom{.} \mu =e_{2,1,0}(f_{a_1} \cdot \mu) $ & \qquad & $f_{a_2} \cdot e_{2,1,1}\phantom{.} \mu =e_{2,1,1}(f_{a_1} \cdot \mu)$ & \qquad & $f_{a_2} \cdot e_{2,2,0}\phantom{.} \nu =e_{2,2,1}\phantom{.}\nu$ \\
$f_{a_2} \cdot e_{2,2,1}\phantom{.} \nu =e_{2,2,0}(f_{a_2} \cdot \nu) $ &&&&\\
\end{tabular}
\end{equation*}

\begin{prop}\label{prop:KatsuraEx}
We take $(\AA,E)$ as above. Then the self-similar groupoid action $(G_{\AA},E)$ contracts to 
\begin{equation}\label{nucKat}
\NN:=\{\id_1,\id_2,f_{a_1},f_{a_1}^{-1},f_{a_2},f_{a_2}^{-1}\},
\end{equation}
and for each connected component $C_i=\{i\}$ the reduction $G_{\{i\}}$ is isomorphic to~$\Z$.
\end{prop}
\begin{proof}
We first show that $(G_{\AA},E)$ contracts to $\NN$ by showing that for each $g \in G$ and sufficiently long $\mu \in d(g)E^*$ we have $g|_{\mu} \in \NN$.

The reduced elements of length two in $G$ are $S:=\{f_{a_1}^2$, $(f_{a_1}^2)^{-1}$, $f_{a_2}^2$, $(f_{a_2}^2)^{-1}$\}.
We compute
\begin{equation}\label{ordertwo_nonsymmetric}
\begin{tabular}{lclclcl}
$f_{a_1}^2|_{e_{1,1,0}}=f_{a_1}$ & \qquad & $f_{a_1}^2|_{e_{1,1,1}}=f_{a_1}$ & \qquad & $f_{a_1}^2|_{e_{1,2,0}}=\id_2$ &&\\
$f_{a_2}^2|_{e_{2,1,0}}=f_{a_1}^2$ & \qquad & $f_{a_2}^2|_{e_{2,1,1}}=f_{a_1}^2$ & \qquad & $f_{a_2}^2|_{e_{2,2,0}}=f_{a_2}$ & \qquad & $f_{a_2}^2|_{e_{2,2,1}}=f_{a_2}$ \\
\end{tabular}
\end{equation}
We see that the length of $f_{a_1}^2|_e$ is less than two for all $e \in d(f_{a_1})$ and $f_{a_1}^2|_{e'}$ is either $f_{a_1}^2$ or has length one for all $e' \in d(f_{a_2})$. We conclude that $f_{a_i}^2|_{\mu}$ has length less than two for all $\mu \in d(f_{a_i})$ with $|\mu|=2$. Since $(f_{a_i}^2)^{-1}|_{\mu}=f_{a_i}^2|_{(f_{a_i}^2)^{-1}\cdot \mu}$ by Proposition \ref{SSA:extension to paths}\eqref{path_rest_inv} and partial automorphisms are length preserving, we also conclude that $(f_{a_i}^2)^{-1}|_{\mu}$ has length less than two for all $\mu \in d(f_{a_i}^{-1})$ with $|\mu|=2$. So for any $g\in S$ and $\mu \in d(g)E^2$, we have that~$g|_{\mu} \in \NN$. Thus if $g \in G_{\AA}$ is the product of $n$ generators, then $g|_{\nu}$ belongs to $\NN$ for all $\nu \in d(g)E^{2(n-1)}$. Thus $(G_{\AA},E)$ contracts to $\NN$, as claimed.

We now show that $G_{\{i\}}\cong \Z$. Fix $i \in \{1,2\}$, and recall from  Example~\ref{Katsura example} that either $G_{\{i\}}=\{i, f_{a_i},f_{a_i}^2,\cdots,f_{a_i}^{n_i-1}\}$ or $G_{\{i\}}=\{f_{a_i}^k:k\in \Z\}$.
To show the latter, we suppose $k \in \N$ and prove that $f_{a_i}^k \neq \id_i$. Since $f_{a_i}\cdot e_{i,i,0}=e_{i,i,1}$ and $f_{a_i}\cdot e_{i,i,1}=e_{i,i,0}$, we have that $f_{a_i}^k\neq \id_v$ for $k$ odd. So suppose $k=2n$ is even. As shown in \eqref{ordertwo_nonsymmetric} we have $f_{a_i}^2|_{e_{i,i,0}}=f_{a_i}$, and hence
\[
f_{a_i}^4|_{e_{i,i,0}}=f_{a_i}^2|_{f_{a_i}^2\cdot e_{i,i,0}}f_{a_i}^2|_{e_{i,i,0}}=f_{a_i}^2|_{e_{i,i,0}}f_{a_i}=f_{a_i}^2.
\]
Continuing in this way, we find that $f_{a_i}^{2n}|_{e_{i,i,0}}=f_{a_i}^{n}$. Considering the action of $f_{a_i}^{2n}$ on sufficiently long words of the form $\mu=e_{i,i,0}e_{i,i,0}\dots e_{i,i,0}$, we can reduce $2n$ by factors of $2$ and arrive at $f_{a_i}^{2n}|_{\mu}=f_{a_i}^{m}$ with $m$ odd. Then $f_{a_i}^{2n}\cdot \mu \phantom{.}e_{i,i,0}=\mu (f_{a_i}^{m}\cdot e_{i,i,0})=\mu \phantom{.} e_{i,i,1}$. Thus $f_{a_i}^k$ acts nontrivially on $iE^*$ for all $k \in \Z$, and $G_{\{i\}}\cong \Z$.
\end{proof}

We have drawn  the Moore diagram for $\NN$ in Figure~\ref{Moore:KMS interesting example}. Since every element in $\NN$ is in a cycle, this contracting set $\NN$ is in fact  minimal. Note that the central horizontal layer in Figure~\ref{Moore:KMS interesting example} is a copy of the graph $E$ of Figure~\ref{fig:graph Katsura 1} with the arrows reversed. 

\begin{figure}
\begin{center}
\begin{tikzpicture}
\node at (0,0) {$\scriptstyle \id_2$};
\node[vertex] (vertex2) at (0,0)   {$\quad$}
	edge [->,>=latex,out=20,in=340,loop,thick,looseness=20] node[right,pos=0.6]{\,\,$\scriptstyle (e_{2,2,0},e_{2,2,0})$} (vertex2)
	edge [->,>=latex,out=40,in=320,loop,thick,looseness=15] node[right,pos=0.4]{$\scriptstyle (e_{2,2,1},e_{2,2,1})$} (vertex2);
\node at (-3,0) {$\scriptstyle \id_1$};
\node[vertex] (vertex1) at (-3,0)   {$\quad$}
	edge [->,>=latex,out=160,in=200,loop,thick,looseness=20] node[left,pos=0.4]{$\scriptstyle (e_{1,1,0},e_{1,1,0})$\,\,} (vertex1)
	edge [->,>=latex,out=140,in=220,loop,thick,looseness=15] node[left,pos=0.6]{$\scriptstyle (e_{1,1,1},e_{1,1,1})$} (vertex1)
	edge [->,>=latex,out=16,in=164,thick] node[above,pos=0.4]{$\scriptstyle (e_{1,2,0},e_{1,2,0})$} (vertex2)
	edge [<-,>=latex,out=342,in=198,thick] node[above,pos=0.5]{$\scriptstyle (e_{2,1,0},e_{2,1,0})$} (vertex2)
	edge [<-,>=latex,out=330,in=210,thick] node[below,pos=0.4]{$\scriptstyle (e_{2,1,1},e_{2,1,1})$} (vertex2);
\node at (0,3) {$\scriptstyle f_{a_2}$};
\node[vertex] (vertex2a) at (0,3)   {$\quad$}
	edge [->,>=latex,out=270,in=90,thick] node[right,pos=0.5]{\,$\scriptstyle (e_{2,2,0},e_{2,2,1})$} (vertex2)
	edge [->,>=latex,out=40,in=320,loop,thick,looseness=15] node[right,pos=0.5]{$\scriptstyle (e_{2,2,1},e_{2,2,0})$} (vertex2a);
\node at (-3,3) {$\scriptstyle f_{a_1}$};
\node[vertex] (vertex1a) at (-3,3)   {$\quad$}
	edge [->,>=latex,out=270,in=90,thick] node[left,pos=0.5]{$\scriptstyle (e_{1,1,0},e_{1,1,1})$} (vertex1)
	edge [->,>=latex,out=140,in=220,loop,thick,looseness=15] node[left,pos=0.5]{$\scriptstyle (e_{1,1,1},e_{1,1,0})$} (vertex1a)
	edge [->,>=latex,out=315,in=135,thick] node[right,pos=0.3]{$\scriptstyle (e_{1,2,0},e_{1,2,0})$} (vertex2)
	edge [<-,>=latex,out=10,in=170,thick] node[above,pos=0.5]{$\scriptstyle (e_{2,1,0},e_{2,1,0})$} (vertex2a)
	edge [<-,>=latex,out=350,in=190,thick] node[below,pos=0.5]{$\scriptstyle (e_{2,1,1},e_{2,1,1})$} (vertex2a);
\node at (0,-3) {$\scriptstyle f_{a_2}^{-1}$};
\node[vertex] (vertex-2a) at (0,-3)   {$\quad$}
	edge [->,>=latex,out=90,in=270,thick] node[right,pos=0.5]{\,$\scriptstyle (e_{2,2,1},e_{2,2,0})$} (vertex2)
	edge [->,>=latex,out=40,in=320,loop,thick,looseness=15] node[right,pos=0.5]{$\scriptstyle (e_{2,2,0},e_{2,2,1})$} (vertex-2a);
\node at (-3,-3) {$\scriptstyle f_{a_1}^{-1}$};
\node[vertex] (vertex-1a) at (-3,-3)   {$\quad$}
	edge [->,>=latex,out=90,in=270,thick] node[left,pos=0.5]{$\scriptstyle (e_{1,1,1},e_{1,1,0})$} (vertex1)
	edge [->,>=latex,out=140,in=220,loop,thick,looseness=15] node[left,pos=0.5]{$\scriptstyle (e_{1,1,0},e_{1,1,1})$} (vertex-1a)
	edge [->,>=latex,out=45,in=225,thick] node[right,pos=0.3]{$\scriptstyle (e_{1,2,0},e_{1,2,0})$} (vertex2)
	edge [<-,>=latex,out=10,in=170,thick] node[above,pos=0.5]{$\scriptstyle (e_{2,1,0},e_{2,1,0})$} (vertex-2a)
	edge [<-,>=latex,out=350,in=190,thick] node[below,pos=0.5]{$\scriptstyle (e_{2,1,1},e_{2,1,1})$} (vertex-2a);
\end{tikzpicture}
\end{center}
\caption{The Moore diagram for $\NN$ of Proposition~\ref{prop:KatsuraEx}.}
\label{Moore:KMS interesting example}
\end{figure}
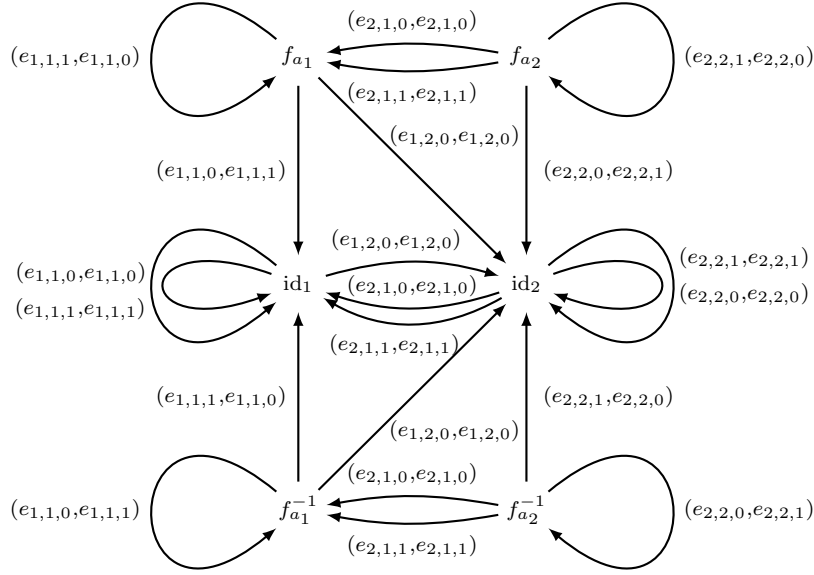

\begin{prop}\label{wildnos}
Keep the above notation, and consider the self-similar groupoid action $(G_{\AA},E)$ with contracting set $\NN$ as in \ref{nucKat}.
Then the Cuntz-Pimsner algebra $(\OO(G_{\AA},E),\sigma)$ has a unique KMS$_{\ln (2+\sqrt{2})}$ state $\psi$ which is given on $\NN$ by  
\[
\psi(u_g)=\begin{cases}
\sqrt{2}-1 & \text{ for } g=\id_1, \\
2-\sqrt{2} & \text{ for } g=\id_2, \\
3-2\sqrt{2} & \text{ for } g=f_{a_1} \text{ and } g=f_{a_1}^{-1}, \\
10-7\sqrt{2} & \text{ for } g=f_{a_2} \text{ and } g=f_{a_2}^{-1}. \\
\end{cases}
\]
\end{prop}

\begin{proof}
A computation shows that the spectral radius of the vertex matrix for $E$ is $\rho(A)=2+\sqrt{2}$ with normalised Perron-Frobenius eigenvector $x=(\sqrt{2}-1,2-\sqrt{2})$. So Theorem \ref{KMSatcritical}\eqref{uniqueKMS} implies that there is a unique KMS$_{\ln (2+\sqrt{2})}$ state $\psi$. Moreover, since $u_v=p_v$ by Proposition \ref{Toeplitz_repn}\eqref{TR_pi}, Proposition \ref{prop:factors} implies that $\psi(u_1)=\sqrt{2}-1$ and $\psi(u_2)=2-\sqrt{2}$.

Since only the stationary paths in the Moore diagram are used to compute nontrivial KMS states at the critical temperature, we consider stationary paths in the Moore diagram in Figure~\ref{Moore:KMS interesting example}, which all lie in the cut-down subdiagram shown in Figure~\ref{Moore:KMS interesting example cut down}.

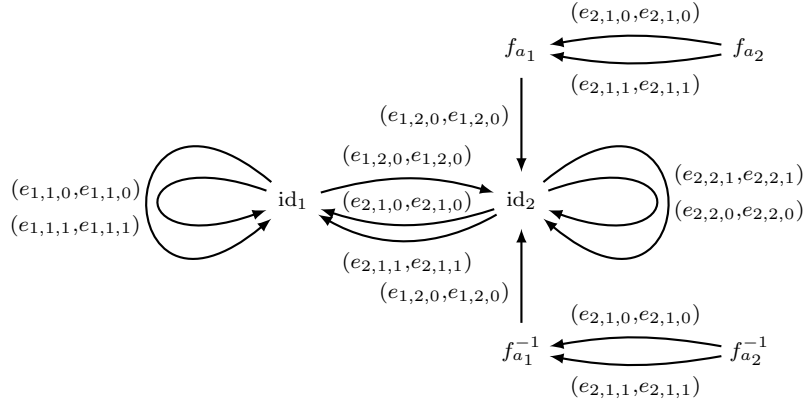
\begin{figure}
\begin{tikzpicture}
\node at (0,0) {$\scriptstyle \id_2$};
\node[vertex] (vertex2) at (0,0)   {$\quad$}
	edge [->,>=latex,out=20,in=340,loop,thick,looseness=20] node[right,pos=0.6]{\,\,$\scriptstyle (e_{2,2,0},e_{2,2,0})$} (vertex2)
	edge [->,>=latex,out=40,in=320,loop,thick,looseness=15] node[right,pos=0.4]{$\scriptstyle (e_{2,2,1},e_{2,2,1})$} (vertex2);
\node at (-3,0) {$\scriptstyle \id_1$};
\node[vertex] (vertex1) at (-3,0)   {$\quad$}
	edge [->,>=latex,out=160,in=200,loop,thick,looseness=20] node[left,pos=0.4]{$\scriptstyle (e_{1,1,0},e_{1,1,0})$\,\,} (vertex1)
	edge [->,>=latex,out=140,in=220,loop,thick,looseness=15] node[left,pos=0.6]{$\scriptstyle (e_{1,1,1},e_{1,1,1})$} (vertex1)
	edge [->,>=latex,out=16,in=164,thick] node[above,pos=0.5]{$\scriptstyle (e_{1,2,0},e_{1,2,0})$} (vertex2)
	edge [<-,>=latex,out=342,in=198,thick] node[above,pos=0.5]{$\scriptstyle (e_{2,1,0},e_{2,1,0})$} (vertex2)
	edge [<-,>=latex,out=330,in=210,thick] node[below,pos=0.5]{$\scriptstyle (e_{2,1,1},e_{2,1,1})$} (vertex2);
\node at (3,2) {$\scriptstyle f_{a_2}$};
\node[vertex] (vertex2a) at (3,2)   {$\quad$};
\node at (0,2) {$\scriptstyle f_{a_1}$};
\node[vertex] (vertex1a) at (0,2)   {$\quad$}
	edge [->,>=latex,out=270,in=90,thick] node[left,pos=0.4]{$\scriptstyle (e_{1,2,0},e_{1,2,0})$} (vertex2)
	edge [<-,>=latex,out=10,in=170,thick] node[above,pos=0.5]{$\scriptstyle (e_{2,1,0},e_{2,1,0})$} (vertex2a)
	edge [<-,>=latex,out=350,in=190,thick] node[below,pos=0.5]{$\scriptstyle (e_{2,1,1},e_{2,1,1})$} (vertex2a);
\node at (3,-2) {$\scriptstyle f_{a_2}^{-1}$};
\node[vertex] (vertex2a-) at (3,-2)   {$\quad$};
\node at (0,-2) {$\scriptstyle f_{a_1}^{-1}$};
\node[vertex] (vertex1a-) at (0,-2)   {$\quad$}
	edge [->,>=latex,out=90,in=270,thick] node[left,pos=0.3]{$\scriptstyle (e_{1,2,0},e_{1,2,0})$} (vertex2)
	edge [<-,>=latex,out=10,in=170,thick] node[above,pos=0.5]{$\scriptstyle (e_{2,1,0},e_{2,1,0})$} (vertex2a-)
	edge [<-,>=latex,out=350,in=190,thick] node[below,pos=0.5]{$\scriptstyle (e_{2,1,1},e_{2,1,1})$} (vertex2a-);
\end{tikzpicture}
\caption{The cut-down Moore diagram.}
\label{Moore:KMS interesting example cut down}
\end{figure}

We now compute $\psi(u_{g})$ for $g=f_{a_1}$. Observe that there is a length one stationary path from  $g$ to $\id_2$. Let $a_{ij}^k$ denote the entries of the matrix $A^k$. Then the sets $|F_{g}^k(i)|$ are given by $|F_{g}^k(1)|=a_{21}^{k-1}$ and $|F_{g}^k(2)|=a_{22}^{k-1}$. To compute $\psi(u_{g})=c_{g}$ we use Proposition \ref{Fng} to obtain
\[
c_{g,k}= \rho(B)^{-k} \sum_{v\in E^0} |F_g^k(v)|x_v = (2+\sqrt{2})^{-k}\big(a_{21}^{k-1}(\sqrt{2}-1)+a_{22}^{k-1}(2-\sqrt{2})\big).
\]
Then a bit of algebra gives that $\psi(u_{g})=c_{g}=\lim_{k \to \infty} c_{g,k}=3-2\sqrt{2}$. Since the computation for $\psi(u_{g^{-1}})$ is symmetric, we also have $\psi(u_{g^{-1}})=3-2\sqrt{2}$.

We now compute $\psi(u_{h})$ for $h=f_{a_2}$. We use \eqref{sum c_g formula} and the two stationary paths from $f_{a_2}$ to $f_{a_1}$ to compute
\[
\psi(u_{h})=2(2+\sqrt{2})^{-1}\psi(u_g)=10-7\sqrt{2}.
\]
By symmetry, we also have $\psi(u_{h^{-1}})=10-7\sqrt{2}$.
\end{proof}

\appendix

\section{Exel-Pardo Actions}\label{App:Exel-Pardo Actions}

In this appendix we show that if the self-similar groups on graphs of Exel and Pardo \cite{EP2} satisfy a certain faithfulness condition, then they give rise to faithful self-similar groupoid actions. We begin by recalling Exel and Pardo's definition of a self-similar group on a graph \cite[\S 2]{EP2}. A tuple $(K,E,\sigma,\phi)$ is a self-similar group on a graph if $K$ is a group and $E$ is a graph with no sources such that
\begin{enumerate}
\item There is a group action $\sigma:K \to \Aut E$ such that $\sigma^0_k:E^0 \to E^0$ and $\sigma^1_k:E^1 \to E^1$ are bijections, for $k \in K$, satisfying 
\begin{align}
\label{EP_range} 
r \circ \sigma_k^1&=\sigma_k^0 \circ r, \\
\label{EP_source}
s \circ \sigma_k^1&=\sigma_k^0 \circ s, and \\
\label{EP_action} 
\sigma_k \sigma_l&=\sigma_{kl}.
	\end{align}
\item There is a group $1$-cocycle $\phi:K \times E^1 \to K$ satisfying
\begin{equation}\label{EP_cocycle}
\phi(kl,e)=\phi(k,\sigma_l^1(e))\phi(l,e).
\end{equation}
\item For $k \in K$ and $e \in E^1$ there is a standing assumption that
\begin{equation}\label{EP_SA}
\sigma_{\phi(k,e)}^0=\sigma_k^0.
\end{equation}
\end{enumerate}

\begin{prop}[{\cite[Proposition 2.4]{EP2}}]\label{EP extension to paths}
Suppose $(K,E,\sigma,\phi)$ is an Exel-Pardo self-similar group on a graph. The pair $(\sigma,\phi)$ define a unique pair $(\sigma^*,\phi^*)$ that extend $(\sigma,\phi)$ such that $\sigma^*$ is an action of $K$ on $E^*$ and $\phi^*$ is a group one-cocycle $\phi:K \times E^* \to K$ such that, for every $n \geq 0$, every $k \in K$, and every $v \in E^0$,
\begin{enumerate}
\item\label{EP contained in n} $\sigma_k^*(E^n) \subset E^n$,
\item\label{EP range} $r \circ \sigma_k^*=\sigma_k \circ r$ on $E^n$,
\item\label{EP source} $s \circ \sigma_k^*=\sigma_k \circ s$ on $E^n$,
\item\label{EP rest fixed} $\sigma_{\phi^*(k,\mu)}(v)=\sigma_k(v)$ for all $\mu \in E^n$,
\item\label{EP SSCondition} $\sigma_k^*(\mu\nu)=\sigma_k^*(\mu)\sigma^*_{\phi^*(k,\mu)}(\nu)$ where $\mu\nu \in E^n$,
\item\label{EP rest rest} $\phi^*(k,\mu\nu)=\phi^*(\phi^*(k,\mu),\nu)$ where $\mu\nu \in E^n$.
\end{enumerate}
\end{prop}

If $E$ has no sources, Exel and Pardo  construct a unital $C^*$-algebra 
$\OO_{K,E}$ as the universal algebra generated by elements $\{u_k\colon k\in K\}$, mutually orthogonal projections $\{p_v \colon v\in E^0 \}$, 
and partial isometries $\{s_e \colon e\in E^1 \}$ 
subject to the conditions that
\begin{itemize}
\item $\{p_v \colon v\in E^0 \}\cup \{s_e \colon e\in E^1 \}$ is a Cuntz-Krieger $E$-family 
\item the map $u\colon K\to \OO_{(K,E)}\colon k\mapsto u_k$ is a unitary representation of $K$
\item $u_ks_e = s_{\sigma^1_k(e)}u_{\phi(k,e)}$, for every $k\in K$, and $e\in E^1$, and
\item $u_kp_v = p_{\sigma^0_k(v)}u_k$, for every $k\in K$, and $v\in E^0$
\end{itemize}

\begin{prop}\label{EP trans group}
Suppose $(K,E,\sigma,\phi)$ is an Exel-Pardo self-similar group on a graph and let $K \times E^0:=\{(k,v)\}$ be the transformation groupoid with $d(k,v)=v$, $c(k,v)=\sigma_k^0(v)$, and unit space $E^0$, and consider the maps
\begin{align}
\label{EP action map}
(k,v) \cdot \mu &:= \sigma_k^*(\mu) \quad \text{ for $\mu \in vE^*$ and} \\
\label{EP restriction map}
(k,v)|_{\mu} &:= (\phi^*(k,\mu),s(\mu)) \quad \text{ for $\mu \in vE^*$}.
\end{align}
Then \eqref{EP action map} defines a groupoid action of $K \times E^0$ on $T_E$ by partial isomorphisms. If the groupoid action is faithful, then $(K \times E^0,E)$ is a self-similar groupoid action.

Moreover, if $E$ has no sources then the algebra $\OO(K \times E^0,E)$ from~Proposition~\ref{Cuntz_repn} 
 is isomorphic to the algebra $\OO_{K,E}$ defined in~\cite{EP2}.
\end{prop}

\begin{proof}
We begin by showing that $(k,v):d(k,v)E^* \to c(k,v)E^*$ defines an action by partial isomorphism, as in Definition \ref{defn: partial iso}. Proposition \ref{EP extension to paths}\eqref{EP contained in n} shows that if $\mu \in d(k,v)E^n$, then $(k,v)\cdot \mu \in c(k,v)E^n$.
We now show that $(k,v):d(k,v)E^n \to c(k,v)E^n$ is a bijection by induction on $n$. The result holds for $n=1$ since $\sigma^1_k:E^1 \to E^1$ is a bijection and $\sigma_k^*$ extends $\sigma_k^1$. So assume $(k,v):d(k,v)E^n \to c(k,v)E^n$ is a bijection for all $(k,v) \in K \times E^0$.
For injectivity suppose that 
$(k,v)\cdot(e\mu)= (k,v)\cdot(f\nu)$ for $e,f \in d(k,v)E^1$, $\mu \in s(e)E^n$ and $\nu \in s(f)E^n$. 
Using Definition \ref{EP extension to paths}\eqref{EP SSCondition} we compute
\begin{align*}
(k,v)\cdot(e\mu)= (k,v)\cdot(f\nu)&\implies \sigma_k^*(e\mu)=\sigma_k^*(f\nu) \\
& \implies \sigma_k^*(e)\sigma^*_{\phi^*(k,e)}(\mu)=\sigma_k^*(f)\sigma^*_{\phi^*(k,f)}(\nu) \\
& \implies \sigma_k^*(e)= \sigma_k^*(f) \text{ and } \sigma^*_{\phi^*(k,e)}(\mu) = \sigma^*_{\phi^*(k,f)}(\nu) \\
& \implies (k,v)\cdot e= (k,v)\cdot f \text{ and } (k,v)|_{e} \cdot \mu = (k,v)|_{f} \cdot \nu \\
& \implies e= f \text{ and } \mu = \nu,
\end{align*}
where the final implication holds by the inductive hypothesis. Hence $e\mu=f\nu$, thereby proving injectivity. 

For surjectivity, suppose $f \in c(k,v)E^{1}$ and $\nu \in s(f)E^n$. By the induction hypothesis there exists a unique 
$e \in d(k,v)E^{1}$ satisfying 
$(k,v) \cdot e=f$ and a unique $\mu \in d((k,v)|_e)E^{n}$ 
satisfying $(k,v)|_e\cdot \mu=\nu$. To conclude the proof of surjectivity we need to show that $s(e)=r(\mu)$ so that $e\mu$ is a path in $E^{n+1}$. For this, we compute
\begin{align*}
(k,v)|_e \cdot s(e) &= (\phi^*(k,e),s(e))\cdot s(e) \quad \text{ by } \eqref{EP restriction map} \\
&=\sigma^*_{\phi^*(k,e)}(s(e)) \quad \text{ by } \eqref{EP action map} \\
&= \sigma_k^*(s(e)) \quad \text{ by Proposition \ref{EP extension to paths}\eqref{EP rest fixed}} \\
&= s(\sigma_k^*(e)) \quad \text{ by Proposition \ref{EP extension to paths}\eqref{EP source}} \\
&= s(f) =r(\nu)=r(\sigma^*_k(\mu)) \quad \text{ by hypothesis}\\
&= \sigma_k^*(r(\mu)) \quad \text{ by Proposition \ref{EP extension to paths}\eqref{EP range}} \\
&=\sigma^*_{\phi^*(k,e)}(r(\mu)) \quad \text{ by Proposition \ref{EP extension to paths}\eqref{EP rest fixed}} \\
&= (\phi^*(k,e),s(e))\cdot r(\mu) \quad \text{ by } \eqref{EP action map} \\
&= (k,v)|_e \cdot r(\mu) \quad \text{ by } \eqref{EP restriction map}.
\end{align*}
Applying $(k,v)|_e^{-1}$ to both sides of $(k,v)|_e \cdot s(e)=(k,v)|_e \cdot r(\mu)$ implies that $s(e)=r(\mu)$ as desired. Thus $(k,v):s(k,v)E^* \to r(k,v)E^*$ is a bijection satisfying the first statement in \eqref{defnpi}. The second statement in \eqref{defnpi} follows straight from Proposition \ref{EP extension to paths}\eqref{EP SSCondition}. Thus $(k,v):d(k,v)E^* \to c(k,v)E^*$ is a partial isomorphism. It follows from equation~(\ref{EP_cocycle}) that $((k,v),\mu)\to(k,v)\cdot\mu$ is an action of the groupoid~$K\times E$.

Now suppose the action by partial isomorphism is faithful. To show $(K \times E^0,E)$ is a self-similar groupoid action we prove that \eqref{selfsimilar groupoid defn} of Definition \ref{faithful self-similar groupoid action} holds. Suppose $(k,v) \in K \times E^0$ and $e \in d(k,v)E^1$. Then for $h:=(k,v)|_{e}$ and $\mu \in s(e)E^*$ we have
\begin{align*}
(k,v)\cdot e\mu &=\sigma_k^*(e\mu) \quad \text{ by \eqref{EP action map}} \\
&=\sigma_k^*(e)\sigma_{\phi^*(k,e)}^*(\mu) \quad \text{ by Proposition \ref{EP extension to paths}\eqref{EP SSCondition}} \\
&= \big((k,v)\cdot e\big)\big((\phi^*(k,e),s(e))\cdot \mu\big)  \quad \text{ by \eqref{EP action map}} \\
&= \big((k,v)\cdot e\big)\big((k,v)|_e \cdot \mu\big)  \quad \text{ by \eqref{EP restriction map}} \\
&= \big((k,v)\cdot e\big)\big(h \cdot \mu\big),
\end{align*}
as desired.

For the final assertion note that there is a direct correspondence between the projections and partial isometries in the algebras. We claim that we can embed 
$\OO(K \times E^0,E)$ in $\OO_{K,E}$ by defining $u_{(k,v)}:=u_kp_v$  
and $\OO_{K,E}$ in $\OO(K \times E^0,E)$ by defining $u_k:=\sum_{v\in E^0} u_{(k,v)}$. We leave the details as an exercise for the reader.
\end{proof}

\begin{remark}\label{sources}
Suppose that $E$ is a directed graph and $G$ is a groupoid with unit space $E^0$ acting self-similarly on $T_E$. By Lemma~\ref{source and range observations}\eqref{equivariance}, for $g\in G$ we have
\begin{equation}\label{req1}
s(g\cdot e)=g|_e\cdot s(e)\quad\text{for all $e\in d(g)E^1$.}
\end{equation}
Both Exel--Pardo \cite{EP2} and B\'edos--Kaliszewski--Quigg \cite{BKQ} consider self-similar automorphisms of $T_E$ built from graph automorphisms of $E$. Such automorphisms satisfy
\begin{equation}\label{req2}
s(g\cdot e)=g\cdot s(e)\quad\text{for all $e\in E^1$.}
\end{equation}
Combining the requirements \eqref{req1} and \eqref{req2} leads to the standing assumptions
\begin{align*}
g|_e\cdot s(e)&=g\cdot s(e)\quad\text{for all $e\in E^1$, imposed in \cite{BKQ}, and}\\
g|_\mu\cdot v&=g\cdot v\quad\text{for all $\mu\in E^*$ and $v\in E^1$, imposed in \cite{EP2}.}
\end{align*} 
\end{remark}

\begin{remark}
There are self-similar groupoid actions $(G,E)$ which do not come from self-similar group actions as in \cite{EP2} or \cite{BKQ}. To see this, we note that every every action arising via Proposition~\ref{EP extension to paths} has the properties that $r(g\cdot \mu)=g\cdot r(\mu)$ and $s(g\cdot \mu)=g\cdot s(\mu)$ (see Remark~\ref{sources}). However, for the self-similar groupoid action $(G_A,E)$ of Example~\ref{s,r clarifying example}, we have
\[
f_a(1)=4\quad\text{and}\quad s(f_a(1))=s(4)=v,\quad\text{but}\quad f_a(s(1))=f_a(v)=w.
\]
\end{remark}

\end{document}